\documentclass[12pt,reqno]{amsart}

\usepackage{amsmath,amsthm,amssymb,amsfonts,amscd}
\usepackage{mathrsfs}
\usepackage{graphicx}
\usepackage{booktabs}
\usepackage{float}
\usepackage{needspace}
\usepackage[section]{placeins}
\usepackage{geometry}
\usepackage{xcolor}
\usepackage{listings}
\usepackage[backref=page,hidelinks]{hyperref}

\lstdefinestyle{paperpython}{
  language=Python,
  basicstyle=\ttfamily\tiny,
  keywordstyle=\color{blue!55!black},
  commentstyle=\color{green!35!black},
  stringstyle=\color{red!50!black},
  numbers=left,
  numberstyle=\ttfamily\tiny\color{black!50},
  stepnumber=10,
  numbersep=7pt,
  xleftmargin=2.2em,
  framexleftmargin=1.8em,
  frame=single,
  framerule=0.2pt,
  breaklines=true,
  breakatwhitespace=false,
  columns=fullflexible,
  keepspaces=true,
  showstringspaces=false,
  tabsize=4,
  captionpos=t
}

\hypersetup{
  pdftitle={Universal murmuration and Hecke augmentation},
  pdfauthor={Shenghao Hua and Chung Pang Mok}
}

\numberwithin{equation}{section}

\theoremstyle{plain}
\newtheorem{theorem}{Theorem}[section]
\newtheorem{lemma}[theorem]{Lemma}
\newtheorem{corollary}[theorem]{Corollary}
\newtheorem{proposition}[theorem]{Proposition}

\theoremstyle{definition}

\newtheorem{conjecture}[theorem]{Conjecture}

\theoremstyle{remark}
\newtheorem{remark}[theorem]{Remark}

\newcommand{\rank}{\operatorname{rank}}

\newcommand{\sym}{\operatorname{sym}}

\newcommand{\N}{\operatorname{N}}
\newcommand{\GL}{\operatorname{GL}}

\begin{document}

\title[Universal Murmuration]{Universal murmuration and Hecke augmentation}
\author{Shenghao Hua}
\address[1]{Shanghai Institute for Mathematics and Interdisciplinary Sciences (SIMIS), Shanghai, 200433, China}
\address[2]{Research Institute of Intelligent Complex Systems, Fudan University, Shanghai, 200433, China}
\email{huashenghao@vip.qq.com}

\author{Chung Pang Mok}
\address[1]{Shanghai Institute for Mathematics and Interdisciplinary Sciences (SIMIS), Shanghai, 200433, China}
\email{cpmok@simis.cn}

\begin{abstract}
Prime coefficients of elliptic curves exhibit murmurations, statistical patterns that support the prediction of arithmetic labels.
We conjecture that root number weighted averages of unitary normalized coefficients at primes and prime powers sample the same leading profile when placed at the effective position $p^k/X$,
where $X$ is the conductor scale.
For weight $2$ newforms of squarefree level in the level aspect, we prove this principle for every fixed $k$ under suitable short-window and growth conditions, extending Zubrilina's prime case and the square-case analysis of Kundu and M\"uller.
Experiments with elliptic curve isogeny class representatives show that the resulting prime power features improve root number prediction and give a smaller gain in distinguishing ranks $0$ and $1$.
\end{abstract}


\maketitle

\section{Introduction}

Murmurations are small coherent biases in Fourier coefficients when arithmetic objects are ordered by a global invariant and separated by an arithmetic label.
He, Lee, Oliver, and Pozdnyakov~\cite{HeLeeOliverPozdnyakov2025} discovered this remarkable oscillatory pattern in average Frobenius traces for elliptic curves grouped by rank and ordered by conductor within bounded intervals.
The pattern was first detected through machine learning and computational experiments.
Sutherland and the same authors later observed related biases in broader families of arithmetic $L$-functions
~\cite{Sutherland2022,HeLeeOliverPozdnyakov2025}.
Their role in arithmetic prediction was developed further
in~\cite{HeLeeOliverPozdnyakov2026,Pozdnyakov2024}.
Sarnak observed~\cite{Sarnak2023} that murmuration is closely related to low-lying zeros~\cite{IwaniecLuoSarnak2000}.

The unconditioned local distribution of these coefficients is also well understood in the level aspect.
For each fixed prime $p$, Sarnak and Zubrilina~\cite{SarnakZubrilina2024} proved quantitative
convergence of the Hecke eigenvalues of weight $2$ forms to the $p$-adic Plancherel measure as the level tends to infinity.
Murmuration describes a finer phenomenon.
Zubrilina gave the first unconditional proof of murmuration by establishing a level-aspect asymptotic formula for the average of prime coefficients weighted by the root number~\cite{Zubrilina2025}.
Kundu and M\"uller treated the square-coefficient case and recorded the prime-power trace formula in~\cite{KunduMuller2026}.
Related results now appear for several families of automorphic objects in~\cite{BBLLD2026,BLLDSHZ2024,Cowan2024,LeeOliverPozdnyakov2025,SawinSutherland2025}.
The root number is the sign in the functional equation of the $L$-function.
For an elliptic curve, the functional equation gives
$\varepsilon(E)=(-1)^{r_{\mathrm{an}}(E)}$; the parity conjecture predicts
$\varepsilon(E)=(-1)^{\rank E(\mathbb Q)}$.
This relation explains the close connection between root-number and rank prediction.

For a finite family or multiset $\mathcal F$, we use the uniform average unless another weight is stated.
In the weight $2$ level aspect, $\mathbb E_{X,Y}$ is the uniform average over pairs $(N,f)$, where $N$ is squarefree in $[X,X+Y]$ and $f$ is a normalized newform of level $N$ with trivial nebentypus.
Every normalized newform is counted once.

The usual picture places the coefficient at a prime $p$ at the scaled position $p/X$, where $X$ is the conductor scale of the family.
We ask whether the coefficient at a prime power $p^k$ belongs to the same picture when placed at $p^k/X$.
This question leads to the following conjectural framework for each specified full family.

\begin{conjecture}[Universal prime and prime-power murmuration]
\label{conjUniversal}
Let $\mathcal F$ be a self-dual family of primitive $\GL_2$ forms, with root numbers in $\{-1,1\}$, ordered by analytic conductor $C(f)$, and write
\begin{equation*}
\mathcal F_{X,Y}=\{f\in\mathcal F:X\leq C(f)\leq X+Y\}.
\end{equation*}
Let $Y=Y(X)>0$ be a family-specific regime in which $|\mathcal F_{X,Y}|\to\infty$.
Assume that $\mathcal F$ is specified by fixed local and archimedean conditions, with no selection based on Fourier coefficients or root signs.
Let $\varepsilon(f)$ denote the root number.  At an unramified prime $p$, write the local factor in unitary normalization as
\begin{equation*}
L_p(s,f)=\bigl(1-\alpha_{f,p}p^{-s}\bigr)^{-1}
\bigl(1-\beta_{f,p}p^{-s}\bigr)^{-1}
\end{equation*}
and set
\begin{equation*}
\lambda_{f,k}(p)=\sum_{j=0}^{k}
\alpha_{f,p}^{k-j}\beta_{f,p}^{j}.
\end{equation*}
At a ramified prime, let $\lambda_{f,k}(p)$ be the coefficient of $p^{-s}$ in the unitary local symmetric-power Euler factor.
Assume that a continuous function $G_{\mathcal F}:(0,\infty)\to\mathbb R$ gives the prime-coefficient profile: for every compact set $K\Subset(0,\infty)$,
\begin{equation*}
\sqrt X\,
\mathbb E_{\mathcal F_{X,Y}}
\bigl[\varepsilon(f)\lambda_{f,1}(p)\bigr]
=
G_{\mathcal F}\left(\frac pX\right)+o_{\mathcal F,K}(1)
\end{equation*}
locally uniformly along primes $p=p(X)$ with $p/X\in K$.
Then, for every fixed integer $k\geq2$ and every compact set $K\Subset(0,\infty)$,
\begin{equation*}
\sqrt X\,
\mathbb E_{\mathcal F_{X,Y}}
\bigl[\varepsilon(f)\lambda_{f,k}(p)\bigr]
=
G_{\mathcal F}\left(\frac{p^k}{X}\right)+o_{\mathcal F,k,K}(1)
\end{equation*}
as $X\to\infty$ along primes $p=p(X)$ satisfying $p^k/X\in K$.
Thus the same function $G_{\mathcal F}$ applies to every fixed power.
\end{conjecture}

Here ``universal'' refers to the common profile sampled by every fixed power within one family; different families may have different profiles.

For weight $2$ modular forms, one precise instance of the conjecture follows from the calculation in
Section~\ref{sec:scaled-murmuration}.
For this fixed-weight level-aspect instance, we take the conductor scale to be $C(f)=N$.
The next theorem is a short form of that result.
It also records the known prime case.

For a normalized weight $2$ newform $f$, let
$\lambda_{\sym^k f}(P)$ denote the coefficient at $P$ of the local symmetric-power Euler factor in unitary normalization.  Thus, at an unramified prime,
\begin{equation*}
\lambda_{\sym^k f}(P)=\frac{a_f(P^k)}{P^{k/2}}.
\end{equation*}

\begin{theorem}[The common profile in the weight $2$ level aspect]
\label{thmIntroProfile}
There is a single continuous function $G:(0,\infty)\to\mathbb R$ with the following property.
For every fixed integer $k\geq1$, let $X$, $Y$, and the prime $P$ tend to infinity while $P^k/X$ remains in a fixed compact subset of $(0,\infty)$.
For $k\geq2$, assume the short-window and growth conditions of Theorem~\ref{thm:scaled-murmuration}; for $k=1$, assume the conditions of Zubrilina~\cite[Theorem~1]{Zubrilina2025}.
Then, for normalized weight $2$ newforms of squarefree level in $[X,X+Y]$ with trivial nebentypus,
\begin{equation*}
\sqrt X\,
\mathbb E_{X,Y}
\bigl[\varepsilon(f)\lambda_{\sym^k f}(P)\bigr]
=
G\left(\frac{P^k}{X}\right)+o_k(1),
\end{equation*}
locally uniformly in $P^k/X$.
\end{theorem}

\begin{remark}
The case $k=1$ is due to Zubrilina.
Kundu and M\"uller~\cite{KunduMuller2026} analyzed the square case.  Their trace formula has the range $r\leq2P/\sqrt N$, which yields the cutoff $r\leq2\sqrt y$ below.  The cutoff $r\leq\sqrt y$ in their final displayed profile omits this factor of $2$.
Theorem~\ref{thm:scaled-murmuration} gives the corrected square-case cutoff in the present range and proves the statement for every fixed $k\geq3$.
\end{remark}

The function in Theorem~\ref{thmIntroProfile} is explicit.
With $D_0$ defined immediately before Theorem~\ref{thm:scaled-murmuration} and $A$, $B$, and $C(r)$ defined in its statement, put
\begin{equation}
\label{eqMainProfile}
\mathcal M(y)
=
\frac{12}{\pi D_0}
\left(
A\sqrt y+
B\sum_{1\leq r\leq2\sqrt y}
C(r)\sqrt{4y-r^2}
-\pi y
\right)
\end{equation}
and
\begin{equation*}
G(y)=\frac{\mathcal M(y)}{\sqrt y}.
\end{equation*}
The constants in Zubrilina's notation specialize to
\begin{equation*}
\alpha=\frac{12B}{\pi D_0},\qquad
\beta=\frac{12A}{\pi D_0},\qquad
\gamma=\frac{12}{D_0},\qquad
\nu(r)=C(r).
\end{equation*}
Thus the prime formula has exactly the same $\mathcal M$ as formula~\eqref{eqMainProfile}.
For $k\geq2$, division of Theorem~\ref{thm:scaled-murmuration} by $P^{k/2}=\sqrt{X(P^k/X)}$ gives the displayed normalized formula.
The compact condition on $P^k/X$ keeps this division uniform.

The theorem suggests a prediction method under a fixed coefficient budget.
Once $a_p$ is known, the Hecke relation produces $a_{p^2}$ and $a_{p^3}$ without another database query.
These derived coordinates sample the murmuration at new effective positions.
We therefore keep the full prime panel and add a small number of prime-power coordinates.

The derived coordinates are deterministic functions of the original prime coefficients, so the augmented and original data contain the same arithmetic information in the information-theoretic sense.
Their value lies in enlarging the collection of functions available to a linear classifier.
The empirical question is whether this structured representation makes rank and root number more visible.

Our experiments give an answer that depends on the task.
On five overlapping test assignments drawn from one fixed balanced population, the joint quadratic and cubic augmentation improves direct root-number balanced accuracy in all five cases.
Its mean gain is $0.7425$ percentage points.
This removes $2.56$ percent of the errors made by the ordinary prime model.
The improvement for rank $0$ versus rank $1$ is smaller in absolute size; applying formula~\eqref{eqErrorReduction} to the five-assignment mean accuracies gives a relative error reduction of $7.77$ percent.
Performance remains task-dependent: the same augmentation lowers the accuracy of the three-class rank model.
Additional matched experiments compare the Hecke coordinates with ordinary monomials, add the log conductor and low-prime bad-reduction indicators in separate controls, hold out complete conductor quintiles, and replace stochastic-gradient fitting with deterministic logistic regression.
They locate the predictive gain in the shared quadratic--cubic function class, confirm an incremental contribution after the explicit controls, and show that transport across conductor ranges is less stable than random-assignment performance.

\Needspace{0.3\textheight}

\begin{figure}[H]
\centering
\includegraphics[width=0.82\textwidth]
{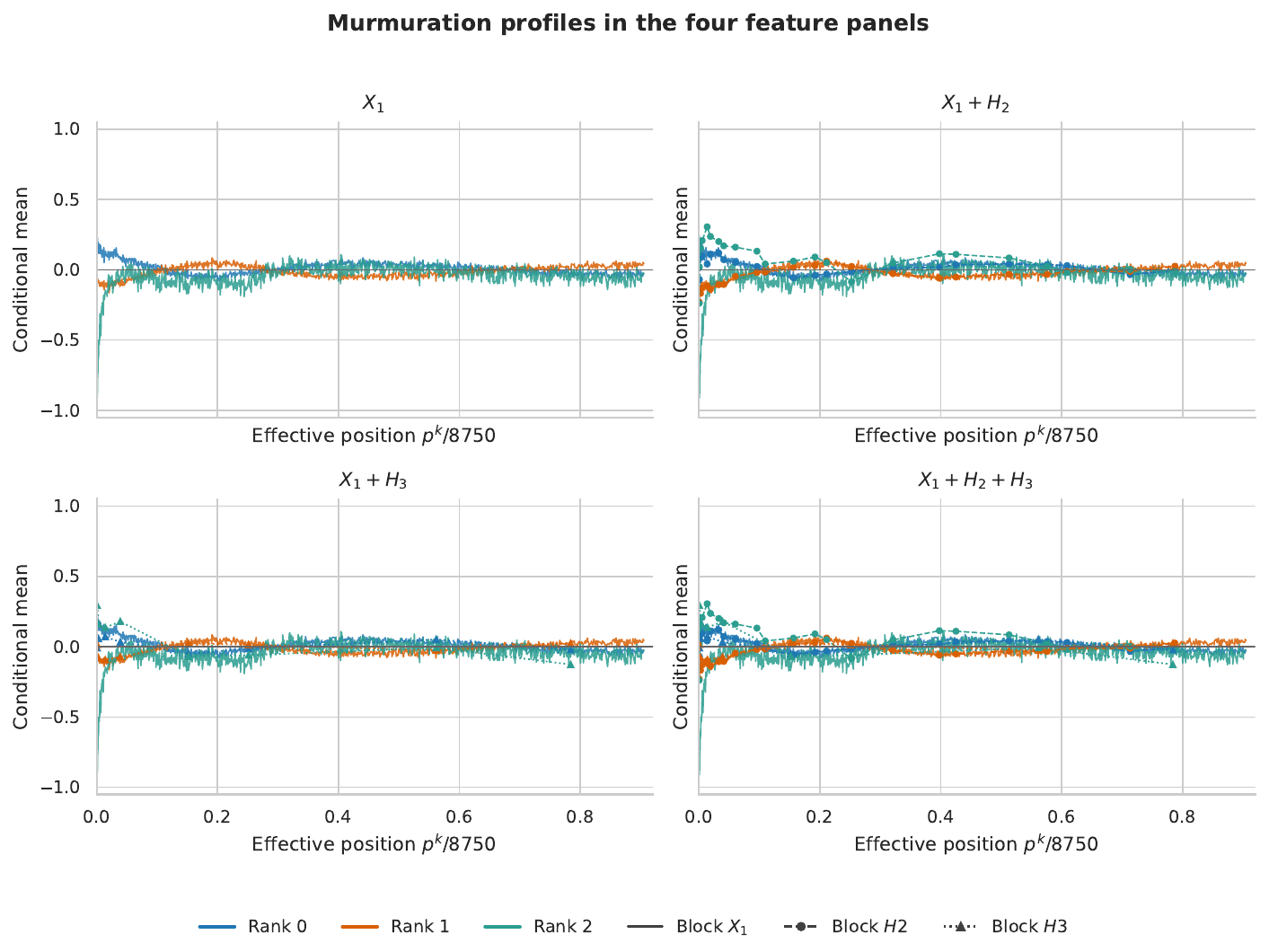}
\caption{Rank-conditioned murmuration profiles in the four feature representations.}
\label{figRepresentations}
\end{figure}

\medskip

Figure~\ref{figRepresentations} illustrates the rank-conditioned normalized profiles used in the experiment.
For a normalized prime coefficient $x_p=a_p/\sqrt p$, the quadratic and cubic coordinates are
\begin{equation*}
H_2(x_p)=x_p^2-1
\qquad\text{and}\qquad
H_3(x_p)=x_p^3-2x_p.
\end{equation*}
They are placed at the effective positions $p^2$ and $p^3$, respectively.
The ordinary prime block contains $1000$ positions, while the quadratic and cubic blocks contain $23$ and $8$ positions.
All profiles are computed from elliptic curves with
$7500\leq N\leq10000$.
In Figure~\ref{figRepresentations}, color indicates the analytic rank and line style indicates the Hecke block.  Each block is drawn only along its own effective positions.
No line joins points belonging to different powers.
The four panels reuse the same three coefficient blocks and form nested views of a single murmuration experiment.

We restrict the main augmentation to $k\leq3$ because, under the cutoff $q=7919$, the condition $p^k\leq q$ leaves only a few small base primes for $k\geq4$; these primes are also frequent places of bad reduction in the sample.
Experiments with a substantially larger prime cutoff could include higher prime-power coordinates.

We also study a complementary substitution design.
It uses one fixed symmetric-power block $H_k$ by itself after removing the ordinary-prime panel.
For $k=2,\ldots,12$, these models usually perform near chance when the labels are the rank and root number of the original elliptic curve.
This stress test shows that a sparse prime-power sequence loses much of the signal carried by the ordinary prime sequence.
The universal profile therefore motivates retaining the prime panel and augmenting it with selected prime-power coordinates.

Three statements should be kept separate.
The first is an asymptotic statement about a mean weighted by the root number.
This is the content of Conjecture~\ref{conjUniversal} and
Theorem~\ref{thmIntroProfile}.
The second is a representation statement.
Hecke relations supply quadratic and cubic functions of the known prime coefficients to a linear model.
The third is an empirical statement about prediction on held-out data.
The theorem motivates the choice of representation; the observed accuracy gain requires separate empirical evidence.

This distinction also explains why the positive and negative experiments belong together.
The common profile identifies effective positions where a weak average bias may occur.
The map $x_p\mapsto H_k(x_p)$ can discard other features of the original coefficient, so a higher-power block may display the common mean profile while performing poorly on its own.
Retaining $H_1$ and adding selected higher blocks uses both observations.

The Sato--Tate law describes the limiting distribution of normalized Frobenius traces when sufficiently many primes are sampled.
Before this limiting distribution becomes dominant, finite families may still retain weak biases depending on the conductor and other global arithmetic data.
Murmuration may be viewed as a distributional pattern
in this pre-Sato--Tate regime.
It becomes visible when elliptic curves are ordered by conductor and separated according to an arithmetic
label.
This viewpoint is related to the pseudorandomness conjecture proposed by Mok in 2023~\cite{Mok2023}, which predicts multivariable pseudorandom behavior for vectors formed from elliptic curve coefficients.
Later work of Mok and Zheng~\cite{MokZheng2025a,MokZheng2025b} provided numerical evidence and applications for this conjectural picture.

The paper is organized as follows.  Section~\ref{secPrimePowerMainTerm} formulates the von Mangoldt average over all prime powers and gives the exact cancellation criterion that preserves the prime-averaged leading term; it then treats fixed prime-power coordinates in the weight $2$ level aspect.
Section~\ref{secFeatures} explains the Hecke feature map and the restricted linear model.
Section~\ref{sec:scaled-murmuration} proves the common profile in the weight $2$ level aspect.
Section~\ref{secDesign} describes the elliptic curve data and the fixed coefficient budget.  Section~\ref{secResults} reports the murmuration plots, the prediction results, the polynomial and conductor controls, the quantity controls, and the single-power substitution experiment.
The appendices record the full fitting protocol and the error accounting.

\section*{Acknowledgements}
S.H. would like to thank Bingrong Huang and Haozhe Gou for their valuable feedback.
S.H. was partially supported by NSFC (No. 12601008), and the China Postdoctoral Science Foundation (No. 2026M793348).

The experiment code was generated by ChatGPT 5.6 Sol.

\section{Prime powers in averaged and coordinatewise trace formulas}
\label{secPrimePowerMainTerm}

The weight-aspect theorem of Bober, Booker, Lee, and Lowry-Duda~\cite[Theorem~1.1]{BBLLD2026} and the Maass-form theorem of Booker, Lee, Lowry-Duda, Seymour-Howell, and Zubrilina~\cite[Theorem~1.1]{BLLDSHZ2024} place a prime Hecke index on the scale of the analytic conductor.
Their underlying trace formulas accept general integral Hecke indices in the relevant sign.
This leads to two complementary formulations.
The first enlarges the outer prime average to the natural von Mangoldt average over all prime powers.
The second fixes one prime-power index and studies its conductor-scaled position.

\subsection{The Hecke index and the effective coordinate}

Let $r\geq2$ be fixed, let $P$ be an unramified prime, and put $Q=P^r$.
We use $\kappa$ for a holomorphic weight, reserving $r$ for the degree of the prime power.
In unitary normalization the Hecke relation is
\begin{equation}
\label{eqArchimedeanHeckeIdentity}
\lambda_f(P^r)=U_r\left(\frac{\lambda_f(P)}{2}\right),
\end{equation}
where $U_r$ is the Chebyshev polynomial of the second kind.
For the weight $2$ normalization used later, the same identity reads
\begin{equation}
\label{eqPrimePowerHeckeIdentity}
P^{r/2}\lambda_{\sym^r f}(P)=a_f(P^r)=a_f(Q).
\end{equation}
Thus a prime-power coefficient enters each trace formula through the single index $Q$.

For level $1$ holomorphic forms, the Eichler--Selberg formula used in~\cite[Section~2]{BBLLD2026} is valid for every positive integer $n$ and has the form
\begin{equation}
\label{eqHolomorphicGeneralTrace}
\operatorname{Tr}(T_n\mid S_\kappa(1))
=A_1(n)-A_2(n)-A_3(n)+A_4(n).
\end{equation}
The oscillatory term $A_2(n)$ is built from the discriminants $t^2-4n$ and the phase
\begin{equation*}
\phi_{t,n}=\arcsin\left(\frac{t}{2\sqrt n}\right).
\end{equation*}
Substitution of $n=Q$ therefore replaces $t^2-4P$ by $t^2-4Q$ and places the phase at $Q/\mathcal N(K)$, where $\mathcal N(K)$ is the analytic-conductor scale of the weight window.

For the Maass family, the Selberg trace formula in~\cite[Theorem~2.1]{BLLDSHZ2024} is valid for every nonzero integer $n$.
The identity
\begin{equation*}
a_j(-Q)=\varepsilon_j a_j(Q)
\end{equation*}
puts the root-number-weighted coefficient on the spectral side at $n=-Q$.
The leading hyperbolic terms then contain $t^2+4Q$ and
\begin{equation*}
G\left(\log\left(\frac{(|t|+\sqrt{t^2+4Q})^2}{4Q}\right)\right).
\end{equation*}
The same substitution places the Maass calculation at $Q/\mathcal N(R)$.
Thus both formulas retain their geometric phase and conductor exponent after $P$ is replaced by $P^r$.
They identify the effective coordinate
\begin{equation*}
y=\frac{P^r}{\mathcal C}
\end{equation*}
while the averaged leading term also depends on the sampling measure for the Hecke indices.

\subsection{A transfer criterion for the average over all prime powers}

The natural extension of the outer prime sum uses the von Mangoldt function.
For a compact interval $E=[a,b]\Subset(0,\infty)$ with $0<a<b$ and a conductor scale $\mathcal C$, put
\begin{equation}
\label{eqPrimeAndAllPowerMasses}
\Theta_E(\mathcal C)
=\sum_{\substack{P\ \mathrm{prime}\\ P/\mathcal C\in E}}\log P,
\qquad
\Psi_E(\mathcal C)
=\sum_{\substack{m\geq1,\ P\ \mathrm{prime}\\ P^m/\mathcal C\in E}}\log P
=\sum_{\substack{n\geq1\\ n/\mathcal C\in E}}\Lambda(n).
\end{equation}
Thus every index $n=P^m$ carries the weight $\Lambda(n)=\log P$.
The trace formula sees $n$ itself, so its discriminant is $t^2-4n$ in the holomorphic case and $t^2+4n$ in the Maass case.

\begin{proposition}
\label{propAllPrimePowersAverage}
Assume, in each case, the hypotheses and parameter range of the corresponding cited theorem.
For the holomorphic weight family, set $\mathcal C=\mathcal N(K)$ and
\begin{equation*}
\mathscr W_{K,H,\delta}
=\{(k,f):k\equiv2\delta\pmod4,\ |k-K|\leq H,\ f\in H_k(1)\},
\end{equation*}
where $H_k(1)$ is the normalized Hecke eigenbasis of $S_k(1)$ used in~\cite{BBLLD2026}.
For the Maass family, set $\mathcal C=\mathcal N(R)$ and
$\mathscr M_{R,H}=\{j:|r_j-R|\leq H\}$.
Each assertion below uses the corresponding choice of $\mathcal C$, $M_{\mathcal C}$, and $T_{\mathcal C}$.
In the corresponding case, define
\begin{equation*}
M_{\mathcal C}=
\begin{cases}
\#\mathscr W_{K,H,\delta},& \mathcal C=\mathcal N(K),\\
\#\mathscr M_{R,H},& \mathcal C=\mathcal N(R),
\end{cases}
\qquad
T_{\mathcal C}(n)=
\begin{cases}
\displaystyle\sum_{(k,f)\in\mathscr W_{K,H,\delta}}\lambda_f(n),
& \mathcal C=\mathcal N(K),\\[6pt]
\displaystyle\sum_{j\in\mathscr M_{R,H}}\varepsilon_j a_j(n),
& \mathcal C=\mathcal N(R).
\end{cases}
\end{equation*}
Define the aggregate contribution from proper prime powers by
\begin{equation}
\label{eqProperPowerCancellation}
R_{\geq2}
=\sum_{\substack{m\geq2,\ P\ \mathrm{prime}\\P^m/\mathcal C\in E}}
(\log P)T_{\mathcal C}(P^m).
\end{equation}
As $K\to\infty$ in the holomorphic case and $R\to\infty$ in the Maass case, the all-prime-power average has the same leading term as the corresponding prime-only average if and only if
\begin{equation}
\label{eqProperPowerCancellationCriterion}
R_{\geq2}
=o_E\left(M_{\mathcal C}\mathcal C^{1/2}\right).
\end{equation}
Under this equivalent condition, the holomorphic weight family satisfies
\begin{equation}
\label{eqAllPowerWeightAspect}
\frac{
\displaystyle\sum_{\substack{m\geq1,\ P\ \mathrm{prime}\\P^m/\mathcal C\in E}}
(\log P)\sum_{(k,f)\in\mathscr W_{K,H,\delta}}\lambda_f(P^m)
}{
\displaystyle\Psi_E(\mathcal C)
\sum_{(k,f)\in\mathscr W_{K,H,\delta}}1
}
=\frac{(-1)^\delta}{\sqrt{\mathcal C}}
\left(\frac{\nu(E)}{|E|}+o_{E,\varepsilon}(1)\right).
\end{equation}
The Maass family satisfies
\begin{equation}
\label{eqAllPowerMaassAspect}
\frac{
\displaystyle\sum_{\substack{m\geq1,\ P\ \mathrm{prime}\\P^m/\mathcal C\in E}}
(\log P)\sum_{j\in\mathscr M_{R,H}}\varepsilon_j a_j(P^m)
}{
\displaystyle\Psi_E(\mathcal C)
\sum_{j\in\mathscr M_{R,H}}1
}
=\frac{1}{\sqrt{\mathcal C}}
\left(\frac{\nu(E)}{|E|}+o_{E}(1)\right).
\end{equation}
The measure $\nu$ is the same measure that occurs in the prime-only statements.
\end{proposition}

\begin{proof}
Suppress the argument $\mathcal C$ in $\Theta_E(\mathcal C)$ and $\Psi_E(\mathcal C)$ within this proof.
The all-power numerator is
\begin{equation*}
S_{\Lambda}
=S_{\mathbb P}+R_{\geq2},
\qquad
R_{\geq2}
=\sum_{\substack{m\geq2,\ P\ \mathrm{prime}\\P^m/\mathcal C\in E}}
(\log P)T_{\mathcal C}(P^m),
\end{equation*}
where $S_{\mathbb P}$ is the numerator in the published prime average.
The prime number theorem and the elementary bound for proper prime powers give
\begin{equation}
\label{eqProperPowerMass}
\Theta_E(\mathcal C)\sim |E|\mathcal C,
\qquad
\Psi_E(\mathcal C)-\Theta_E(\mathcal C)
\ll_E\mathcal C^{1/2}.
\end{equation}

Finally,
\begin{equation*}
\frac{S_{\Lambda}}{M_{\mathcal C}\Psi_E}
-\frac{S_{\mathbb P}}{M_{\mathcal C}\Theta_E}
=\frac{R_{\geq2}}{M_{\mathcal C}\Psi_E}
+\frac{S_{\mathbb P}}{M_{\mathcal C}}
\left(\frac1{\Psi_E}-\frac1{\Theta_E}\right).
\end{equation*}
The published prime asymptotic gives $S_{\mathbb P}/M_{\mathcal C}=O_E(\mathcal C^{1/2})$; together with~\eqref{eqProperPowerMass}, this makes the second term $O_E(\mathcal C^{-1})$.
Since $\Psi_E\asymp_E\mathcal C$, the difference of the two normalized averages is $o_E(\mathcal C^{-1/2})$ exactly when~\eqref{eqProperPowerCancellationCriterion} holds.
The two published prime asymptotics then give~\eqref{eqAllPowerWeightAspect} and~\eqref{eqAllPowerMaassAspect}.
\end{proof}

\begin{conjecture}
Under the hypotheses and parameter ranges of the corresponding prime-only theorems, equation~\eqref{eqProperPowerCancellationCriterion} holds for both families.
\end{conjecture}

The cited prime-only theorems supply the $m=1$ asymptotic.
The conjecture isolates the additional cancellation estimate required for the von Mangoldt extension.
The proper-power mass in~\eqref{eqProperPowerMass} lies at the critical scale of the murmuration numerator, so the mass estimate alone cannot yield the required small-$o$ bound.
Its proof requires a separate uniform analysis of the square and higher-power terms in the smoothed trace formulas, together with the sharp-cutoff errors.

This extension uses all powers in one von Mangoldt average.
A fixed-degree average with a conductor-linear Jacobian weight defines a separate statistic whose local factors and exceptional trace terms require their own analysis.
We now turn to a coordinatewise result in which one fixed Hecke index $Q=P^k$ is held during the level average.

\subsection{Exact transfer in the weight-two level aspect}

The family studied in this paper averages the level $N$ while keeping the Hecke index $Q=P^k$ fixed.
The Skoruppa--Zagier formula then gives an exact common profile at this fixed effective index.
For $k\geq2$, the formula used below is~\cite[Theorem~2.1]{KunduMuller2026}
\begin{equation}
\label{eqPrimePowerTraceRoadmap}
\sum_{f\in H_{\mathrm{new}}(N)}
P^{k/2}\lambda_{\sym^k f}(P)\varepsilon(f)
=
\frac12H_1(-4P^kN)
+\sum_{0<s\leq2\sqrt{P^k/N}}H_1(s^2N^2-4P^kN)
-\sum_{j=0}^kP^j.
\end{equation}
As proved in Section~\ref{sec:scaled-murmuration}, averaging over $N\in[X,X+Y]$ gives leading parabolic and hyperbolic contributions that depend on $k$ through
\begin{equation*}
y=\frac{P^k}{X}.
\end{equation*}
The term $P^k$ in the elementary divisor sum contributes $-\pi y$ to the profile.
The remaining powers $\sum_{0\leq j<k}P^j$, the levels divisible by $P$, and the window-approximation errors are lower order under the hypotheses of Theorem~\ref{thm:scaled-murmuration}.
If $y$ remains in a compact set $K\Subset(0,\infty)$, the averaged unnormalized coefficient has the form
\begin{equation}
\label{eqPrimePowerMainTermRoadmap}
\frac{
\displaystyle\sideset{}{'}\sum_{N\in[X,X+Y]}
\sum_{f\in H_{\mathrm{new}}(N)}
P^{k/2}\lambda_{\sym^k f}(P)\varepsilon(f)
}{
\displaystyle\sideset{}{'}\sum_{N\in[X,X+Y]}
\sum_{f\in H_{\mathrm{new}}(N)}1
}
=\mathcal M(y)+o_{k,K}(1).
\end{equation}
Since $P^{k/2}=\sqrt{Xy}$, division by this normalization gives
\begin{equation}
\label{eqPrimePowerNormalizedRoadmap}
\sqrt X\,\mathbb E_{X,Y}
\bigl[\varepsilon(f)\lambda_{\sym^k f}(P)\bigr]
=G\left(\frac{P^k}{X}\right)+o_{k,K}(1),
\qquad
G(y)=\frac{\mathcal M(y)}{\sqrt y}.
\end{equation}
For this weight $2$ squarefree-level family, every fixed power therefore samples the same leading function at its effective position.
The square case is developed in Kundu and M\"uller~\cite{KunduMuller2026}, and Section~\ref{sec:scaled-murmuration} supplies the uniform bookkeeping used here for every fixed $k\geq2$.
The prime case of Zubrilina~\cite{Zubrilina2025} supplies the $k=1$ baseline.

This coordinatewise first-moment statement motivates adding prime-power samples to a feature panel.
The admissible range and the error constants may depend on the fixed power.
Joint distributions, predictive performance, and finite-sample compactness remain separate questions.

\section{Hecke features and prediction under a fixed budget}
\label{secFeatures}

\subsection{Prime-power coordinates}

Let $E$ be an elliptic curve and let
\begin{equation*}
x_p(E)=\frac{a_p(E)}{\sqrt p}
\end{equation*}
at a prime of good reduction.  At a bad prime, the experimental panel records a distinguished missing value.  Every derived coordinate is formed before imputation and inherits the missing pattern of its base-prime coordinate.
The Hecke recurrence gives
\begin{equation*}
H_1(x)=x,\qquad
H_2(x)=x^2-1,\qquad
H_3(x)=x^3-2x.
\end{equation*}
Thus
\begin{equation*}
H_k\bigl(x_p(E)\bigr)
=
\frac{a_{p^k}(E)}{p^{k/2}}
\end{equation*}
for the good prime coordinates used in the experiment.  More generally, these functions are the Chebyshev characters
\begin{equation*}
H_k(x)=U_k(x/2).
\end{equation*}

Fix the largest observed prime $q$, and denote the ordinary panel by $X_1$.
It contains $H_1(x_p)$ for $p\leq q$.
A power $k$ contributes a derived coordinate only when
\begin{equation*}
p^k\leq q.
\end{equation*}
This rule gives every representation the same largest effective index.
In the experiment $q=7919$.
There are $1000$ ordinary prime coordinates, $23$ quadratic coordinates, and $8$ cubic coordinates.

\subsection{Information content and model class under augmentation}

Let $\mathbf X$ denote the ordinary prime vector, including its missing-value symbols, and let $\mathbf Z$ denote a block of derived coordinates.
For any discrete arithmetic label $L$, the deterministic relation $\mathbf Z=g(\mathbf X)$ gives
\begin{equation*}
I(L;\mathbf Z\mid\mathbf X)=0.
\end{equation*}
Thus an unrestricted predictor receives the same information from $\mathbf X$ and $(\mathbf X,\mathbf Z)$.

\begin{proposition}
\label{propRepresentation}
The raw symbol-valued augmented data and the ordinary prime data generate the same sigma-field.
Fix a training sample, give each derived coordinate the missing mask of its base coordinate, and fill every numeric column with its training-sample mean.
Within an affine linear model, the imputed panel $X_1+H_2$ has the same column span as the imputed ordinary prime panel together with the selected squares $x_p^2$.
The corresponding statement holds for $X_1+H_3$ and the selected cubes $x_p^3$.
Subsequent columnwise standardization preserves these affine spans for nonconstant columns.
\end{proposition}

\begin{proof}
Every $H_k(x_p)$ is a polynomial in $x_p$, while the augmented panel contains the original vector.
The equality of sigma fields follows.
Let a tilde denote training-mean imputation with the common mask.
Linearity of the mean gives
\begin{equation*}
\widetilde{H_2(x_p)}=\widetilde{x_p^2}-1
\end{equation*}
and
\begin{equation*}
\widetilde{H_3(x_p)}=\widetilde{x_p^3}-2\widetilde{x_p}.
\end{equation*}
The constant function and $x_p$ are already present in the affine model.
Centering and rescaling a nonconstant column are invertible affine transformations, which proves the final statement.
\end{proof}

This equivalence concerns the unregularized column span. The $\ell_2$ regularization depends on the chosen basis after standardization.
We therefore describe the empirical gain as an effect of a polynomial representation organized by the Hecke relations.
The experiment evaluates this Hecke-organized basis; other quadratic or cubic parameterizations may behave similarly.

A linear score has a different restriction.
After augmentation it takes the form
\begin{equation}
\label{eqScore}
s(E)=\beta_0+
\sum_{p\leq q}\beta_{1,p}x_p(E)+
\sum_{p^2\leq q}\beta_{2,p}\bigl(x_p(E)^2-1\bigr)+
\sum_{p^3\leq q}\beta_{3,p}\bigl(x_p(E)^3-2x_p(E)\bigr).
\end{equation}
The new terms turn an additive linear boundary into an additive polynomial boundary at small primes.
A small number of such coordinates can matter if the residual arithmetic signal is concentrated at their effective positions.

The gain depends on more than the number of added columns.
The following standard identity identifies the relevant quantity.

\begin{proposition}
\label{propLinearIncrement}
Let $L$, $\mathbf X$, and $\mathbf Z$ have finite second moments, with $L$ centered.  Let $\mathbf Z^\perp$ be the residual after the best affine linear projection of $\mathbf Z$ on $\mathbf X$.
If $\operatorname{Var}(\mathbf Z^\perp)$ is invertible, the decrease in the best population affine squared error after adding $\mathbf Z$ is
\begin{equation}
\label{eqLinearGain}
\operatorname{Cov}(\mathbf Z^\perp,L)^{\mathsf T}
\operatorname{Var}(\mathbf Z^\perp)^{-1}
\operatorname{Cov}(\mathbf Z^\perp,L).
\end{equation}
\end{proposition}

\begin{proof}
The affine residual from the best fit on $\mathbf X$ is orthogonal to the affine span of $\mathbf X$.
Adding $\mathbf Z$ is therefore equivalent to fitting that residual with $\mathbf Z^\perp$.
The normal equation gives the coefficient
\begin{equation*}
\operatorname{Var}(\mathbf Z^\perp)^{-1}
\operatorname{Cov}(\mathbf Z^\perp,L).
\end{equation*}
Substitution into the squared error gives formula~\eqref{eqLinearGain}.
\end{proof}

The gain is controlled by residual signal and covariance; the nominal column count alone is insufficient.  Formula~\eqref{eqLinearGain} serves only as motivation.
Our classifiers use logistic loss, and their performance is measured on held-out test data.

\subsection{Profile signal and the size of a power block}

Fix $J$ and pairs $(k_j,P_j)$ with each $k_j$ fixed.  Put
\begin{equation*}
y_j=\frac{P_j^{k_j}}X,
\qquad
Z_j=\lambda_{\sym^{k_j}f}(P_j),
\qquad
\mathbf Z=(Z_1,\ldots,Z_J)^{\mathsf T},
\end{equation*}
and assume that all $y_j$ lie in a common compact subset of $(0,\infty)$.
Suppose that the covariance matrix of $\mathbf Z$ satisfies
$\Sigma_J=I_J+o(1)$ in operator norm.  Write
\begin{equation*}
\mathbf b_J=\operatorname{Cov}(\mathbf Z,\varepsilon).
\end{equation*}
Suppose also that
\begin{equation*}
\mathbf b_J
=
X^{-1/2}
\bigl(G(y_1),\ldots,G(y_J)\bigr)^{\mathsf T}
+o(X^{-1/2}).
\end{equation*}
Then the part of the optimal affine squared signal visible in these coordinates is
\begin{equation}
\label{eqProfileEnergy}
\mathbf b_J^{\mathsf T}\Sigma_J^{-1}\mathbf b_J
=
\frac{1}{X}\sum_{j=1}^J G(y_j)^2+o(X^{-1}).
\end{equation}
This follows directly by matrix inversion.  It explains why several weak profile samples may contribute together when their covariance is well conditioned.
Theorem~\ref{thmIntroProfile} supplies each one-coordinate signed mean.
To obtain the covariance vector above, one also needs
\begin{equation*}
\bigl(\mathbb E\varepsilon\bigr)
\bigl(\mathbb E Z_j\bigr)=o(X^{-1/2})
\end{equation*}
for every fixed $j$.
This centering condition, the covariance assumption, and the passage to finite-sample classification are separate inputs.
If the block is first residualized against $X_1$, the same calculation with residual covariances describes the incremental squared-error signal in Proposition~\ref{propLinearIncrement}.

The effective index rule also gives a simple count.  The number of available degree $k$ coordinates is
\begin{equation}
\label{eqPowerCount}
m_k(q)=\pi\bigl(q^{1/k}\bigr).
\end{equation}
The prime number theorem gives
\begin{equation*}
m_k(q)\sim \frac{kq^{1/k}}{\log q}
\end{equation*}
and
\begin{equation*}
\frac{m_k(q)}{\pi(q)}
\sim kq^{1/k-1}.
\end{equation*}
For $q=7919$, the exact values are $m_1=1000$, $m_2=23$, and $m_3=8$.
Higher degree blocks become very sparse on the effective integer scale.
Equation~\eqref{eqProfileEnergy} shows why the useful quantity is a sum of squared residual signals.
The count in formula~\eqref{eqPowerCount} is only one part of that quantity.

\subsection{Prediction model}

For a binary label $L$, the fitted model has the form
\begin{equation*}
\widehat{\mathbb P}(L=1\mid V)
=
\frac{\exp(\alpha+\beta^{\mathsf T}V)}
{1+\exp(\alpha+\beta^{\mathsf T}V)}.
\end{equation*}
For the three rank classes, we use multiclass logistic regression with a one-versus-rest scheme.  Every model uses $\ell_2$ regularization.
Missing values and scales are estimated from the training data only.

The main performance measure is balanced accuracy.  If $A_0$ is the balanced accuracy of the ordinary prime model and $A_1$ is the accuracy after augmentation, we report
\begin{equation}
\label{eqErrorReduction}
\mathcal R_{\mathrm{err}}
=
100\frac{A_1-A_0}{1-A_0}.
\end{equation}
On the class-balanced tests used below, this is the percentage of baseline errors removed by the augmented model.
We also report the Matthews correlation coefficient and the area under the receiver operating characteristic curve when they help to check the direction of the result.

For a binary task, balanced accuracy is the mean of the two class recalls.
For the three-class task, it is the mean of the three class recalls.
Since every test split is exactly balanced, balanced accuracy equals ordinary accuracy in the reported experiments.
The error counts can therefore be read directly from the displayed accuracy.

\section{Scaled murmuration in a short level window}
\label{sec:scaled-murmuration}

In this section we prove a correlation between the coefficients of a fixed symmetric power and the root number of the underlying modular form.
We work with weight $2$, squarefree level, and trivial nebentypus, as in Zubrilina~\cite{Zubrilina2025} and Kundu--M\"uller~\cite{KunduMuller2026}.
Throughout the section, $k$ is fixed.

The proof has three main steps.
The Skoruppa--Zagier trace formula expresses the root-number-weighted Hecke trace as a parabolic class-number term, a hyperbolic class-number term, and an elementary term.
We average the first two terms separately over the short
squarefree level window and then divide by the total number of newforms.
After writing $Q=P^k$, the leading terms depend on $k$ only through $Q$, while the normalization is
$P^{k/2}=\sqrt Q$.
This produces the same function of $y=Q/X$ for every fixed $k$.
The admissible range and the error term may depend on $k$.

For $P\nmid N$, the $P$-th coefficient of the symmetric-power $L$-function satisfies
\[
a_{\sym^k f}(P)=a_f(P^k).
\]
We normalize it by
\[
\lambda_{\sym^k f}(P)
:=
\frac{a_{\sym^k f}(P)}{P^{k/2}}
=
\frac{a_f(P^k)}{P^{k/2}}.
\]
When $P\mid N$, we use the coefficient given by the local symmetric-power Euler factor.
In this case
\[
P^{k/2}\lambda_{\sym^k f}(P)=a_f(P)^k.
\]

Let
$S_{\mathrm{new}}(N)=S_2^{\mathrm{new}}(\Gamma_0(N))$ with trivial nebentypus, and let $H_{\mathrm{new}}(N)$ be its normalized Hecke eigenbasis.
The symbol $\varepsilon(f)$ denotes the root number of $f$.

Set
\[
D_0=\prod_p\left(1-\frac{1}{p(p+1)}\right).
\]

\begin{theorem}\label{thm:scaled-murmuration}
Let $k\geq 2$ be fixed.  Let $P\neq 2$ be a prime, and suppose that $P$, $X$, and $Y$ tend to infinity.
Assume that
\[
Y=(1+o(1))X^{1-\delta_1},
\qquad
P^k\ll X^{1+\delta_2},
\]
where
\[
0<2\delta_2<\delta_1<\frac18,
\qquad
(k-1)\delta_2<1.
\]
Put
\[
\eta=\frac{\delta_1}{2}-\delta_2,
\qquad
y=\frac{P^k}{X}.
\]
Let
\[
A=\prod_p\left(1+\frac{p}{(p+1)^2(p-1)}\right),
\]
\[
B=\prod_p\left(1-\frac{p}{(p^2-1)^2}\right),
\]
and
\[
C(r)=\prod_{p\mid r}
\left(1+\frac{p^2}{p^4-2p^2-p+1}\right).
\]
Then, for every $\varepsilon>0$,
\begin{multline*}
\frac{
\displaystyle
\sideset{}{'}\sum_{N\in[X,X+Y]}
\ \sum_{f\in H_{\mathrm{new}}(N)}
P^{k/2}\lambda_{\sym^k f}(P)\,\varepsilon(f)
}{
\displaystyle
\sideset{}{'}\sum_{N\in[X,X+Y]}
\ \sum_{f\in H_{\mathrm{new}}(N)}1
}
\\
=
\frac{12}{\pi D_0}
\left(
A\sqrt y
+
B\sum_{1\leq r\leq 2\sqrt y}
C(r)\sqrt{4y-r^2}
-\pi y
\right)
+
O_{k,\varepsilon}\left(
X^{-\eta+\varepsilon}
+X^\varepsilon\frac{\sqrt y}{P}
+\frac{1+y}{P}
\right).
\end{multline*}
Here the prime indicates that the sum is restricted to squarefree integers $N$.
\end{theorem}

\begin{remark}
The factor $P^{k/2}$ is the same normalization as the factors $\sqrt P$ and $P$ in the results for $P$- and $P^2$-coefficients, respectively.
Theorem~\ref{thm:scaled-murmuration} uses the root number of the underlying form $f$.
For even $k$, the resulting correlation is a family-level effect of selection by the underlying conductor.
The statement concerns averages across the family and treats each fixed $k$ separately; its constants may grow with $k$.
\end{remark}

\subsection{The trace formula}

Fix a squarefree positive integer $N$.
Recall that $-\varepsilon(f)$ is the eigenvalue of  $f$ under the Atkin--Lehner involution $W_N$.
Thus, when $P\nmid N$,
\[
\operatorname{tr}\left(-T_{P^k}W_N\mid S_{\mathrm{new}}(N)\right)
=
\sum_{f\in H_{\mathrm{new}}(N)}
a_f(P^k)\varepsilon(f)
=
\sum_{f\in H_{\mathrm{new}}(N)}
P^{k/2}\lambda_{\sym^k f}(P)\varepsilon(f).
\]

We use the following special case of the Skoruppa--Zagier trace formula~\cite[\S2, formulas~(5) and~(7)]{SkoruppaZagier1988}; see also~\cite[Theorem~2.1]{KunduMuller2026}.
Here $H_1$ denotes the Hurwitz class number in the convention of the cited formula: classes with automorphism groups of orders $4$ and $6$ have weights $1/2$ and $1/3$, respectively, and the remaining classes have weight $1$.

\begin{proposition}[Skoruppa--Zagier]\label{prop:trace-formula}
Suppose that $P\nmid N$.  Then
\begin{equation*}
\sum_{f\in H_{\mathrm{new}}(N)}
P^{k/2}\lambda_{\sym^k f}(P)\varepsilon(f)
=
\frac12H_1(-4P^kN)
+
\sum_{0<r\leq 2P^{k/2}/\sqrt N}
H_1(r^2N^2-4P^kN)
-
\sum_{j=0}^kP^j.
\end{equation*}
\end{proposition}

We write $h(D)$ for the Gauss class number of primitive positive-definite binary quadratic forms of discriminant $D<0$, and set $h(D)=0$ for other integers.
For $D>0$, the relation between the Hurwitz and Gauss class numbers is
\[
H_1(-D)
=
\sum_{e^2\mid D}h(-D/e^2)+O(1),
\]
where the uniform $O(1)$ term accounts for the exceptional automorphism weights and vanishes unless $D=3s^2$ or $4s^2$.

Since $P\nmid N$ and $N$ is squarefree, the square divisors of $4P^kN$ that contribute to the Hurwitz class number are $P^{2j}$ and $4P^{2j}$.  Consequently,
\begin{equation}\label{eq:parabolic-decomposition}
\frac12H_1(-4P^kN)
=
\frac12
\sum_{j=0}^{\lfloor k/2\rfloor}
\left(
h(-P^{k-2j}N)+h(-4P^{k-2j}N)
\right)
+O_k(1).
\end{equation}
The exceptional forms of discriminant $-3m^2$ and $-4m^2$ are included in the $O_k(1)$ term.

\subsection{The parabolic term}

We first reduce the powers of $P$ occurring in~\eqref{eq:parabolic-decomposition}.

\begin{lemma}\label{lem:class-number-reduction}
Let $P\nmid N$, with $P$ odd, $N$ squarefree, and $PN>4$.  For $m\geq1$, whenever the arguments are discriminants, one has
\[
h(-P^{2m}N)=P^{m-1}h(-P^2N),
\qquad
h(-4P^{2m}N)=P^{m-1}h(-4P^2N),
\]
and
\[
h(-P^{2m+1}N)=P^mh(-PN),
\qquad
h(-4P^{2m+1}N)=P^mh(-4PN).
\]
\end{lemma}

\begin{proof}
This follows from the class number formula for quadratic orders.
If $\Delta_K$ is a fundamental discriminant and $f\geq1$, then
\[
h(\Delta_Kf^2)
=
\frac{h(\Delta_K)f}{[\mathcal O_K^\times:\mathcal O_f^\times]}
\prod_{q\mid f}
\left(1-\frac{1}{q}\left(\frac{\Delta_K}{q}\right)\right).
\]
For the even powers in the statement, the conductor already contains $P$, and for the odd powers, $P\mid\Delta_K$.
Thus multiplication of the conductor by $P$ does not introduce a new Euler factor and multiplies the class number by $P$.
The condition $PN>4$ excludes the exceptional unit groups of discriminants $-3$ and $-4$, so the unit indices remain unchanged as the conductor is multiplied by $P$.
\end{proof}

The estimates needed for $h(-PN)$ and $h(-P^2N)$ are
Proposition~3.1 of Zubrilina~\cite{Zubrilina2025} and
Proposition~3.1 of Kundu--M\"uller~\cite{KunduMuller2026}, respectively.
Combining them with Lemma~\ref{lem:class-number-reduction} gives the following.

\begin{proposition}\label{prop:parabolic-average}
Under the hypotheses of Theorem~\ref{thm:scaled-murmuration}, the following estimates hold for every $\varepsilon>0$ as $X\to\infty$.
Let
\[
\mathcal P_k
=
\frac{\zeta(2)\pi}{XY}
\sideset{}{'}\sum_{\substack{N\in[X,X+Y]\\P\nmid N}}
\frac12H_1(-4P^kN).
\]
If $k$ is even, then
\[
\mathcal P_k
=
\frac{AP^{k/2}}{\sqrt X}
+
O_{k,\varepsilon}\left(
(PXY)^\varepsilon
\left(
\frac{P^{k/2-1}}{\sqrt X}
+
\frac{P^{k/2+1/6}X^{1/12}}{Y^{5/6}}
+
\frac{P^{k/2}Y}{X^{3/2}}
\right)
\right).
\]
If $k$ is odd, then
\[
\mathcal P_k
=
\frac{AP^{k/2}}{\sqrt X}
+
O_{k,\varepsilon}\left(
(PXY)^\varepsilon
\left(
\frac{P^{k/2-1}}{\sqrt X}
+
\frac{P^{k/2+3/38}}{Y^{16/19}}
+
\frac{P^{k/2}Y}{X^{3/2}}
\right)
\right).
\]
\end{proposition}

\begin{proof}
Suppose first that $k=2m$.
Lemma~\ref{lem:class-number-reduction} gives
\begin{multline*}
\frac12
\sum_{j=0}^{m}
\left(
h(-P^{2m-2j}N)+h(-4P^{2m-2j}N)
\right)
\\
=
\frac12
\sum_{\ell=1}^{m}P^{\ell-1}
\left(
h(-P^2N)+h(-4P^2N)
\right)
+
\frac12\left(h(-N)+h(-4N)\right).
\end{multline*}
Put
\[
A_N=\frac{h(-P^2N)+h(-4P^2N)}2,
\qquad
B_N=\frac{h(-N)+h(-4N)}2.
\]
If $S_m$ denotes the left-hand side of the preceding identity, then
\begin{equation}\label{eq:even-parabolic-algebra}
S_m
=P^{m-1}(A_N+B_N)
+\left(\sum_{\ell=1}^{m}P^{\ell-1}-P^{m-1}\right)A_N
+(1-P^{m-1})B_N.
\end{equation}
In the present notation, Proposition~3.1 of~\cite{KunduMuller2026} states that
\begin{multline}\label{eq:base-even-parabolic}
\frac{\zeta(2)\pi}{XY}
\sideset{}{'}\sum_{\substack{N\in[X,X+Y]\\P\nmid N}}
\frac{
h(-P^2N)+h(-4P^2N)+h(-N)+h(-4N)
}{2}
\\
=
\frac{AP}{\sqrt X}
+
O_\varepsilon\left(
(PXY)^\varepsilon
\left(
\frac1{\sqrt X}
+
\frac{P^{7/6}X^{1/12}}{Y^{5/6}}
+
\frac{PY}{X^{3/2}}
\right)
\right).
\end{multline}
Multiplication by $P^{m-1}$ gives $AP^m/\sqrt X$ and the three analytic errors carrying $(PXY)^\varepsilon$ in the statement.

It remains to bound the final two terms in~\eqref{eq:even-parabolic-algebra}.
For $m\geq2$,
\[
\sum_{\ell=1}^{m}P^{\ell-1}-P^{m-1}=O_k(P^{m-2});
\]
for $m=1$ the difference vanishes.
Both $A_N$ and $B_N$ are nonnegative.  Hence the normalized average of $A_N$ is at most the left-hand side of~\eqref{eq:base-even-parabolic}.  Multiplying the main term and the three errors in that formula by $O_k(P^{m-2})$ produces terms absorbed, respectively, by the three errors in the even assertion of the proposition.  For the remaining component, the standard class-number bound $h(-D)\ll_\varepsilon D^{1/2+\varepsilon}$ gives
\[
\frac{\zeta(2)\pi}{XY}
\sideset{}{'}\sum_{\substack{N\in[X,X+Y]\\P\nmid N}} B_N
\ll_\varepsilon X^{-1/2+\varepsilon}.
\]
Since $1-P^{m-1}=O_k(P^{m-1})$, its contribution is absorbed by
$(PXY)^\varepsilon P^{m-1}/\sqrt X$, after renaming $\varepsilon$.

If $k=2m+1$, the same argument gives
\[
\frac12
\sum_{j=0}^{m}
\left(
h(-P^{2m+1-2j}N)+h(-4P^{2m+1-2j}N)
\right)
\\
=
\frac12
\sum_{\ell=0}^{m}P^\ell
\left(
h(-PN)+h(-4PN)
\right).
\]
Proposition~3.1 of~\cite{Zubrilina2025} reads
\begin{multline}\label{eq:base-odd-parabolic}
\frac{\zeta(2)\pi}{XY}
\sideset{}{'}\sum_{\substack{N\in[X,X+Y]\\P\nmid N}}
\frac{h(-PN)+h(-4PN)}2
\\
=
\frac{A\sqrt P}{\sqrt X}
+
O_\varepsilon\left(
(PXY)^\varepsilon
\left(
\frac1{\sqrt{PX}}
+
\frac{P^{11/19}}{Y^{16/19}}
+
\frac{\sqrt P\,Y}{X^{3/2}}
\right)
\right).
\end{multline}
Multiplying this by
\[
\sum_{\ell=0}^mP^\ell=P^m+O_k(P^{m-1})
\]
multiplies each analytic error by $O_k(P^m)$ and gives the three errors in the odd assertion.  The lower geometric coefficients in the main term contribute
$O_k(P^{m-1/2}/\sqrt X)$, which is absorbed by the first of those errors.
Finally, the total contribution of the $O_k(1)$ term in~\eqref{eq:parabolic-decomposition} is $O_k(X^{-1})$.
\end{proof}

\subsection{The hyperbolic term}

The following elementary observation removes the extra $P$-adic cases which otherwise appear for higher powers.

\begin{lemma}\label{lem:p-coprime}
Assume the hypotheses of Theorem~\ref{thm:scaled-murmuration}.
If
\[
0<r\leq\frac{2P^{k/2}}{\sqrt X}
\]
and
\[
d^2\mid r^2N-4P^k,
\qquad
P\nmid N,
\]
then, for sufficiently large $X$,
\[
P\nmid rd.
\]
\end{lemma}

\begin{proof}
We have
\[
\frac rP
\leq
\frac{2P^{(k-2)/2}}{\sqrt X}
\ll
X^{\frac{-2+(k-2)\delta_2}{2k}}.
\]
Since $(k-1)\delta_2<1$, the exponent is negative.
Hence $r<P$ for sufficiently large $X$, and therefore $P\nmid r$.
If $P\mid d$, then reduction of
\[
d^2\mid r^2N-4P^k
\]
modulo $P$ gives $P\mid r^2N$, which is impossible.
\end{proof}

For the rest of this subsection, put
\[
Q=P^k.
\]
For positive integers $r,d$, put
\[
\mathcal A_{r,d}(Q)
=
\left\{
N\in\mathbb Z_{>0}:
\begin{array}{c}
\mu^2(N)=1,\ P\nmid N,\ d^2\mid r^2N-4Q,\\[1mm]
\displaystyle
\frac{N(r^2N-4Q)}{d^2}\equiv0\ \text{or}\ 1\pmod4
\end{array}
\right\}.
\]
The point of the next lemma is that the last two conditions, although they involve a quotient, are described by fixed residue classes modulo $d^2$.

\begin{lemma}\label{lem:admissible-classes}
Suppose that $P\nmid rd$.  There is a set
$\mathcal R_{r,d}(Q)\subset\mathbb Z/d^2\mathbb Z$ such that
\[
\mathcal A_{r,d}(Q)
=
\left\{
N\in\mathbb Z_{>0}:
\mu^2(N)=1,\ P\nmid N,\
N\bmod d^2\in\mathcal R_{r,d}(Q)
\right\}.
\]
Every class in $\mathcal R_{r,d}(Q)$ is coprime to $d$, and
\[
\#\mathcal R_{r,d}(Q)
=
\begin{cases}
1,&(r,d)=1,\quad 2\nmid rd,\\
1,&(r,d)=1,\quad 2\mid r,\\
1,&(r,d)=2,\quad 2\Vert d,\\
2,&(r,d)=2,\quad 4\mid d,\\
0,&\text{otherwise}.
\end{cases}
\]
If $d$ is even, every integer represented by an admissible class is odd.
Writing
\[
s=\frac{N(r^2N-4Q)}{d^2},
\]
the parity information is
\[
\begin{array}{c|c}
\text{case}&\text{parity of }s\\ \hline
2\Vert d,\ 2\Vert r&\text{even}\\
2\Vert d,\ 4\mid r&\text{odd}\\
4\mid d&\text{one even class and one odd class}.
\end{array}
\]
In particular, the number and parity type of these classes are independent of $k$.
\end{lemma}

\begin{proof}
If $r$ is odd, squarefreeness forces $d$ to be odd, and the discriminant condition is then automatic.
The congruence has the unique solution
\[
a\equiv4Qr^{-2}\pmod {d^2}
\]
when $(r,d)=1$, and has no solution otherwise.

Now write $r=2\ell$.  If $d$ is odd, division by $4$ gives
\[
\ell^2a\equiv Q\pmod {d^2},
\]
so again there is one class precisely when $(\ell,d)=1$.
If $d=2b$, the discriminant condition forces $a$ to be odd.
Reducing the quotient modulo $4$ gives the two possible congruences
\[
\ell^2a\equiv Q\pmod {4b^2},
\qquad
(\ell^2-b^2)a\equiv Q\pmod {4b^2}.
\]
A direct parity check is as follows.  When $2\Vert d$, the integer $b$ is odd.
If $\ell$ is odd, only $\ell^2$ is invertible modulo $4b^2$, and if $\ell$ is even, only $\ell^2-b^2$ is
invertible.
Thus there is one class.
When $4\mid d$, the integer
$b$ is even and $(r,d)=2$ forces $\ell$ to be odd.
Both coefficients are then invertible, and the two solutions are distinct because their equality would imply $b^2a\equiv0\pmod{4b^2}$, which is impossible for odd $a$.
In every other case a common divisor of the coefficient and $d$ would divide $Q$, contradicting
$P\nmid d$.
This also proves that every resulting class is coprime to $d$.

For the parity assertion, the first of the two displayed congruences gives
\[
\frac{N(\ell^2N-Q)}{b^2}\equiv0\pmod4,
\]
The second gives
\[
\frac{N(\ell^2N-Q)}{b^2}\equiv N^2\equiv1\pmod4.
\]
When $2\Vert d$, the first congruence occurs precisely for $\ell$ odd and the second for $\ell$ even.
When $4\mid d$, both occur.
This proves both the stated equivalence defining $\mathcal R_{r,d}(Q)$ and its parity description.
The argument uses only that $Q$ is odd and
$(Q,rd)=1$, so it is independent of $k$.
\end{proof}

We call $(r,d)$ admissible if $P\nmid rd$ and
$\mathcal R_{r,d}(Q)$ is nonempty.

We shall use Hooley's estimate for squarefree integers in arithmetic progressions in a form that retains the condition $P\nmid N$.
Set
\[
\kappa(m)=\prod_{q\mid m}\frac{q}{q+1},
\qquad
\rho_P(m)=
\begin{cases}
1,&P\mid m,\\[2mm]
\dfrac{P}{P+1},&P\nmid m.
\end{cases}
\]

\begin{lemma}\label{lem:prime-deleted-hooley}
If $(a,m)=1$, then
\[
\sum_{\substack{N\in[X,X+Y]\\N\equiv a\;(\bmod m)\\P\nmid N}}
\mu^2(N)
=
\frac{Y}{\zeta(2)}
\frac{\kappa(m)}{\varphi(m)}\rho_P(m)
+
O_\varepsilon\left(
X^\varepsilon
\left(
\sqrt{\frac Xm}+m^{1/2+\varepsilon}
\right)
\right).
\]
\end{lemma}

\begin{proof}
When $P\mid m$, the congruence itself implies $P\nmid N$, and this is Hooley's estimate~\cite{Hooley1975}.  Suppose that $P\nmid m$.
The identity
\[
\mu^2(N)\mathbf 1_{P\nmid N}
=
\sum_{j\geq0}(-1)^j\mathbf 1_{P^j\mid N}\,
\mu^2\left(\frac{N}{P^j}\right)
\]
reduces the left-hand side to a finite alternating sum of unrestricted squarefree progression sums, in intervals of length $Y/P^j$ and with residue class $aP^{-j}\bmod m$.
The main terms form the geometric series
\[
\sum_{j\geq0}\frac{(-1)^j}{P^j}=\frac{P}{P+1}.
\]
Replacing the finite series by this infinite series costs $O(1)$, which is absorbed by the displayed error.
The square-root errors form a convergent geometric series, while the remaining $O(m^{1/2+\varepsilon})$-errors acquire at most a logarithmic factor, which is absorbed by $X^\varepsilon$.
\end{proof}

Put
\[
BC(r)=B\,C(r).
\]
For
\[
1\leq r<\frac{2P^{k/2}}{\sqrt{X+Y}},
\]
define
\[
\mathcal H_k(r)
=
\sideset{}{'}\sum_{\substack{N\in[X,X+Y]\\P\nmid N}}
H_1(r^2N^2-4P^kN).
\]

For an admissible pair $(r,d)$, define
\[
\mathcal S_{d,n,r}
=
\sum_{\substack{N\in[X,X+Y]\\
N\bmod d^2\in\mathcal R_{r,d}(Q)\\P\nmid N}}
\mu^2(N)
\left(\frac{N}{n}\right)
\left(\frac{(r^2N-4Q)/d^2}{n}\right),
\]
and set $\mathcal S_{d,n,r}=0$ for non-admissible pairs.
Here and below all symbols in parentheses are Kronecker symbols.

Let
\[
f_n=
\begin{cases}
1,&n\ {\rm odd},\\
4,&n\ {\rm even},
\end{cases}
\qquad
\widehat n=f_nn,
\qquad
\widehat g=(d^\infty,\widehat n),
\qquad
\widehat n_0=\frac{\widehat n}{\widehat g}.
\]
Here
\[
(d^\infty,m)=\prod_{q\mid d}q^{v_q(m)}
\]
denotes the largest divisor of $m$ supported on the primes dividing $d$.
Thus neither $\widehat g$ nor $\widehat n_0$ has $2$-adic valuation $1$ or $2$.
For $m$ with this restriction, introduce the local character sum
\[
\theta_{r,Q}(m)
=
\sum_{a\bmod m}
\left(\frac{a}{m}\right)
\left(\frac{r^2a-4Q}{m}\right)
\]
and, for $g\mid d^\infty$ with the same restriction, put
\[
\varphi^{\,o}_{r,d,Q}(g)
=
\sum_{\substack{a\bmod d^2g\\
a\bmod d^2\in\mathcal R_{r,d}(Q)}}
\left(\frac{a}{g}\right)
\left(\frac{(r^2a-4Q)/d^2}{g}\right).
\]
Finally, put
\[
\Gamma_{r,Q}(d,n)
=
\rho_P(f_nd^2n)
\frac{\kappa(f_nd^2n)}{\varphi(f_nd^2n)}
\sum_{\substack{a\bmod f_nd^2n\\
a\bmod d^2\in\mathcal R_{r,d}(Q)}}
\left(\frac a n\right)
\left(\frac{(r^2a-4Q)/d^2}{n}\right).
\]

\begin{lemma}\label{lem:local-arithmetic}
Let $D,M\geq1$ and $1\leq r<2\sqrt Q/\sqrt{X+Y}$ with $P\nmid r$.
The series
\[
\sum_{\substack{d,n\geq1\\ (r,d)\ {\rm admissible}}}
\frac{\Gamma_{r,Q}(d,n)}{nd}
\]
is absolutely convergent, and
\begin{equation}\label{eq:local-series}
\sum_{\substack{d\leq D,\ n\leq M\\ (r,d)\ {\rm admissible}}}
\frac{\Gamma_{r,Q}(d,n)}{nd}
=BC(r)+O(P^{-2})
+O_\varepsilon\left(
(QDM)^\varepsilon\left(D^{-2}+M^{-1/5}\right)
\right).
\end{equation}
The constants and tail bounds are uniform in the displayed range of $r$; $k$ remains fixed.
\end{lemma}

\begin{proof}
We give the local comparison, since this is where replacing $P$ or $P^2$ by $P^k$ has to be justified.
Applying the Chinese remainder theorem to the safe modulus $\widehat n=f_nn$ gives
\begin{equation}\label{eq:residue-factorization}
\sum_{\substack{a\bmod f_nd^2n\\
a\bmod d^2\in\mathcal R_{r,d}(Q)}}
\left(\frac a n\right)
\left(\frac{(r^2a-4Q)/d^2}{n}\right)
=
\varphi^{\,o}_{r,d,Q}(\widehat g)
\theta_{r,Q}(\widehat n_0).
\end{equation}
Multiplying an even $n$ by $4$ preserves the product of the two Kronecker symbols on the nonzero residue classes and supplies the correct $2$-adic conductor.
More explicitly, if $g=(d^\infty,n)$ and $n_0=n/g$, then for odd $n$ one has $(\widehat g,\widehat n_0)=(g,n_0)$.
If $n$ and $d$ are even, then
$(\widehat g,\widehat n_0)=(4g,n_0)$; if $n$ is even and $d$ is odd, then $(\widehat g,\widehat n_0)=(g,4n_0)$.
These are the three CRT factorizations represented at once by~\eqref{eq:residue-factorization}; the factor $d^{-2}$ is a square modulo $\widehat n_0$ and leaves the second character in $\theta_{r,Q}$ unchanged.

For an odd prime $q\neq P$, direct reduction modulo $q$ gives
\[
\theta_{r,Q}(q^\alpha)
=
\begin{cases}
-q^{\alpha-1},&q\nmid r,\quad \alpha\ {\rm odd},\\
q^{\alpha-1}(q-2),&q\nmid r,\quad \alpha\ {\rm even},\\
0,&q\mid r,\quad \alpha\ {\rm odd},\\
q^{\alpha-1}(q-1),&q\mid r,\quad \alpha\ {\rm even}.
\end{cases}
\]
Indeed, when $q\nmid r$, the substitution $a=4Qr^{-2}u$ reduces the product of the two characters to the one attached to $u(u-1)$, up to a square.
When $q\mid r$, the odd powers have zero mean and the even powers are principal.
At $q=2$ and $\alpha\geq3$,
\[
\theta_{r,Q}(2^\alpha)
=
\begin{cases}
(-1)^\alpha2^{\alpha-1},&r\ {\rm odd},\\
0,&r\ {\rm even},
\end{cases}
\]
because $Q$ is odd and $4Qr^{-2}\equiv4\pmod8$ when $r$ is odd.
The exponents $v_2(n)=1,2$ are incorporated by the factor $f_n$ in~\eqref{eq:residue-factorization}.
The admissible classes in Lemma~\ref{lem:admissible-classes} also give
\[
\varphi^{\,o}_{r,d,Q}(g)
=
\begin{cases}
\varphi(g)\mathbf1_{g=\square},&2\nmid d,\\
\varphi(g)\mathbf1_{g=\square},&2\Vert d,\quad 2\nmid g,\\
0,&2\Vert d,\quad 2\mid g,\quad 2\Vert r,\\
2\varphi(g)\mathbf1_{g=\square},
 &2\Vert d,\quad 2\mid g,\quad 4\mid r,\\
2\varphi(g)\mathbf1_{g=\square},&4\mid d.
\end{cases}
\]
These are exactly the local sums evaluated in
\cite[Lemmas~3.8 and~3.9]{Zubrilina2025}, and the displayed computations show that their use there requires only that the parameter in place of $P$ be odd and coprime to $rd$.

It remains to inspect the prime $P$.  Lemma~\ref{lem:p-coprime} implies $P\nmid d r$, and hence
\[
\theta_{r,Q}(P^\alpha)=\varphi(P^\alpha)
\qquad(\alpha\geq1).
\]
The $P$-part of the series is therefore absolutely convergent and its local factor, including $\rho_P(f_nd^2n)$, is
\[
\frac{P}{P+1}
+
\sum_{\alpha\geq1}
\frac{P}{P+1}\frac1{P^\alpha}
=
\frac{P^2}{P^2-1}.
\]
For exponent $0$, the factor $P/(P+1)$ comes from $\rho_P$; for exponent $\alpha\geq1$, one has $\rho_P=1$, and the same factor comes from $\kappa(P^\alpha)$.
Since $P\nmid r$, the corresponding generic local factor is
\[
B_P=\frac{P^4-2P^2-P+1}{(P^2-1)^2},
\]
and
\[
\frac{P^2}{P^2-1}-B_P
=\frac{P^2+P-1}{(P^2-1)^2}
=O(P^{-2}).
\]
At every other prime the Euler factors are those computed in~\cite[Lemmas~3.8--3.10]{Zubrilina2025}.
Since $P\nmid r$, their product is
\[
\frac{BC(r)}{B_P}.
\]
The full product is therefore
\[
BC(r)\frac{P^2/(P^2-1)}{B_P}=BC(r)+O(P^{-2}).
\]
The product away from $P$ converges absolutely and uniformly.
The product defining $C(r)$ is uniformly bounded over $r$, so the error is uniform in the displayed range.
For the tails, write $\widehat n=mg$, where
\[
g=(d^\infty,\widehat n),
\qquad (m,d)=1,
\qquad g\mid d^\infty,
\qquad mg=\widehat n\leq4M.
\]
The Euler-factor majorant in~\cite[Lemma~3.10]{Zubrilina2025} applies to the factors away from $P$ and gives the uniform tails
$O_\varepsilon((QDM)^\varepsilon(D^{-2}+M^{-1/5}))$.
The isolated $P$-factor is uniformly bounded and has already been accounted for by the $O(P^{-2})$ term.
If $D=1$ or $M=1$, the estimate follows after enlarging the absolute constant because $D^{-2}+M^{-1/5}\geq1$.
This proves~\eqref{eq:local-series} in the stated range.
\end{proof}

\begin{lemma}\label{lem:small-n}
For $0\leq\sigma,\tau\leq1$,
\begin{equation*}
\begin{aligned}
\sum_{\substack{n\leq Y^\sigma\\d\leq Y^\tau}}
\frac{\mathcal S_{d,n,r}}{nd}
&=
\frac{Y}{\zeta(2)}BC(r)
+O\left(\frac{Y}{P^2}\right)\\
&\quad+
O_\varepsilon\left(
(QXY)^\varepsilon
\left(
\sqrt X\,Y^{\sigma/2}
+Y^{\tau+3\sigma/2}
+Y^{1-2\tau}
+Y^{1-\sigma/5}
\right)
\right).
\end{aligned}
\end{equation*}
\end{lemma}

\begin{proof}
For a fixed admissible $d$, split the $N$-sum into classes modulo $f_nd^2n$.
Only classes coprime to this modulus contribute.  Admissibility gives $(N,d)=1$; either Kronecker symbol vanishes when the corresponding class is not coprime to $n$; and, when $f_n=4$, the first symbol also forces the class to be odd.
Applying Lemma~\ref{lem:prime-deleted-hooley} to the remaining classes and then using~\eqref{eq:residue-factorization} gives
\begin{equation}\label{eq:S-small}
\mathcal S_{d,n,r}
=
\frac{Y}{\zeta(2)}
\Gamma_{r,Q}(d,n)
+
O_\varepsilon\left(
(QXn)^\varepsilon
\left(
\frac{\sqrt{Xn}}d+d\,n^{3/2}
\right)
\right).
\end{equation}
The number of admissible classes is at most two, so it is absorbed in the implied constant.  Divide by $nd$ and sum.
The two errors in~\eqref{eq:S-small} contribute, respectively,
\[
\ll_\varepsilon
(QXY)^\varepsilon\sqrt X\,Y^{\sigma/2}
\quad\text{and}\quad
\ll_\varepsilon
(QXY)^\varepsilon Y^{\tau+3\sigma/2}.
\]
Applying Lemma~\ref{lem:local-arithmetic} with $D=Y^\tau$ and $M=Y^\sigma$ completes the proof.
\end{proof}

\begin{lemma}[Prime-deleted completed character sum]
\label{lem:prime-deleted-character-sum}
Let $X\geq2$ and $1\leq Y\leq X$.  Let $m,w,D\geq1$, and let $q,t\in\mathbb Z$ satisfy $(t,D)=1$ and $wt\equiv q\pmod D$.
Let $P$ be a prime with $P\nmid D$, and put $\chi_m(\cdot)=(\frac{\cdot}{m})$.
Define
\[
\mathcal P(m;w,D,q)
=
\{\ell\mid m:\ \ell\text{ odd},\ v_\ell(m)\text{ odd},\
\ell\nmid q,\ (D,w,\ell)=1\}
\]
and
\[
\mathfrak m
=
\frac{m}{\displaystyle
\prod_{\ell\in\mathcal P(m;w,D,q)}(\sqrt\ell/2)}.
\]
Then
\begin{equation}
\begin{aligned}
\sum_{\substack{M\in[X,X+Y]\\M\equiv t\ (\bmod D)\\P\nmid M}}
\mu^2(M)\chi_m(M)
\chi_m\left(\frac{wM-q}{D}\right)
&\ll
\sqrt X+\frac{\mathfrak mY}{mD}
+\frac{\log(2X)\sqrt{\mathfrak mY}}{\sqrt D},
\end{aligned}
\label{eq:prime-deleted-character-sum}
\end{equation}
with an absolute implied constant.
\end{lemma}

\begin{proof}
The unrestricted version of this sum is covered by
\cite[Lemma~3.19]{Zubrilina2025}.
If $P\mid m$, every term with $P\mid M$ already vanishes because $\chi_m(M)=0$, so there is
nothing to prove.
Suppose $P\nmid m$.
Use the pointwise identity
\[
\mu^2(M)\mathbf1_{P\nmid M}
=
\sum_{j\geq0}(-1)^j\mathbf1_{P^j\mid M}\mu^2(M/P^j).
\]
The sum is finite.  In its $j$-th term write $M=P^jM_0$.
Apart from the harmless constant $\chi_m(P)^j$, the resulting sum has the same form with
\[
w_j=wP^j,
\qquad
t_j\equiv tP^{-j}\pmod D.
\]
Moreover, $M_0$ lies in an interval with initial point $X/P^j$ and length $Y/P^j$, and
\[
w_jt_j\equiv q\pmod D.
\]
Separate the indices satisfying $X/P^j<2$ or $Y/P^j<1$.
For each such index, the scaled interval has length below $2$ or an initial point below $2$ together with length at most that initial point, and hence contains $O(1)$ integers.
There are $O(\log X)$ indices, so their total contribution is $O(\log X)$ and is absorbed by the $\sqrt X$ term in the claimed bound.
For every remaining index, $X/P^j\geq2$ and $1\leq Y/P^j\leq X/P^j$, so \cite[Lemma~3.19]{Zubrilina2025} applies to the scaled sum.
Since $P\nmid mD$, the set of saving primes and hence $\mathfrak m$ are unchanged.
Its absolute value is at most a constant times
\[
\sqrt{X/P^j}
+\frac{\mathfrak mY}{mDP^j}
+\frac{\log(2X)\sqrt{\mathfrak mY/P^j}}{\sqrt D}.
\]
Summing the three convergent geometric series proves
\eqref{eq:prime-deleted-character-sum}.
\end{proof}

\begin{lemma}\label{lem:large-n}
Let $T\geq Y^\sigma$.  For $0\leq\sigma,\tau\leq1$,
\[
\sum_{\substack{Y^\sigma<n\leq T\\d\leq Y^\tau}}
\frac{\mathcal S_{d,n,r}}{nd}
\ll_\varepsilon
(QXT)^\varepsilon
\left(
\sqrt X+Y^{1-\sigma/2}+\sqrt Y\,T^{1/4}
\right).
\]
\end{lemma}

\begin{proof}
For an integer $n$, let
\[
\mathcal P(n)
=
\{q\mid n:q\ {\rm odd},\ v_q(n)\ {\rm odd},\ q\neq P\},
\qquad
\mathfrak n
=
\frac{n}{\displaystyle\prod_{q\in\mathcal P(n)}(\sqrt q/2)}.
\]
For each admissible residue class in the definition of
$\mathcal S_{d,n,r}$, apply Lemma~\ref{lem:prime-deleted-character-sum} with
\[
m=n,
\qquad
w=r^2,
\qquad
D=d^2,
\qquad
q=4Q.
\]
The congruence and coprimality hypotheses of that lemma are exactly the defining conditions on an admissible class.
By admissibility, $P\nmid rd$ and no odd prime divides both $d$ and $r$.
Since $Q=P^k$, the saving-prime set in the lemma is precisely $\mathcal P(n)$.
There are at most two admissible classes, and hence
\begin{equation}\label{eq:large-n-pointwise}
\mathcal S_{d,n,r}
\ll_\varepsilon
(QXn)^\varepsilon
\left(
\sqrt X+\frac{\mathfrak nY}{nd^2}
+\frac{\sqrt{\mathfrak nY}}d
\right).
\end{equation}

Write uniquely
\[
n=a^2b\,2^\alpha P^\beta,
\qquad
b\ {\rm square\mbox{-}free},
\qquad
(ab,2P)=1.
\]
The divisor bound gives
\[
\mathfrak n
\ll_\varepsilon
a^2b^{1/2+\varepsilon}2^\alpha P^\beta.
\]
Consequently,
\[
\sum_{n>Y^\sigma}\frac{\mathfrak n}{n^2}
\ll_\varepsilon Y^{-\sigma/2+\varepsilon},
\qquad
\sum_{n\leq T}\frac{\sqrt{\mathfrak n}}n
\ll_\varepsilon T^{1/4+\varepsilon}.
\]
Indeed, with $U=Y^\sigma$, the two left-hand sides are bounded by
\[
\sum_{\alpha,\beta,a}
\frac1{2^\alpha P^\beta a^2}
\sum_{b>U/(2^\alpha P^\beta a^2)}
\frac1{b^{3/2-\varepsilon}}
\]
and
\[
\sum_{\substack{\alpha,\beta,a\\
2^\alpha P^\beta a^2\leq T}}
\frac1{2^{\alpha/2}P^{\beta/2}a}
\sum_{b\leq T/(2^\alpha P^\beta a^2)}
\frac1{b^{3/4-\varepsilon}},
\]
respectively.
Splitting the first sum according as
$2^\alpha P^\beta a^2\leq U$ or $>U$, and evaluating the inner $b$-sums, gives the two claimed estimates.
Dividing~\eqref{eq:large-n-pointwise} by $nd$, summing over $d$ and $n$, and absorbing logarithms in $(QXT)^\varepsilon$ proves the
lemma.
\end{proof}

\begin{proposition}\label{prop:hyperbolic-one-r}
Uniformly in this range of $r$, one has
\begin{multline*}
\mathcal H_k(r)
=
\frac{Y\sqrt{4P^kX-r^2X^2}}{\zeta(2)\pi}BC(r)
\\
+
O_{k,\varepsilon}\left(
(P^kXY)^\varepsilon
\left(
(YP^kX)^{3/5}
+
\frac{Y^2P^{k/2}}{\sqrt X}
+
r\sqrt X\,Y^{3/2}
+
P^{k/2}XY^{5/18}
\right.\right.
\\ \left.\left.
+
P^{k/2}\sqrt X\,Y^{8/9}
\right)
\right)
+
O_k\left(P^{k/2-2}\sqrt X\,Y\right).
\end{multline*}
\end{proposition}

\begin{proof}
Put
\[
\Delta_r=4QX-r^2X^2.
\]
For $N$ in the interval, the class number formula and
Lemma~\ref{lem:admissible-classes} give
\begin{equation}\label{eq:class-number-double-sum}
H_1(r^2N^2-4QN)
=
\frac1\pi
\sum_{\substack{d^2\mid r^2N-4Q\\
N\bmod d^2\in\mathcal R_{r,d}(Q)}}
\frac{\sqrt{4QN-r^2N^2}}{d}
L\left(1,\chi_{\frac{N(r^2N-4Q)}{d^2}}\right)
+
O(1).
\end{equation}
To justify that this accounts for every square divisor in the Hurwitz class number, note that $N$ is squarefree and $P\nmid N$.
A square factor which uses a prime from both $N$ and $r^2N-4Q$ can therefore occur only at $2$.
If $N=2N_0$, with $N_0$ odd, removal of this extra factor $4$ would leave
\[
\frac{r^2N_0^2-2QN_0}{d^2}.
\]
Here $d$ is odd: the odd parts of $N_0$ and
$r^2N_0-2Q$ are coprime, and the latter has $2$-adic valuation at most $1$.
This is $2\pmod4$ when $r$ is even and $3\pmod4$ when $r$ is odd, so its Gauss class number is zero.
Thus only the divisors displayed in~\eqref{eq:class-number-double-sum} contribute.
The discriminant condition is encoded by Lemma~\ref{lem:admissible-classes}.
Moreover $d<2\sqrt Q$, and Lemma~\ref{lem:p-coprime} gives $P\nmid rd$.
The $O(1)$-term in~\eqref{eq:class-number-double-sum}, summed over $N$, is $O(Y)$.
Since the present range is nonempty only when $Q>(X+Y)/4$, this is absorbed by $(QXY)^{3/5}$.

Let $0<\tau<1/2$.
The bound
\[
L(1,\chi)\ll_\varepsilon(QX)^\varepsilon
\]
and the fact that each $d$ gives at most two residue classes show that the part of~\eqref{eq:class-number-double-sum} with $d>Y^\tau$, summed over $N$, is
\begin{align}
\ll_\varepsilon
(QX)^{1/2+\varepsilon}
\sum_{Y^\tau<d<2\sqrt Q}
\left(\frac{Y}{d^2}+1\right)\frac1d
\ll_\varepsilon
(QX)^{1/2+\varepsilon}
\left(Y^{1-2\tau}+\log(2Q)\right).
\label{eq:large-d}
\end{align}

For $d\leq Y^\tau$, we freeze the square-root factor at $N=X$.
Uniformly in the interior range,
\[
\left|
\sqrt{4QN-r^2N^2}-\sqrt{\Delta_r}
\right|
\ll
\frac{\sqrt Q\,Y}{\sqrt X}
+r\sqrt{XY}.
\]
After summing over $N$ and $d$, the resulting error is
\begin{equation}\label{eq:square-root-variation}
O_\varepsilon\left(
(QXY)^\varepsilon
\left(
\frac{\sqrt Q\,Y^2}{\sqrt X}
+r\sqrt X\,Y^{3/2}
\right)
\right).
\end{equation}

For a quadratic character of discriminant $D$, partial summation together with P\'olya--Vinogradov gives
\[
L(1,\chi_D)
=
\sum_{n\leq T}\frac{\chi_D(n)}n
+
O\left(\frac{\sqrt{|D|}\log(2|D|)}T\right).
\]
Since $|D|\ll QX/d^2$, inserting this in the remaining part of~\eqref{eq:class-number-double-sum} produces the truncation error
\begin{equation}\label{eq:L-truncation-error}
O_\varepsilon\left(
(QXY)^\varepsilon\frac{QXY}{T}
\right).
\end{equation}
The identity
\[
\left(\frac{N(r^2N-4Q)/d^2}{n}\right)
=
\left(\frac Nn\right)
\left(\frac{(r^2N-4Q)/d^2}{n}\right)
\]
now yields
\begin{multline}\label{eq:hyperbolic-master}
\mathcal H_k(r)
=
\frac{\sqrt{\Delta_r}}{\pi}
\sum_{\substack{d\leq Y^\tau\\n\leq T}}
\frac{\mathcal S_{d,n,r}}{nd}
\\
+
O_\varepsilon\left(
(QXY)^\varepsilon
\left(
\frac{QXY}{T}
+\frac{\sqrt Q\,Y^2}{\sqrt X}
+r\sqrt X\,Y^{3/2}
+\sqrt{QX}\,Y^{1-2\tau}
\right)
\right),
\end{multline}
where the logarithm in~\eqref{eq:large-d} has been absorbed into the last term.
Indeed, $\log(2Q)\ll\log X=o(Y^{1-2\tau})$ because $\tau<1/2$ and the short-window hypotheses make $Y$ a fixed positive power of $X$.

Choose
\[
\tau=\frac1{18},
\qquad
\sigma=\frac59.
\]
Lemmas~\ref{lem:small-n} and~\ref{lem:large-n} give
\begin{equation}\label{eq:arithmetic-double-sum}
\sum_{\substack{d\leq Y^\tau\\n\leq T}}
\frac{\mathcal S_{d,n,r}}{nd}
=
\frac{Y}{\zeta(2)}BC(r)
+
O\left(\frac{Y}{P^2}\right)
+
O_\varepsilon\left(
(QXT)^\varepsilon
\left(
\sqrt X\,Y^{5/18}
+Y^{8/9}
+\sqrt Y\,T^{1/4}
\right)
\right).
\end{equation}
The range of $r$ can be nonempty only if $4Q>X+Y$.
Hence, with
\[
T=(YQX)^{2/5},
\]
\[
\frac TY
=
\frac{(YQX)^{2/5}}Y
\gg X^{(1+3\delta_1)/5+o(1)}.
\]
Thus, in particular, $T>Y$, and Lemma~\ref{lem:large-n} applies.
Finally,
\[
\frac{QXY}{T}
\asymp
\sqrt{QXY}\,T^{1/4}
\asymp
(QXY)^{3/5}.
\]
Substituting~\eqref{eq:arithmetic-double-sum} into \eqref{eq:hyperbolic-master}, using $\sqrt{\Delta_r}\leq\sqrt{QX}$, and recalling $Q=P^k$, gives exactly the error stated in the
proposition.
\end{proof}

Summing Proposition~\ref{prop:hyperbolic-one-r} gives the following.

\begin{corollary}\label{cor:hyperbolic-interior}
Let
\[
R_0=\frac{2P^{k/2}}{\sqrt{X+Y}}.
\]
Then
\begin{equation*}
\frac{\zeta(2)\pi}{XY}
\sum_{1\leq r<R_0}\mathcal H_k(r)
=
B\sum_{1\leq r<R_0}
C(r)\sqrt{\frac{4P^k}{X}-r^2}
+
O_{k,\varepsilon}\left((P^kXY)^\varepsilon\mathcal E_{\mathrm{hyp}}\right)
+O_k\left(\frac{P^{k-2}}X\right),
\end{equation*}
where
\begin{equation*}
\mathcal E_{\mathrm{hyp}}
=
\frac{P^{11k/10}}{Y^{2/5}X^{9/10}}
+
\frac{YP^k}{X^2}
+
\frac{P^kY^{1/2}}{X^{3/2}}
+
\frac{P^k}{X^{1/2}Y^{13/18}}
+
\frac{P^k}{XY^{1/9}}.
\end{equation*}
\end{corollary}

\begin{proof}
There are $O(P^{k/2}/\sqrt X)$ values of $r$, and
\[
\sum_{r\leq 2P^{k/2}/\sqrt X}r
\ll\frac{P^k}{X}.
\]
The assertion follows by summing each error term in Proposition~\ref{prop:hyperbolic-one-r}.
\end{proof}

We finally pass from the common interior range to the full range in the trace formula.

\begin{lemma}\label{lem:hyperbolic-boundary}
Let
\[
R_1=\frac{2P^{k/2}}{\sqrt X}.
\]
The contribution of
\[
R_0\leq r\leq\frac{2P^{k/2}}{\sqrt N},
\qquad
N\in[X,X+Y],
\]
together with the extension of the main term in Corollary~\ref{cor:hyperbolic-interior} from $R_0$ to $R_1$, after
multiplication by $\zeta(2)\pi/(XY)$, is
\[
O_{k,\varepsilon}\left(
(P^kX)^\varepsilon
\left(
\frac{P^kY^{3/2}}{X^{5/2}}
+
\frac{P^{k/2}Y^{1/2}}{X}
\right)
\right).
\]
\end{lemma}

\begin{proof}
In the boundary range,
\[
4P^kN-r^2N^2\ll P^kY.
\]
The bound $H_1(-D)\ll_\varepsilon D^{1/2+\varepsilon}$, together with
\[
R_1-R_0\ll\frac{P^{k/2}Y}{X^{3/2}},
\]
gives the stated estimate, and the additional $1$ in the number of integers in the interval gives the second term.
The same argument, using
\[
\sqrt{\frac{4P^k}{X}-r^2}
\ll\frac{P^{k/2}\sqrt Y}{X},
\]
handles the corresponding extension of the main term.
\end{proof}

\subsection{Completion of the proof}

We first record the remaining terms in Proposition~\ref{prop:trace-formula}.
If $P\mid N$, then the local representation is Steinberg and
\[
P^{k/2}\lambda_{\sym^k f}(P)=a_f(P)^k,
\qquad
a_f(P)\in\{\pm1\}.
\]
Since $\dim S_{\mathrm{new}}(N)\ll N$, the contribution of these levels, after multiplication by $\zeta(2)\pi/(XY)$, is
\begin{equation}\label{eq:ramified-levels}
O_k\left(\frac1P+\frac1Y\right).
\end{equation}

Moreover,
\begin{equation}\label{eq:elementary-term}
\frac{\zeta(2)\pi}{XY}
\sideset{}{'}\sum_{\substack{N\in[X,X+Y]\\P\nmid N}}
\sum_{j=0}^kP^j
=
\pi y
+
O_k\left(
\frac{y}{P}
+
\frac{P^k}{Y\sqrt X}
+
\frac1X
\right).
\end{equation}
Indeed, this follows from
\[
\sum_{N\leq Z}\mu^2(N)=\frac{Z}{\zeta(2)}+O(\sqrt Z)
\]
and from deleting the multiples of $P$.

We now verify the errors.
The five terms in $\mathcal E_{\mathrm{hyp}}$ have respective exponents
\[
-\frac15+\frac{11}{10}\delta_2+\frac25\delta_1,\quad
\delta_2-\delta_1,\quad
\delta_2-\frac{\delta_1}{2},\quad
\delta_2+\frac{13}{18}\delta_1-\frac29,\quad
\delta_2+\frac19\delta_1-\frac19.
\]
Under
\[
0<2\delta_2<\delta_1<\frac18,
\]
adding $\eta$ to these five exponents gives, respectively,
\[
-\frac15+\frac1{10}\delta_2+\frac9{10}\delta_1,\quad
-\frac{\delta_1}{2},\quad
0,\quad
-\frac29+\frac{11}{9}\delta_1,\quad
-\frac19+\frac{11}{18}\delta_1.
\]
The first expression is negative because $\delta_2<\delta_1/2$ and $\delta_1<1/8$; the final two are negative because $\delta_1<1/8$.
Thus each original exponent is at most $-\eta$, with equality only for the third one.
The errors in Lemma~\ref{lem:hyperbolic-boundary} have exponents at most
\[
\delta_2-\frac32\delta_1
\qquad\text{and}\qquad
\frac{\delta_2-\delta_1}{2},
\]
and adding $\eta$ gives $-\delta_1$ and $-\delta_2/2$, respectively.
They are therefore smaller.

We spell out the parabolic errors as well.  In the even case, the exponent of the second error inside the parentheses in Proposition~\ref{prop:parabolic-average} is at most
\[
-\frac14+\frac1{6k}
+\left(\frac12+\frac1{6k}\right)\delta_2
+\frac56\delta_1.
\]
After adding $\eta$, this is at most
\[
-\frac16+\frac43\delta_1<0,
\]
where we used $k\geq2$ and discarded a negative multiple of $\delta_2$.
In the odd case $k\geq3$, the corresponding exponent, after adding $\eta$, is at most
\[
-\frac6{19}+\frac{51}{38}\delta_1<0.
\]
The remaining analytic parabolic error has exponent $\delta_2/2-\delta_1$, which is also smaller than $-\eta$.
The first parabolic error inside the parentheses is
\[
(PXY)^\varepsilon\frac{P^{k/2-1}}{\sqrt X}
=(PXY)^\varepsilon\frac{\sqrt y}{P}.
\]
The growth conditions give $PXY\ll X^C$ for a fixed $C\geq1$.
Applying the proposition with $\varepsilon/C$ and relabeling the exponent bounds this term by
$O_{k,\varepsilon}(X^\varepsilon\sqrt y/P)$.
Finally,
\[
\frac{P^k}{Y\sqrt X}
\ll X^{-1/2+\delta_1+\delta_2}
\ll X^{-\eta}.
\]
Also $X^{-1}$ and $Y^{-1}$ are $O(X^{-\eta})$, since $\eta<\delta_1/2<1/16$.
The exceptional hyperbolic local term is $P^{k-2}/X=y/P^2$.
Together with $P^{-1}$ and $y/P$, it is covered by $(1+y)/P$.
Consequently, Proposition~\ref{prop:trace-formula}, Proposition~\ref{prop:parabolic-average}, Corollary~\ref{cor:hyperbolic-interior}, Lemma~\ref{lem:hyperbolic-boundary}, \eqref{eq:ramified-levels}, and \eqref{eq:elementary-term} give
\begin{multline}\label{eq:numerator-asymptotic}
\frac{\zeta(2)\pi}{XY}
\sideset{}{'}\sum_{N\in[X,X+Y]}
\sum_{f\in H_{\mathrm{new}}(N)}
P^{k/2}\lambda_{\sym^k f}(P)\varepsilon(f)
\\
=
A\sqrt y
+
B\sum_{1\leq r\leq2\sqrt y}
C(r)\sqrt{4y-r^2}
-\pi y
+
O_{k,\varepsilon}\left(
X^{-\eta+\varepsilon}
+X^\varepsilon\frac{\sqrt y}{P}
+\frac{1+y}{P}
\right).
\end{multline}

It remains to divide by the size of the family.
By the dimension estimate used in~\cite[\S3.4]{Zubrilina2025},
\begin{equation*}
\sideset{}{'}\sum_{N\in[X,X+Y]}
\sum_{f\in H_{\mathrm{new}}(N)}1
=
\frac{XY}{12\zeta(2)}D_0
+
O_\varepsilon\left(
YX^\varepsilon+X^{8/5+\varepsilon}+Y^2
\right).
\end{equation*}
Since $C(r)\ll_\varepsilon r^\varepsilon$, the main term in~\eqref{eq:numerator-asymptotic} is $O_\varepsilon(1+y^{1+\varepsilon})$.
The three relative errors in the dimension estimate, after multiplication by this quantity, have exponents at most
\[
-1+\delta_2,\qquad
-\frac25+\delta_1+\delta_2,\qquad
-\delta_1+\delta_2.
\]
They are all smaller than $-\eta$: for the middle one this follows from $3\delta_1/2<3/16<2/5$, and the other two are immediate.
Dividing~\eqref{eq:numerator-asymptotic} by the dimension estimate proves Theorem~\ref{thm:scaled-murmuration}.
Finally, the additional hypothesis $(k-1)\delta_2<1$ gives
\[
\frac yP=\frac{P^{k-1}}X
\ll X^{-c_k},
\qquad
c_k=\frac{1-(k-1)\delta_2}{k}>0.
\]
The growth condition for $P^k$ also gives
\[
\frac{\sqrt y}{P}
\ll X^{-d_k},
\qquad
d_k=\frac{2-(k-2)\delta_2}{2k}>0.
\]
Together with $P\to\infty$, these estimates show, after choosing $0<\varepsilon<\min\{\eta,c_k,d_k\}$, that the error term in the theorem tends to zero.



\section{Data and experimental design}
\label{secDesign}

\subsection{Elliptic curve data}

We use a frozen snapshot of $73\,069$ records constructed from the class-level ecdata tables~\cite{CremonaEcdata}.  Each record contains a conductor, an isogeny-class label, the rank recorded in those tables, a root number, and the first $1000$ prime coefficients.
The acquisition program reads one representative per isogeny class with conductor at most $100\,000$.
It samples $20\,000$ classes from each recorded rank $0$, $1$, and $2$, using seeds $2026$, $2027$, and $2028$, and takes the union with the two conductor-window populations used for Figure~\ref{figBasic}.
After duplicate classes are removed, this construction gives the $73\,069$-record snapshot and the binary indicators used for the model and plotting populations.
PARI supplies the root numbers and prime coefficients.
The supplied acquisition program uses the tagged ecdata release dated April 22, 2026.
The original acquisition run log was not retained, so the frozen snapshot remains the canonical input for every numerical value reported below.
The isogeny-class labels are unique in that snapshot, and the main experiment focuses on the three recorded ranks.

The unit of analysis is therefore an isogeny class representative.
The splitting code applies a target-stratified row assignment after the uniqueness audit.
Since each isogeny-class label occurs once, every class representative belongs to a single assignment part and isogeny leakage cannot occur.

The five prediction tasks are listed in Table~\ref{tabTasks}.
Each task is balanced before splitting.  In the rank-balanced population, the two root signs occur with counts $20\,000$ and $40\,000$.
The root task samples $20\,000$ from each sign.
Its target population is therefore a designed, class-balanced sample drawn from ranks $0$, $1$, and $2$ with conductors at most $100\,000$; natural database frequencies and higher ranks lie outside the scope of the reported accuracy.

\begin{table}[t]
\centering
\caption{Prediction tasks and test sizes.}
\label{tabTasks}
\begin{tabular}{lrrr}
\toprule
Task & Curves per class & Test size & Chance balanced accuracy\\
\midrule
Rank $0$ versus $1$ & 20000 & 8000 & $1/2$\\
Rank $0$ versus $2$ & 20000 & 8000 & $1/2$\\
Rank $1$ versus $2$ & 20000 & 8000 & $1/2$\\
Rank $0,1,2$ & 20000 & 12000 & $1/3$\\
Root number & 20000 & 8000 & $1/2$\\
\bottomrule
\end{tabular}
\end{table}

For a binary task, the data contain $40\,000$ curves.  A complete split has $25\,600$ training observations, $6\,400$ validation observations, and $8\,000$ test observations.
For the three-class task, the corresponding training, validation, and test counts are $38\,400$, $9\,600$, and $12\,000$.
Splits are stratified by the target label after the isogeny-class uniqueness check.
The same split is used for every feature panel in a comparison.

Root number agrees with rank parity throughout the model population, so the two labels encode overlapping empirical information in this data set.

Before fitting, we verify that the coefficient labels are exactly the first $1000$ primes in increasing order and that the largest one is $7919$.
On the first recorded curve and its first eight good primes, the normalized floating-point recurrence agrees with the scaled integer recurrence to the stored numerical tolerance.
The observed normalized values respect the good-prime bounds $|H_1|\leq2$, $|H_2|\leq3$, and $|H_3|\leq4$.
All model matrices contain finite values after imputation.

The finite elliptic-curve sample differs from the family in Theorem~\ref{thm:scaled-murmuration}.
The theorem averages all weight $2$ newforms of squarefree levels in a short level window.
The experiment uses elliptic-curve isogeny-class representatives with general conductors.
The theorem supplies structural motivation for the effective index, and Figure~\ref{figUniversal} provides a finite-data comparison whose scope excludes numerical verification of the asymptotic error term.

\subsection{Coefficient budget and missing values}

The first $1000$ primes end at $q=7919$.  The quadratic block uses the $23$ primes $p\leq83$, for which $p^2\leq q$.
The cubic block uses the $8$ primes $p\leq19$, for which $p^3\leq q$.
Table~\ref{tabBudget} gives the four main panels.

\begin{table}[t]
\centering
\caption{Feature panels under the common raw coefficient budget.}
\label{tabBudget}
\begin{tabular}{lrrrr}
\toprule
Panel & Raw queries & $H_2$ columns & $H_3$ columns & Total columns\\
\midrule
$X_1$ & 1000 & 0 & 0 & 1000\\
$X_1+H_2$ & 1000 & 23 & 0 & 1023\\
$X_1+H_3$ & 1000 & 0 & 8 & 1008\\
$X_1+H_2+H_3$ & 1000 & 23 & 8 & 1031\\
\bottomrule
\end{tabular}
\end{table}

At a prime of bad reduction we mark the corresponding normalized coordinate as missing.
The power coordinate is formed before imputation and has the same missing pattern.
The feature panels omit missing-value indicators.
The mean imputer and the standard scaler are fitted on the training set and then applied to validation and test sets.

This convention differs from the theoretical convention in Section~\ref{sec:scaled-murmuration}.
The theorem uses the coefficient of the ramified local symmetric power Euler factor.
The experiment marks that location as missing so that every observed power coordinate is governed by the displayed good prime Chebyshev formula.
Imputation then occurs only inside the training pipeline.

\subsection{The theorem and the finite experiment}

Table~\ref{tabTheoryExperiment} summarizes the distinct scopes of the arithmetic theorem and the empirical analysis.
The theorem identifies the effective coordinate and the common signed profile in a family that is accessible to the trace formula.
The experiment evaluates the resulting representation on a finite collection of rational elliptic curves.

\begin{table}[t]
\centering
\caption{The theoretical family and the finite data experiment.}
\label{tabTheoryExperiment}
\resizebox{\textwidth}{!}{
\begin{tabular}{lll}
\toprule
Aspect & Theorem~\ref{thm:scaled-murmuration} & Elliptic curve experiment\\
\midrule
Objects & All normalized weight $2$ newforms & Rational isogeny class representatives\\
Conductor condition & Squarefree levels & General elliptic curve conductors\\
Window & $Y$ asymptotically shorter than $X$ & Fixed windows for plots\\
Bad primes & Ramified local Euler factor & Missing value before imputation\\
Horizontal scale & $P^k/X$ with lower endpoint $X$ & $p^k/8750$ in the main plot\\
Conclusion & Root-number-weighted family mean & Held-out prediction and descriptive means\\
\bottomrule
\end{tabular}}
\end{table}

The theorem concerns asymptotic arithmetic averages, and the empirical analysis concerns finite-sample prediction accuracy.
The link between the two parts is the same effective index $p^k$ and the same root sign organization of the coefficient means.

\subsection{Model selection and controls}

The classifier is regularized logistic regression fitted by stochastic gradient descent.
The regularization parameter is selected on the validation set from
\begin{equation*}
\{3\mathord{\times}10^{-5},\,10^{-4},\,3\mathord{\times}10^{-4}\}.
\end{equation*}
Selection uses only the validation set.  The pipeline is then refitted on the union of the $25\,600$ fitting observations and the $6\,400$ validation observations.
For a binary task, preprocessing therefore uses $32\,000$ curves; the SGD classifier internally reserves $10$ percent of these rows for early stopping.
The refitted pipeline estimates imputation means and scales on all $32\,000$ rows before the classifier creates this internal holdout.
Section~\ref{secAdditionalControls} supplies a deterministic logistic fit with no internal early-stopping split.
We repeat the complete stratified assignment five times for the root-number task and the rank $0$ versus rank $1$ task.
All five assignments use the same initially sampled balanced population.
Their test sets overlap, and the split seed also sets the optimizer seed, so the observed spread combines assignment and optimization variation.

The fixed-split $95$ percent uncertainty interval uses the $2.5$th and $97.5$th percentiles of $1000$ class-stratified paired resamples of the common test observations.
It quantifies test-sampling variation conditional on the fitted models and chosen split; split-to-split and optimization variation lie outside its scope.

The initial analysis uses three controls.
A model using only the conductor checks whether sampling alone explains the result.
A label permutation checks the chance level.
A row-shuffled feature control preserves the size and marginal distribution of the added block after breaking its curve-level alignment.

The primary empirical comparison is root-number prediction with the full panel against $X_1$ on the same test curves.
Rank $0$ versus rank $1$ is the key secondary task.  The remaining rank tasks, individual blocks, quantity points, and the scan over $k=2,\ldots,12$ are exploratory.
The pointwise intervals are unadjusted for multiplicity across the exploratory comparisons.

\section{Experimental results}
\label{secResults}

\subsection{The ordinary prime murmuration}

We first reproduce the basic experiment of He, Lee, Oliver, and Pozdnyakov~\cite{HeLeeOliverPozdnyakov2025}.

\Needspace{0.3\textheight}

\medskip

\begin{figure}[H]
\centering
\includegraphics[width=0.96\textwidth]
{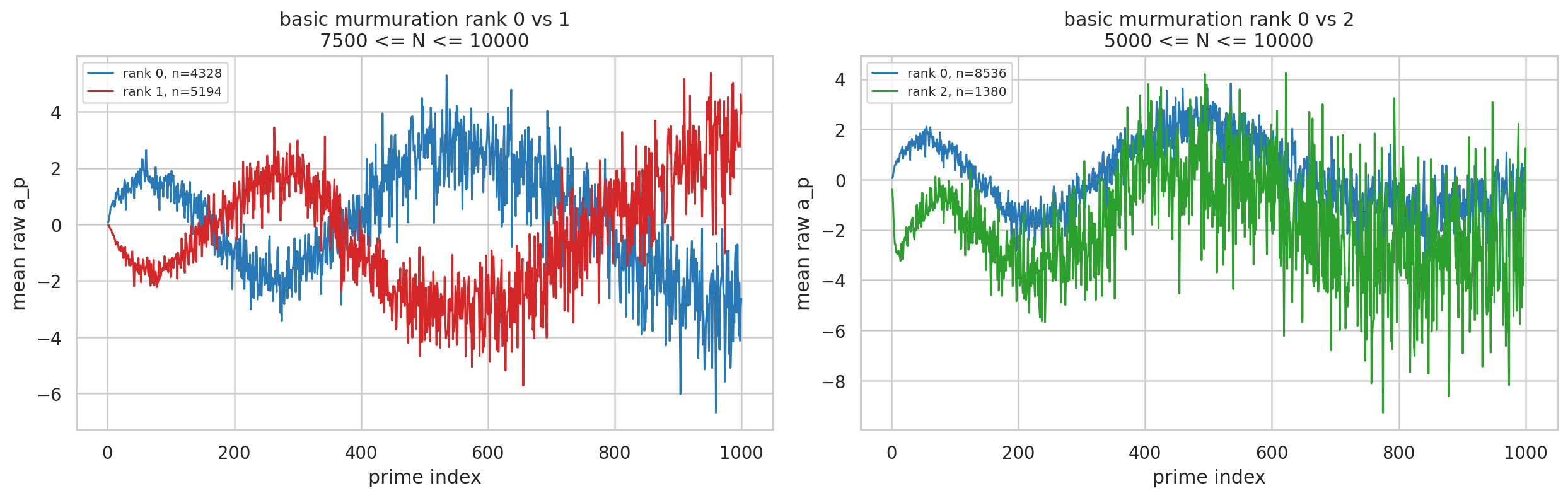}
\caption{Basic murmuration profiles from ordinary prime coefficients.}
\label{figBasic}
\end{figure}

\medskip

\noindent

Figure~\ref{figBasic} displays the mean raw coefficients $a_p$ ordered by prime index.  The left panel compares curves of ranks $0$ and $1$ with conductors in the interval $7500\leq N\leq10000$.
The right panel compares curves of ranks $0$ and $2$ with conductors in the interval $5000\leq N\leq10000$.  The sample sizes are displayed in the
two panels.
The rank $0$ and rank $1$ mean-coefficient sequences have Pearson correlation $-0.7445$, and their signs are opposite at $85.8$ percent of the $1000$ prime positions.
Their opposite waves provide a clear positive control for the later experiments.
The comparison between ranks $0$ and $2$ is noisier because the rank $2$ sample is much smaller.

\subsection{A common effective index profile}

We next place $H_k(x_p)$ at the effective position $p^k/8750$.
The vertical coordinate in Figure~\ref{figUniversal} is the balanced root-sign contrast
\begin{equation*}
\sqrt{8750}\,
\frac{\mu_+(p^k)-\mu_-(p^k)}{2}.
\end{equation*}
Here $\mu_+$ and $\mu_-$ are the conditional means for the two root signs in the conductor window from $7500$ to $10000$.
Before removal of bad prime entries, the window contains $5099$ curves with positive root sign and $5194$ curves with negative root sign.
The balanced contrast gives the two signs equal weight.
The corresponding theoretical statistic is the signed mean
\begin{equation*}
\mathbb E[\varepsilon H_k]
=
\pi_+\mu_+-\pi_-\mu_-.
\end{equation*}
These quantities agree when $\pi_+=\pi_-=1/2$.  Their difference in the displayed window is a small class proportion correction.
The available counts vary with $p$ because bad prime entries are missing.

The denominator $8750$ is the midpoint of the displayed conductor window.
It gives a convenient common scale for a finite descriptive plot.
Theorem~\ref{thm:scaled-murmuration} uses the lower endpoint $X$ of a shorter asymptotic window.
The two horizontal conventions serve different purposes: the midpoint gives a convenient finite-window scale, and the theorem uses the lower endpoint of an asymptotically short window.

The $H_2$ and $H_3$ points follow the same visible oscillation as the ordinary prime curve.
A descriptive interpolation calculation gives Pearson correlation $0.906$ for $H_2$ and $0.960$ for $H_3$.
The signs agree with the interpolated ordinary prime curve at $21$ of $23$ quadratic positions and all $8$ cubic positions.
Pointwise intervals exclude zero at $18$ quadratic positions and $6$ cubic positions.

Table~\ref{tabProfileAlignment} records the descriptive alignment.
The ordinary prime curve is linearly interpolated in the effective coordinate and then evaluated at the quadratic or cubic nodes.
The slope is obtained by an ordinary least squares fit with an intercept, using the power contrast as the response and the interpolated ordinary prime contrast as the explanatory variable.

\begin{table}[t]
\centering
\caption{Descriptive alignment with the ordinary prime profile.}
\label{tabProfileAlignment}
\begin{tabular}{lrrrrr}
\toprule
Block & Nodes & Correlation & Fitted slope & Sign agreement & Intervals excluding zero\\
\midrule
$H_2$ & 23 & 0.906 & 1.046 & $21/23$ & $18/23$\\
$H_3$ & 8 & 0.960 & 0.941 & $8/8$ & $6/8$\\
\bottomrule
\end{tabular}
\end{table}

\begin{figure}[t]
\centering
\includegraphics[width=\textwidth]{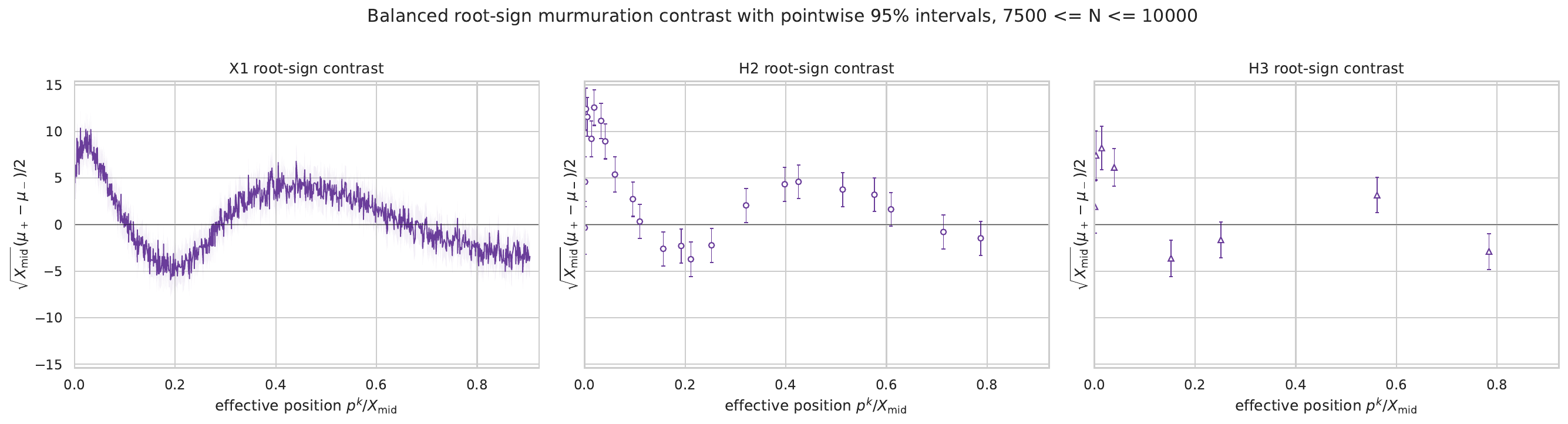}
\caption{Balanced root-sign contrasts for ordinary prime coefficients and the coordinates $H_2$ and $H_3$ obtained from Hecke relations in $7500\leq N\leq10000$.
The horizontal coordinate is the effective position $p^k/8750$.
Error bars are pointwise $95$ percent normal intervals based on standard errors.
The second and third panels sample the same oscillatory pattern as the first panel at the available prime-power locations.}
\label{figUniversal}
\end{figure}

These descriptive comparisons cover one conductor window, $23$ quadratic positions, and $8$ cubic positions.
The positions are strongly dependent, and the intervals are pointwise.
Figure~\ref{figUniversal} provides finite-sample evidence consistent with Conjecture~\ref{conjUniversal}.

Figure~\ref{figPowerCompactness} extends the same balanced root-sign calculation to the first $200$ base primes and every degree $1\leq k\leq12$.
The common base-prime index makes the amplitudes of equal-length transformed panels directly comparable.
The common effective-position display in Figure~\ref{figUniversal} records the conductor-scaled profile.

\begin{figure}[p]
\centering
\includegraphics[width=0.90\textwidth]{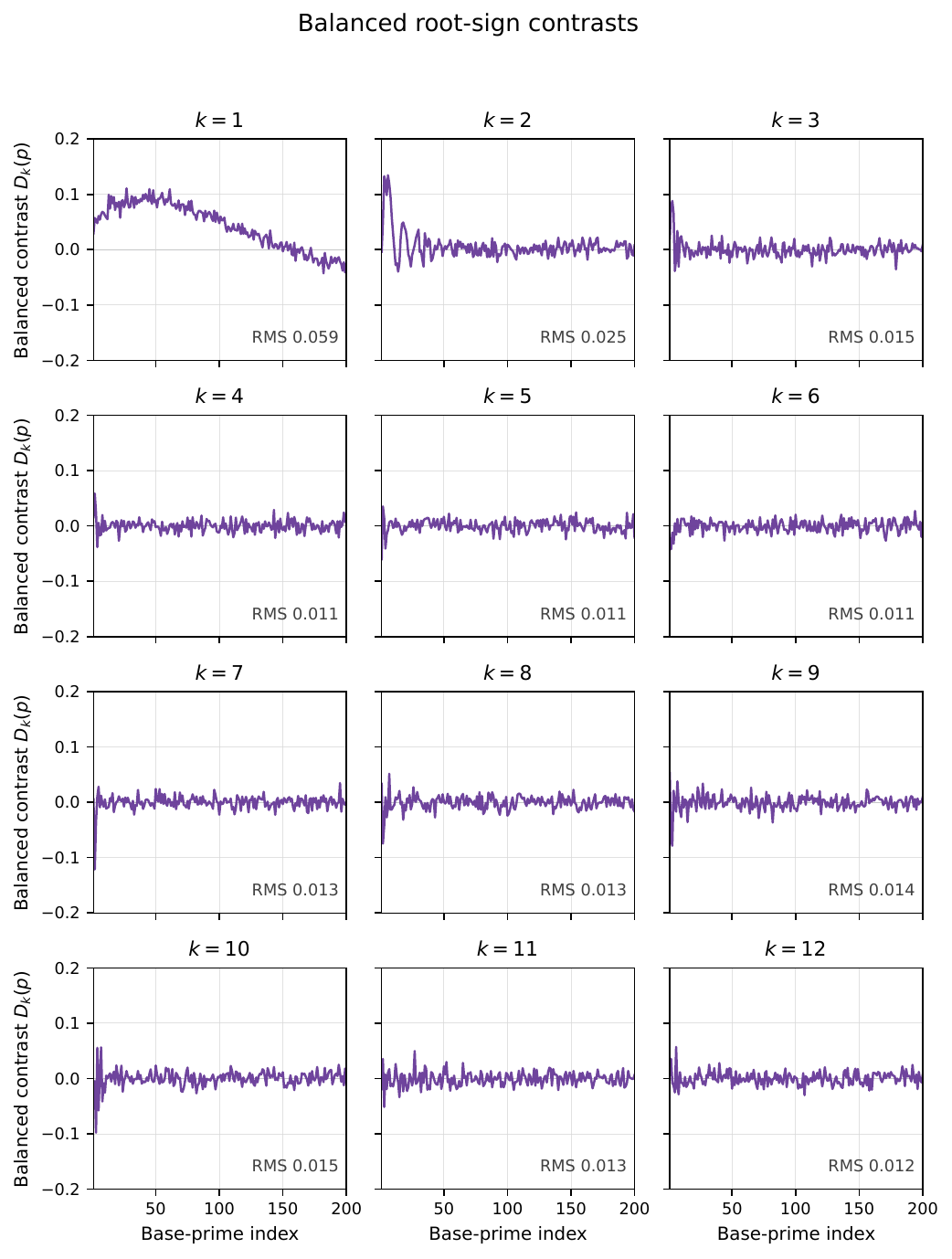}
\caption{Balanced root-sign contrasts for $k=1,\ldots,12$ in $7500\leq N\leq10000$.
At each of the first $200$ base primes,
$D_k(p)=\{\mu_{k,+}(p)-\mu_{k,-}(p)\}/2$, where $\mu_{k,\pm}(p)$ is the mean of $H_k(x_p)=a_{p^k}/p^{k/2}$ over curves with root sign $\pm1$ after ramified entries are omitted.
Every panel uses the same vertical scale.
The displayed value is $R_k=(200^{-1}\sum_{j=1}^{200}D_k(p_j)^2)^{1/2}$.}
\label{figPowerCompactness}
\end{figure}

On this common base-prime display, the higher-power contrasts form a visibly tighter band around zero.
The RMS falls from $0.0588$ at $k=1$ to $0.0250$ at $k=2$ and lies between $0.0110$ and $0.0147$ for $3\leq k\leq12$.
These values summarize a finite-window, equal-base-prime diagnostic.
The prime positions are dependent, so the RMS values are descriptive summaries.

\subsection{Prediction accuracy and error reduction}

Table~\ref{tabFixed} gives the main fixed-split results.
The change column is measured in percentage points.
Its interval is the pointwise paired bootstrap interval.

\begin{table}[t]
\centering
\caption{Prediction on the fixed test split.}
\label{tabFixed}
\resizebox{\textwidth}{!}{
\begin{tabular}{llrrr}
\toprule
Task & Panel & Balanced accuracy & Change with interval & Error reduction\\
\midrule
Root number & $X_1$ & $71.1875\%$ & & \\
 & $X_1+H_2$ & $71.9125\%$ & $0.7250\ [0.0122,1.4625]$ & $2.52\%$\\
 & $X_1+H_3$ & $71.2750\%$ & $0.0875\ [-0.4625,0.6500]$ & $0.30\%$\\
 & $X_1+H_2+H_3$ & $72.1000\%$ & $0.9125\ [0.2250,1.5628]$ & $3.17\%$\\
\midrule
Rank $0$ versus $1$ & $X_1$ & $96.0375\%$ & & \\
 & $X_1+H_2$ & $96.4250\%$ & $0.3875\ [0.0875,0.7000]$ & $9.78\%$\\
 & $X_1+H_3$ & $96.1750\%$ & $0.1375\ [-0.1250,0.4125]$ & $3.47\%$\\
 & $X_1+H_2+H_3$ & $96.6125\%$ & $0.5750\ [0.3000,0.8750]$ & $14.51\%$\\
\bottomrule
\end{tabular}}
\end{table}

For completeness, Table~\ref{tabAllTasks} reports every main task.
The comparison of ranks $0$ and $2$ is especially informative because both classes have the same root sign; it therefore isolates rank information beyond parity.

\begin{table}[t]
\centering
\caption{Balanced accuracy in percent for every main task.}
\label{tabAllTasks}
\begin{tabular}{lrrrr}
\toprule
Task & $X_1$ & $X_1+H_2$ & $X_1+H_3$ & Full panel\\
\midrule
Rank $0,1,2$ & 95.5750 & 95.1417 & 95.8333 & 94.6833\\
Rank $0$ versus $1$ & 96.0375 & 96.4250 & 96.1750 & 96.6125\\
Rank $0$ versus $2$ & 99.6875 & 99.6250 & 99.6000 & 99.6250\\
Rank $1$ versus $2$ & 99.7125 & 99.7000 & 99.7000 & 99.8625\\
Root number & 71.1875 & 71.9125 & 71.2750 & 72.1000\\
\bottomrule
\end{tabular}
\end{table}

\begin{figure}[t]
\centering
\includegraphics[width=0.86\textwidth]{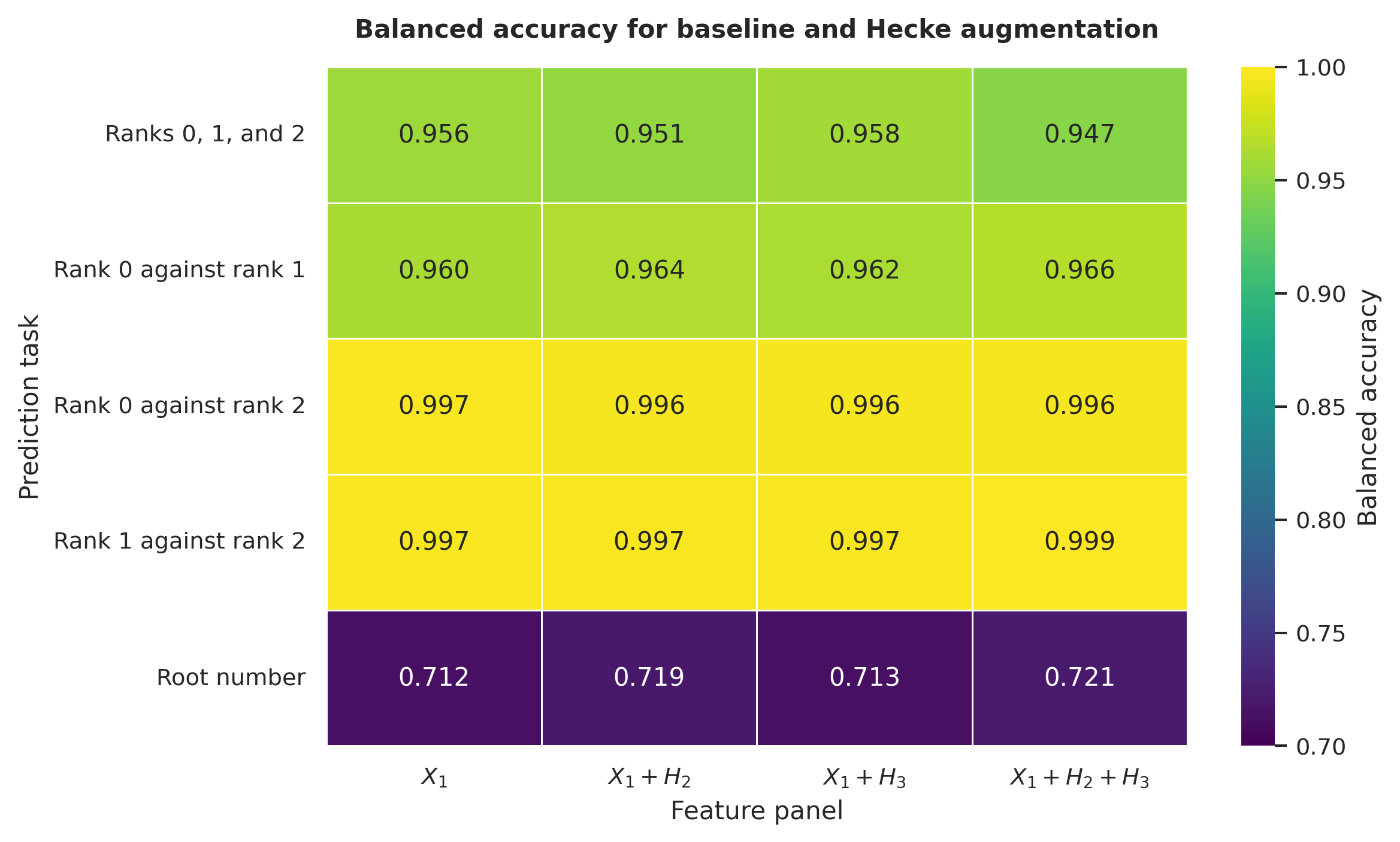}
\caption{Balanced accuracy for the four feature panels and five prediction tasks.
The first three rank rows have a high ordinary prime baseline.
The root-number row leaves more room for a visible absolute gain.}
\label{figAugmentedHeatmap}
\end{figure}

On the root-number test set, the ordinary prime model makes $2305$ errors and the full model makes $2232$.
The $31$ derived columns therefore remove $73$ errors among $8000$ test curves.
This is a $3.17$ percent reduction of the baseline error.
The Matthews correlation coefficient rises from $0.423870$ to $0.442043$.
The area under the receiver operating characteristic curve rises from $0.761864$ to $0.776113$.

For rank $0$ versus rank $1$, the error count falls from $317$ to $271$.
This is a reduction of $46$ errors and a $14.51$ percent relative error reduction.
The Matthews correlation coefficient rises from $0.921131$ to $0.932559$.

The remaining tasks illustrate the complementary roles of baseline error and ceiling effects.
The rank $0$ versus rank $2$ model makes $25$ errors with $X_1$ and $30$ with the full panel.
For rank $1$ versus rank $2$, the error count falls from $23$ to $11$ at an accuracy already near the ceiling.
In the three-class task, it rises from $531$ to $638$, an increase of $107$ errors.
Thus the effect of augmentation varies across rank tasks.

\begin{figure}[t]
\centering
\includegraphics[width=0.94\textwidth]{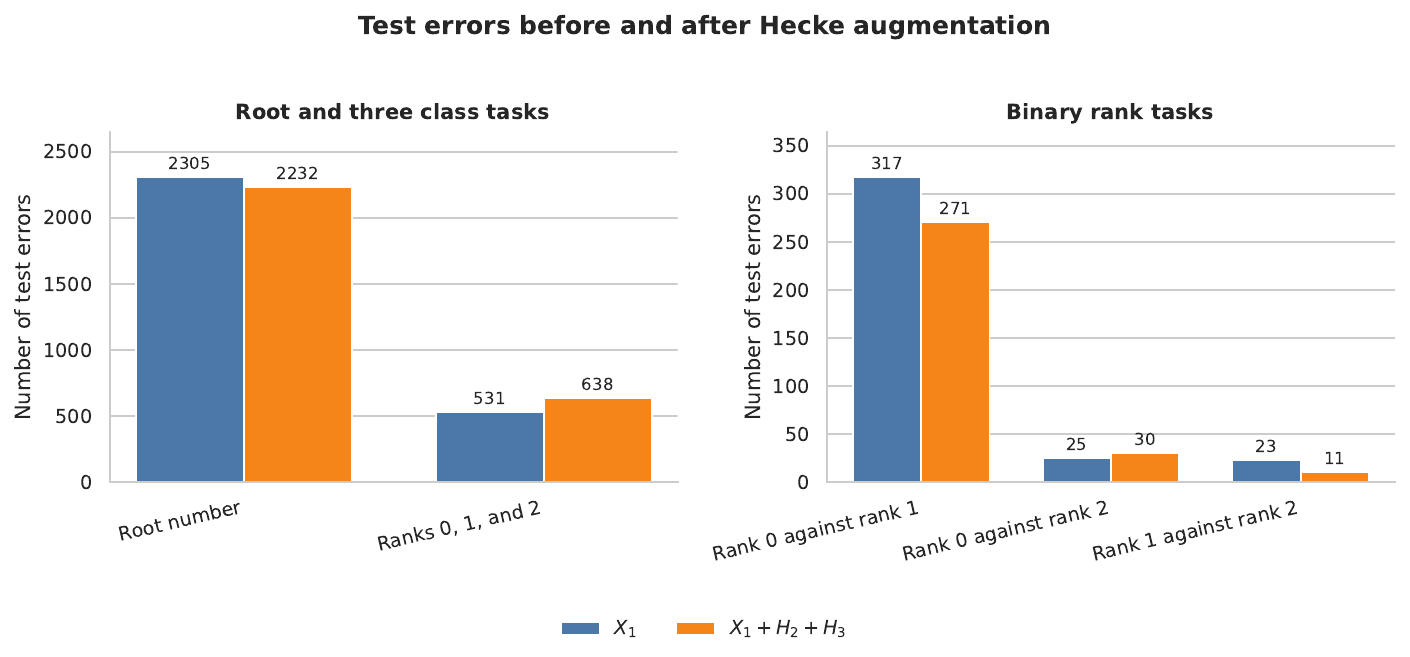}
\caption{Test error counts for the ordinary prime panel and the full augmented panel.
The root number and rank $0$ versus rank $1$ tasks show fewer errors after augmentation.
The three-class task and rank $0$ versus rank $2$ show the task dependence of the result.}
\label{figErrorCounts}
\end{figure}

The assignment stability results appear in Table~\ref{tabStability} and Figure~\ref{figStability}.
Means and standard deviations provide descriptive summaries across five complete fits.

\begin{table}[t]
\centering
\caption{Stability over five stratified assignments after retaining one representative per isogeny class.
Accuracy entries are in percent and change entries are in percentage points.
Parenthetical fractions count the assignments with a positive change.}
\label{tabStability}
\resizebox{\textwidth}{!}{
\begin{tabular}{lrrrrr}
\toprule
Task & $X_1$ accuracy & $H_2$ change & $H_3$ change &
Full change & Full positive assignments\\
\midrule
Root number & $70.9725\pm0.3206$ &
$0.4450\pm0.4169\ (4/5)$ &
$0.4775\pm0.2931\ (5/5)$ &
$0.7425\pm0.2838$ & $5/5$\\
Rank $0$ versus $1$ & $96.2675\pm0.1893$ &
$0.1400\pm0.2011\ (4/5)$ &
$0.1550\pm0.1148\ (5/5)$ &
$0.2900\pm0.2155$ & $4/5$\\
\bottomrule
\end{tabular}}
\end{table}

The full panel improves root-number prediction in all five assignments.
Its mean change is $0.7425$ percentage points.
This corresponds to mean full accuracy $71.7150$ percent and about $59.4$ fewer errors per $8000$ test curves.
The rank comparison improves in four assignments and ties in one.
Its mean change is $0.2900$ percentage points.
Its mean full accuracy is $96.5575$ percent and it corresponds to about $23.2$ fewer errors per $8000$
test curves.
Applying formula~\eqref{eqErrorReduction} to the mean accuracies gives relative error reductions of $2.56$ percent
and $7.77$ percent.

\begin{figure}[t]
\centering
\includegraphics[width=0.94\textwidth]{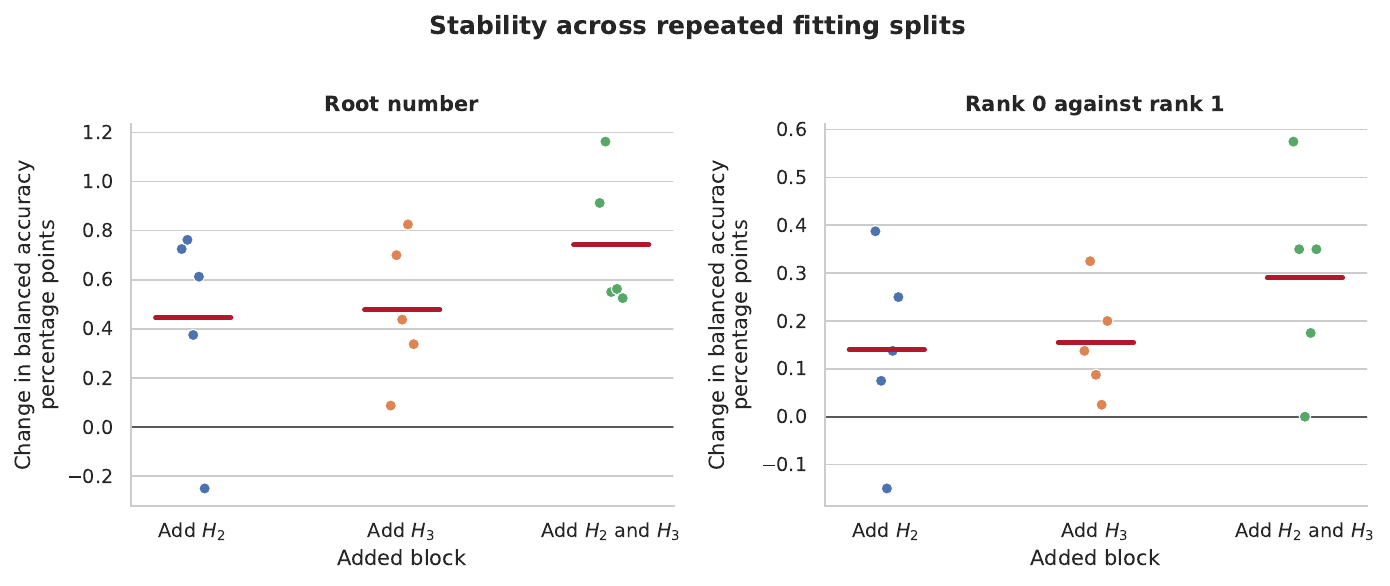}
\caption{Change in balanced accuracy relative to the ordinary prime panel over five stratified assignments from one fixed balanced population.
Each point is one complete fit with regularization selected on validation data.
The red segment is the mean.
The joint $H_2$ and $H_3$ panel improves root-number prediction in all five assignments.}
\label{figStability}
\end{figure}

\FloatBarrier

\subsection{Root number geometry as a model diagnostic}

The root-number classes have a special geometry in this data set.
Ranks $0$ and $2$ share one sign, and rank $1$ has the opposite sign.
If the rank signal is approximately ordered along a low-dimensional direction, one root-sign class occupies two separated regions, a geometry represented poorly by a single logistic hyperplane.

We examine this interaction by fitting the three rank classes and aggregating their fitted probabilities according to parity.
The result appears in Table~\ref{tabParityGeometry}.  This rank-supervised procedure changes both the decision geometry and the supervision available during fitting.

\begin{table}[t]
\centering
\caption{Direct root classification and the rank-then-parity diagnostic on the fixed split.
Accuracy columns are in percent.}
\label{tabParityGeometry}
\begin{tabular}{lrrrr}
\toprule
Panel & Direct accuracy & Direct MCC & Parity accuracy & Parity MCC\\
\midrule
$X_1$ & 71.1875 & 0.4239 & 92.3375 & 0.8495\\
$X_1+H_2$ & 71.9125 & 0.4383 & 93.7750 & 0.8807\\
$X_1+H_3$ & 71.2750 & 0.4256 & 92.1000 & 0.8442\\
$X_1+H_2+H_3$ & 72.1000 & 0.4420 & 94.0000 & 0.8848\\
\bottomrule
\end{tabular}
\end{table}

The large difference between the direct and parity columns is much larger than the augmentation gain.
The comparison leaves the effects of classifier geometry and richer rank supervision entangled.
Within the parity diagnostic, the full panel raises balanced accuracy by $1.6625$ percentage points and reduces the error count from $613$ to $480$.
This result supports the representation within a rank-supervised diagnostic, so its evidential content overlaps with rank prediction.

\FloatBarrier

\subsection{The number of derived coordinates}

We vary the number of added coordinates while keeping the ordinary prime panel fixed.
Natural prefixes add the available $H_2$ and $H_3$ coordinates in increasing effective index.
Random subsets separate the nominal count from the particular low prime positions.
The regularization value is fixed at the value selected for $X_1$.

With the regularization value fixed at the value selected for $X_1$, the root-number gain at the full count $m=31$ is $0.8875$ percentage points.
The paired interval on the fixed test set is
$[0.1747,1.6250]$.
Across five matched row shuffles, the mean difference between the real augmentation and its shuffled control is
$0.9100$ percentage points.
It is positive in all five comparisons.
Across these five controls, breaking curve-feature alignment removes the gain produced by the real $31$-column augmentation.

\begin{figure}[t]
\centering
\includegraphics[width=0.96\textwidth]{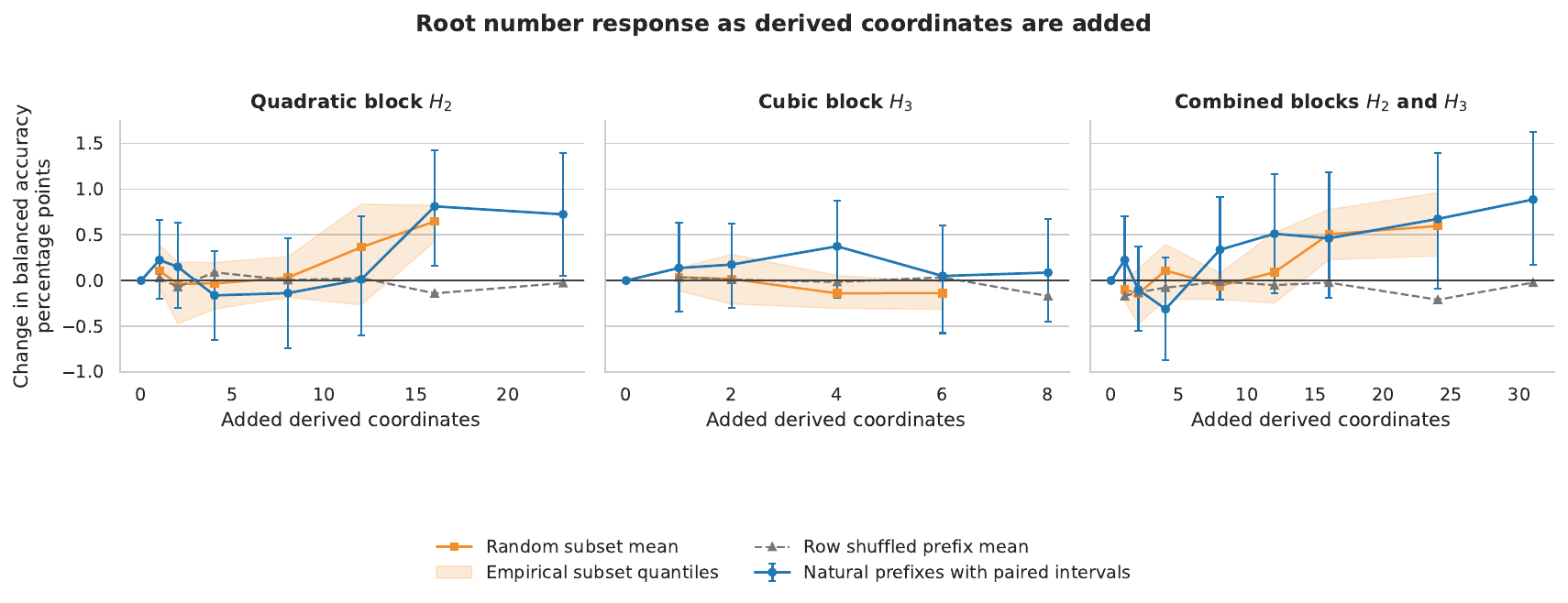}
\caption{Root-number change in balanced accuracy as derived coordinates are added to $X_1$.
Natural prefixes follow effective index order.
Random subsets use only five repetitions at each
displayed size.
The orange bands show the empirical $2.5$th and $97.5$th percentiles from five subsets.
They are descriptive empirical ranges.
The curve is irregular and contains several local reversals.}
\label{figQuantity}
\end{figure}

\begin{figure}[t]
\centering
\includegraphics[width=0.96\textwidth]{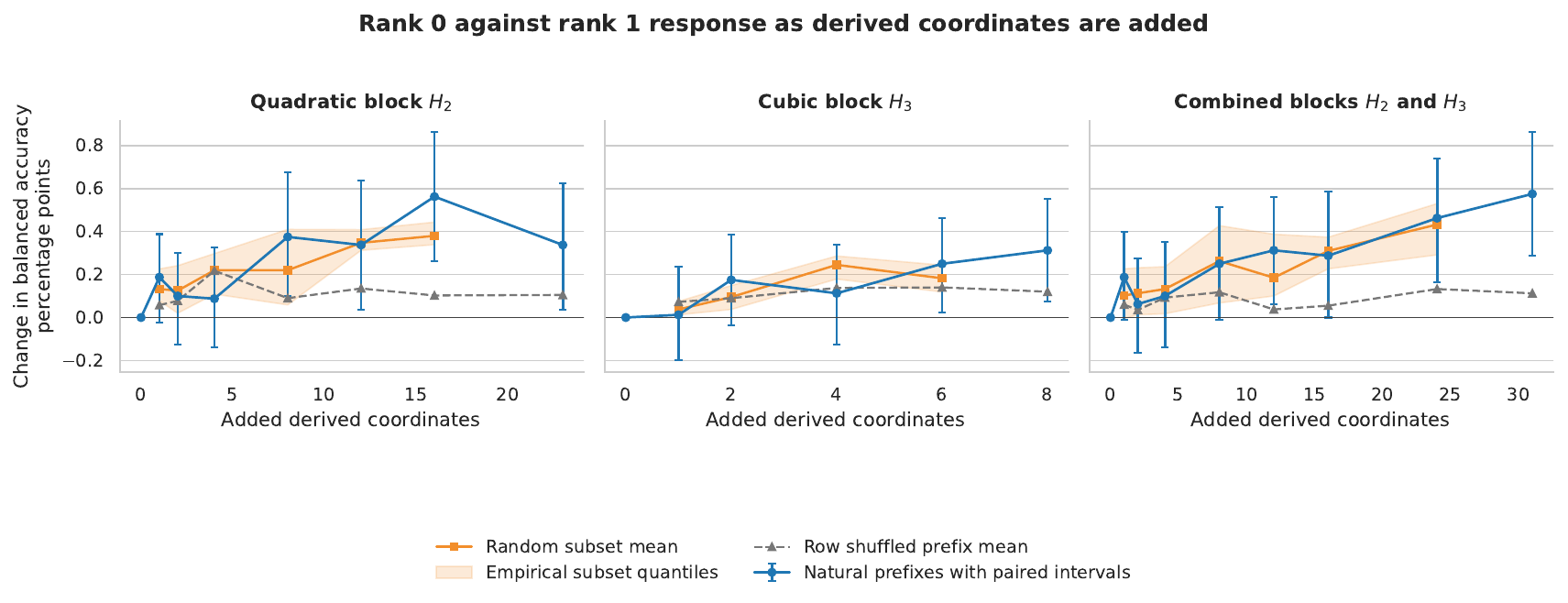}
\caption{Change in balanced accuracy for rank $0$ versus rank $1$ as derived coordinates are added to $X_1$.
Natural prefixes follow effective index order.
Orange bands show the empirical $2.5$th and $97.5$th percentiles from five random subsets.
They are descriptive empirical ranges.
The response is gradual and contains local reversals.}
\label{figQuantityRank}
\end{figure}

The natural-prefix response contains local reversals.
The quadratic block reaches its largest gain at $16$ coordinates and declines slightly at the full count of $23$.
In this quantity curve with $\alpha$ fixed, the endpoint for all $8$ cubic coordinates is $0.0875$ percentage points, with a paired interval containing zero.
This differs from the separately tuned repeated-assignment comparison, where the cubic panel has mean gain $0.4775$ percentage points and is positive in all five assignments.
Here $m$ counts derived coordinates: the full set of $31$ columns arises from $23$ base primes, with each column observed across many curves.
The evidence therefore supports a structured representation effect whose relation to feature count remains task-dependent.

The random subset curves give a second view.
For root number, random quadratic subsets have almost no mean gain up to $8$ coordinates.
At $16$ coordinates their mean gain is $0.6475$ percentage points, and all five subsets give a positive change.
Random joint subsets at size $24$ have mean gain $0.5975$ percentage points and are positive in all five cases.
Cubic subsets alone follow a different pattern.
At size $6$, their mean root-number change is $-0.1375$ percentage points.

For rank $0$ versus rank $1$, the random subset trends are more gradual.
The mean changes at the largest sampled sizes are $0.3800$ percentage points for $16$ quadratic coordinates, $0.1825$ percentage points for $6$ cubic coordinates, and $0.4325$ percentage points for $24$ joint coordinates.
Every one of the five subsets is positive in these three comparisons.

Natural prefixes mix three effects.
They change the number of columns, the base primes, and the proportion of nonmissing values.
For the quadratic prefixes, mean coverage rises from $0.227$ at the first coordinate to $0.868$ at full size.
With five subsets, the random-subset curves reduce position dependence and remain descriptive.
We therefore present them as a quantity diagnostic; fitting a parametric dose model would require a larger sampling study.

\FloatBarrier

\subsection{Sampling and shuffle controls}

Table~\ref{tabControls} collects the simple negative controls.
The conductor-only model is weak and measurably above chance for root number, while the label permutations remain close to chance.
The reported label permutation is a single draw from the null and provides a coarse chance-level check.
The row-shuffled full block keeps $31$ extra columns after breaking their alignment with each elliptic curve.

\begin{table}[H]
\centering
\caption{Balanced accuracy in percent for sampling and shuffle
controls.}
\label{tabControls}
\begin{tabular}{lrr}
\toprule
Control & Root number & Rank $0$ versus $1$\\
\midrule
Only the conductor & 52.0875 & 50.9875\\
Permuted labels with $X_1$ & 49.1500 & 52.0875\\
Permuted labels with full panel & 51.7500 & 50.8625\\
Fit using the row-shuffled full block & 70.1500 & 96.0500\\
Real $X_1$ baseline & 71.1875 & 96.0375\\
Real full panel & 72.1000 & 96.6125\\
\bottomrule
\end{tabular}
\end{table}

The real full panel outperforms the single row-shuffled fit.
The quantity study strengthens this comparison by repeating the matched shuffle five times.
For the full root-number block, the real gain minus the shuffled gain has mean $0.9100$ percentage points and is positive in all five comparisons.
For rank $0$ versus rank $1$, the corresponding mean difference is $0.4625$ percentage points and is again positive in all five comparisons.

\FloatBarrier

\subsection{Polynomial, conductor, and solver controls}
\label{secAdditionalControls}

We ran a matched confirmation experiment to separate three possible sources of the gain: the quadratic--cubic function class, explicit conductor information, and the stochastic optimizer.
The ordinary monomial panel contains the same $1000$ coordinates as $X_1$, together with $x_p^2$ at the $23$ quadratic base primes and $x_p^3$ at the $8$ cubic base primes.
It has the same dimension as $X_1+H_2+H_3$.
Proposition~\ref{propRepresentation} shows that the two augmented panels have the same unregularized affine column span.
Standardization and the $\ell_2$ penalty depend on the chosen coordinates, so the comparison measures the effect of basis parameterization under the recorded fitting rule.

The conductor panels add the single standardized coordinate $\log_{10}N$ to $X_1$ or to the full Hecke panel.
Every representation uses the same balanced population, the same five train--validation--test assignments, and validation selection over the original regularization grid.
The software environment for this confirmation run is Python $3.12.13$, NumPy $2.5.2$, pandas $3.0.5$, SciPy $1.18.1$, and scikit-learn $1.9.0$.
The archived main run lacks package-version metadata, so Table~\ref{tabAdditionalRandom} reports the confirmation run as a separate paired experiment.

\begin{table}[t]
\centering
\caption{Matched controls over five random assignments of one fixed balanced population.
The $X_1$ column gives mean balanced accuracy in percent; the other columns give the mean change in percentage points relative to the matched $X_1$ fit.
Values after $\pm$ are sample standard deviations, and parenthetical fractions count positive changes.}
\label{tabAdditionalRandom}
\resizebox{\textwidth}{!}{
\begin{tabular}{lrrrrr}
\toprule
Task & $X_1$ accuracy & $X_1+H_2+H_3$ & $X_1+x^2+x^3$ & $X_1+\log_{10}N$ & $X_1+H_2+H_3+\log_{10}N$\\
\midrule
Root number & $71.1000\pm0.3202$ & $0.9075\pm0.3494\ (5/5)$ & $0.9400\pm0.2760\ (5/5)$ & $1.2675\pm0.3014\ (5/5)$ & $2.0450\pm0.1258\ (5/5)$\\
Rank $0$ versus $1$ & $96.0150\pm0.2205$ & $0.3150\pm0.2468\ (5/5)$ & $0.2525\pm0.2909\ (4/5)$ & $0.1175\pm0.0763\ (4/5)$ & $0.4225\pm0.1633\ (5/5)$\\
\bottomrule
\end{tabular}}
\end{table}

For root number, the ordinary monomial basis and the Hecke basis produce gains of the same order.
For rank $0$ versus rank $1$, their mean gains are also close relative to the assignment variation.
The empirical support therefore attaches to their shared quadratic--cubic function class, with the Hecke relations supplying canonical arithmetic coordinates and an effective-index rule.

The log conductor raises the root-number baseline by $1.2675$ percentage points on average.
Adding the Hecke block to that conductor-aware baseline supplies a further mean gain of $0.7775\pm0.3813$ percentage points, positive in all five assignments.
The corresponding incremental gain for rank $0$ versus rank $1$ is $0.3050\pm0.1106$ percentage points, also positive in all five assignments.
These are descriptive conditional accuracy differences after explicit inclusion of the conductor.

For each assignment and task, we also applied the two-sided exact McNemar test to the discordant predictions from each augmented fit and its matched $X_1$ fit.
At the unadjusted $0.05$ level, the Hecke comparison against $X_1$ is significant in four of five root-number assignments and two of five rank assignments; the monomial comparison has the same counts.
The conductor-aware Hecke panel, also compared with $X_1$, is significant in all five root-number assignments and four of five rank assignments.
The tests share an underlying population and are reported as descriptive paired checks.

The power coordinates inherit the missing pattern of their base-prime coordinates, so we also isolate low-prime bad reduction.
For each of the $23$ primes $p\leq83$, define
\begin{equation*}
M_p(E)={\boldsymbol 1}_{\{p\mid N_E\}},
\end{equation*}
and write $M_{23}$ for the resulting indicator panel.
Table~\ref{tabMissingnessControls} compares $X_1+M_{23}$ with $X_1$ and then adds the Hecke block to the indicator-aware baseline.

\begin{table}[t]
\centering
\caption{Explicit low-prime missingness controls over the five matched random assignments.
Accuracy and changes use the same units as Table~\ref{tabAdditionalRandom}; parenthetical fractions count positive changes.}
\label{tabMissingnessControls}
\resizebox{\textwidth}{!}{
\begin{tabular}{lrrrr}
\toprule
Task & $X_1+M_{23}$ accuracy & Change from $X_1$ & Full panel $+M_{23}$ accuracy & Change from $X_1+M_{23}$\\
\midrule
Root number & $71.8575\pm0.3600$ & $0.7575\pm0.0753\ (5/5)$ & $72.6775\pm0.2585$ & $0.8200\pm0.2570\ (5/5)$\\
Rank $0$ versus $1$ & $96.3400\pm0.1717$ & $0.3250\pm0.1718\ (5/5)$ & $96.5225\pm0.2581$ & $0.1825\pm0.2126\ (4/5)$\\
\bottomrule
\end{tabular}}
\end{table}

The indicator panel improves both tasks, showing that low-prime divisibility is a meaningful predictor under mean imputation.
After these indicators are explicit, the Hecke block retains a mean gain of $0.8200$ percentage points for root number and $0.1825$ percentage points for rank $0$ versus rank $1$.
These outcomes quantify the two conditional predictive increments under the recorded fitting procedure.
At the unadjusted $0.05$ level, exact McNemar tests for $X_1+M_{23}$ against $X_1$ are significant in five of five root-number assignments and three of five rank assignments.
Tests for the full panel plus $M_{23}$ against $X_1+M_{23}$ are significant in three and two assignments, respectively.

We next sort each class by conductor, divide it into five equal blocks, and hold out one block from every class at a time.
This design tests transport across conductor ranges while preserving class balance.
Table~\ref{tabConductorHoldout} summarizes the five disjoint test folds, and Figure~\ref{figAdditionalControls} shows every fold.

\begin{table}[t]
\centering
\caption{Conductor-quintile holdout over five test folds.
Accuracy and changes use the same units as Table~\ref{tabAdditionalRandom}.}
\label{tabConductorHoldout}
\begin{tabular}{lrrr}
\toprule
Task & $X_1$ accuracy & Hecke change & Monomial change\\
\midrule
Root number & $67.5800\pm10.0085$ & $0.7225\pm1.1789\ (3/5)$ & $0.7425\pm0.8330\ (4/5)$\\
Rank $0$ versus $1$ & $95.6675\pm0.9686$ & $0.2000\pm0.3110\ (4/5)$ & $0.2450\pm0.2228\ (5/5)$\\
\bottomrule
\end{tabular}
\end{table}

The root-number $X_1$ baseline ranges from $50.9750$ percent in the lowest conductor fold to $76.5750$ percent in the middle fold.
The Hecke change ranges from $-0.3875$ to $2.3875$ percentage points and is positive in three folds.
This heterogeneity narrows the random-assignment conclusion to within-population prediction; transport to a new conductor range requires further confirmation.
The rank comparison is more stable, with a positive Hecke change in four folds and a positive monomial change in all five.

\begin{figure}[t]
\centering
\includegraphics[width=\textwidth]{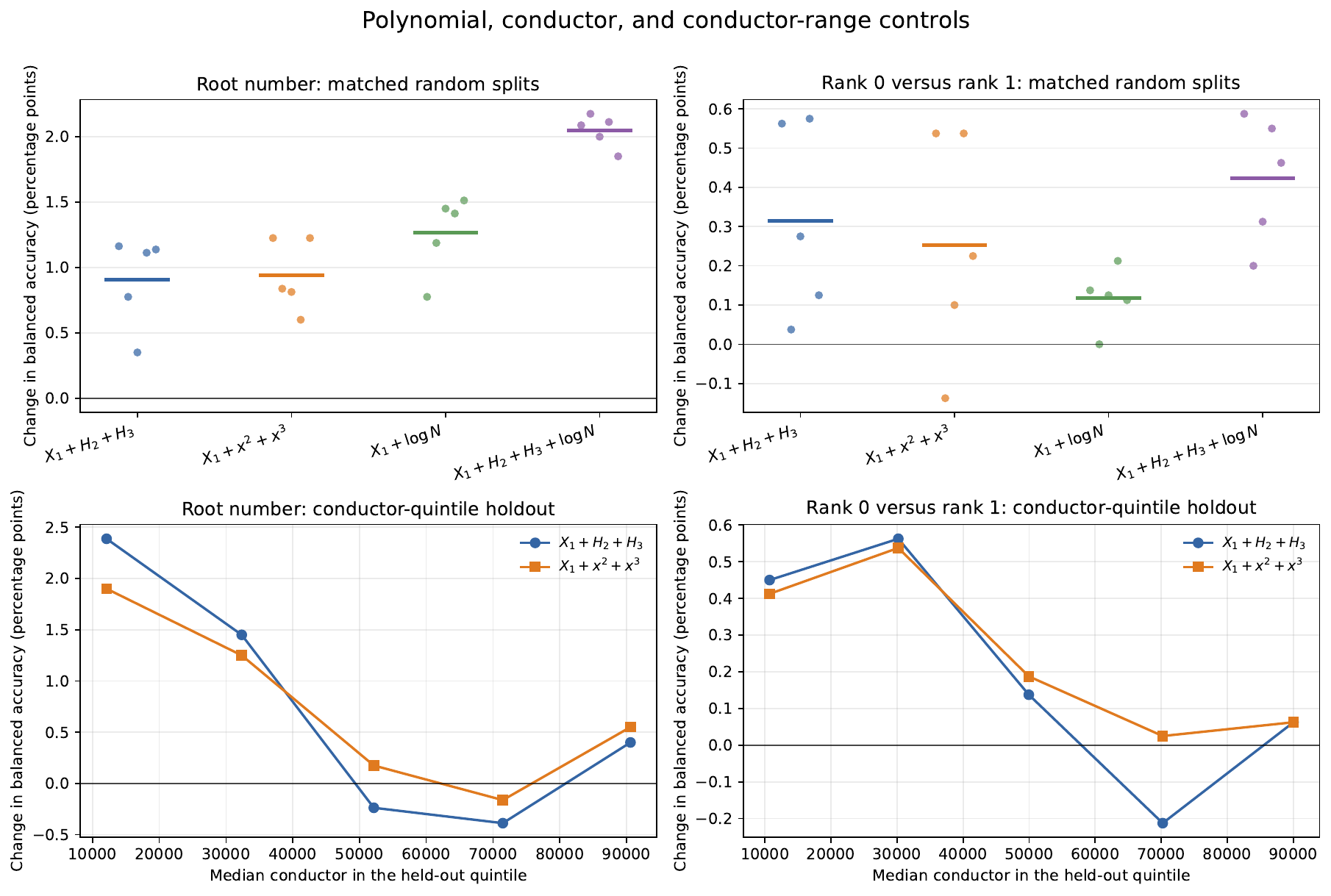}
\caption{Matched polynomial and conductor controls.
The upper panels show changes relative to $X_1$ for five random assignments; horizontal segments give means.
The lower panels hold out complete conductor quintiles and plot each change at the median conductor of its test fold.}
\label{figAdditionalControls}
\end{figure}

Finally, we refit $X_1$ and the Hecke panel with a deterministic L-BFGS logistic solver, select $C$ from $\{0.3,1,3\}$ on the external validation part, and disable internal early stopping.
Across the same five assignments, the root-number mean changes from $72.1400\pm0.2901$ percent to $72.8875\pm0.4104$ percent, a gain of $0.7475\pm0.2798$ percentage points that is positive in every assignment.
For rank $0$ versus rank $1$, the corresponding values are $95.9025\pm0.2338$, $96.1100\pm0.2640$, and $0.2075\pm0.1033$ percentage points, again positive in every assignment.
The direction of the augmentation effect therefore persists under a solver with a distinct optimization and stopping rule.

\FloatBarrier

\subsection{Exploratory single-block substitution}
\label{secNegative}

As a complementary representation test, we evaluate each symmetric-power block on its own.
This exploratory analysis comes from a separately recorded run, so all comparisons in this subsection are made internally within that run.
For each $k=1,\ldots,12$, the ordinary-prime panel is replaced by
\begin{equation*}
\bigl(V_{k,p}(E)\bigr)_{p\leq7919},
\qquad
V_{k,p}(E)=
\begin{cases}
H_k\bigl(a_p(E)/\sqrt p\bigr),&p\nmid N_E,\\
\text{missing},&p\mid N_E.
\end{cases}
\end{equation*}
Every model receives $1000$ transformed columns after training-set imputation.
The design uses all base primes $p\leq7919$ for every $k$, including effective indices above $7919$.
It is an equal-column substitution stress test with a different effective-index range from Theorem~\ref{thmIntroProfile}.
Each block remains sparse on the effective integer scale while retaining the same model dimension.

The labels throughout are the analytic rank and root number of the original elliptic curve.
The internal $k=1$ balanced accuracies are $0.706500$ for root number and $0.961250$ for rank $0$ versus rank $1$.
The available metadata record an $80/20$ split, $20\,000$ curves per class, and the missing-value policy at bad primes; validation records and split assignments are unavailable.
These metadata motivate the exploratory interpretation adopted here.

Under the effective index cap used in the main experiment, the number of available columns would fall rapidly.
The counts for $k=1,\ldots,12$ are
\begin{equation*}
1000,\ 23,\ 8,\ 4,\ 3,\ 2,\ 2,\ 2,\ 1,\ 1,\ 1,\ 1.
\end{equation*}
The equal column run deliberately grants each higher power many more locations than this capped design.

\begin{table}[t]
\centering
\caption{Balanced-accuracy proportions for one fixed symmetric-power block used alone.}
\label{tabSinglePower}
\resizebox{\textwidth}{!}{
\begin{tabular}{rrrrrr}
\toprule
$k$ & Rank $0,1,2$ & Rank $0$ versus $1$ &
Rank $0$ versus $2$ & Rank $1$ versus $2$ & Root number\\
\midrule
1  & 0.958 & 0.961 & 0.995 & 0.996 & 0.707\\
2  & 0.359 & 0.535 & 0.518 & 0.547 & 0.538\\
3  & 0.358 & 0.513 & 0.519 & 0.521 & 0.515\\
4  & 0.347 & 0.510 & 0.505 & 0.504 & 0.501\\
5  & 0.339 & 0.502 & 0.515 & 0.497 & 0.503\\
6  & 0.349 & 0.496 & 0.520 & 0.512 & 0.502\\
7  & 0.343 & 0.501 & 0.520 & 0.512 & 0.499\\
8  & 0.341 & 0.503 & 0.525 & 0.511 & 0.509\\
9  & 0.396 & 0.532 & 0.595 & 0.545 & 0.502\\
10 & 0.346 & 0.518 & 0.531 & 0.518 & 0.511\\
11 & 0.406 & 0.526 & 0.612 & 0.567 & 0.513\\
12 & 0.351 & 0.506 & 0.528 & 0.522 & 0.514\\
\bottomrule
\end{tabular}}
\end{table}

\begin{figure}[t]
\centering
\includegraphics[width=0.88\textwidth]{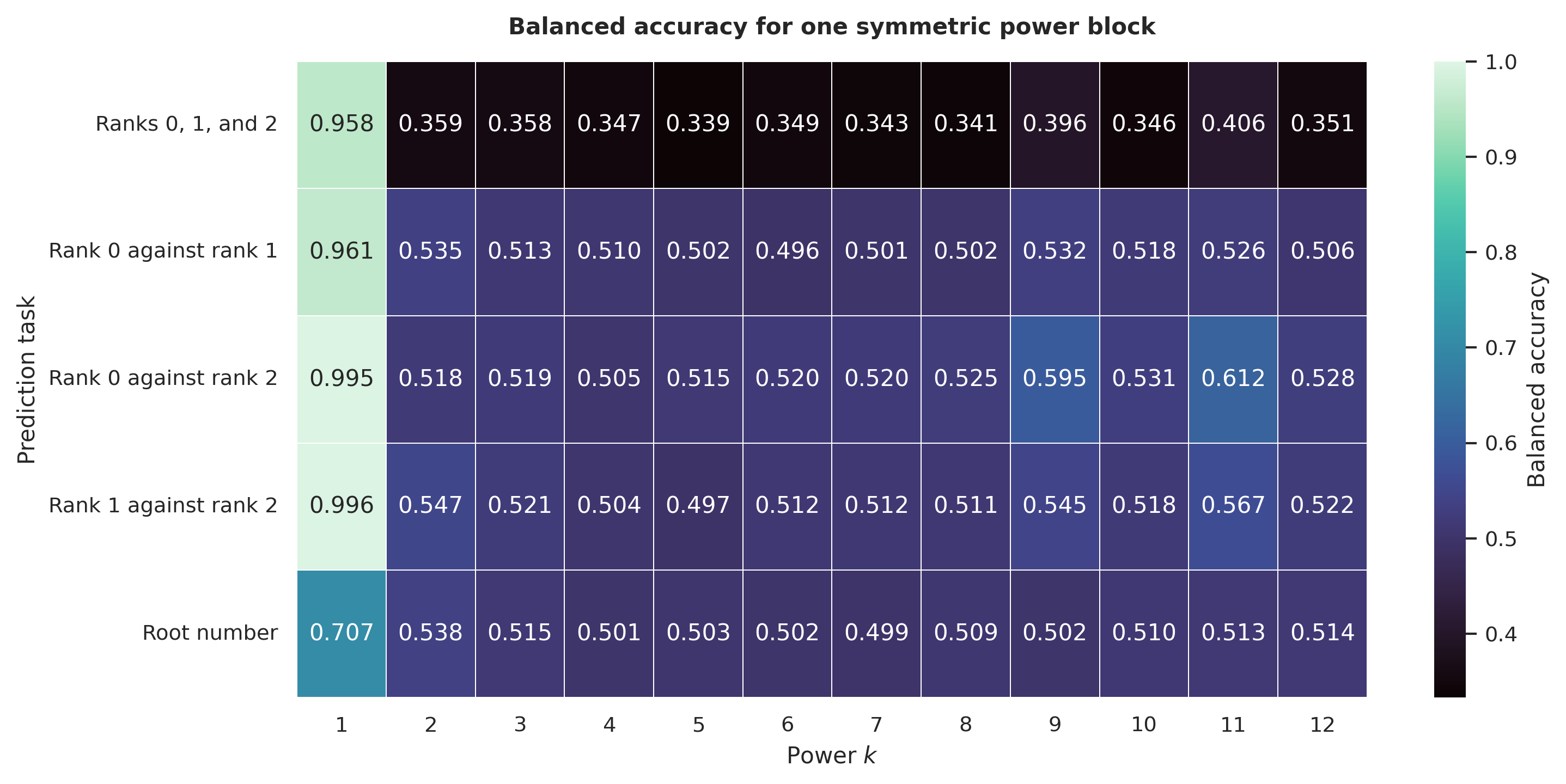}
\caption{Balanced accuracy for the single-block experiment.
Every column $k$ uses $1000$ features $H_k(x_p)$.
Binary chance is $0.5$, and three-class chance is $1/3$.
The labels are arithmetic invariants of the original elliptic curve.}
\label{figSinglePower}
\end{figure}

The contrast with the augmented experiment is sharp.  For root number, every $k\geq2$ lies between $0.499$ and $0.538$.
At $k=2$, the quadratic block shows weak predictive signal, with area under the ROC curve $0.5523$, and its balanced accuracy remains far below that of $k=1$.
Rank $0$ versus rank $1$ behaves similarly.
There are larger exploratory values at $k=9$ and $k=11$ in several rank tasks.
For example, rank $0$ versus rank $2$ reaches $0.5954$ and $0.6123$.
These peaks arise from an exploratory scan of $55$ higher-power task comparisons on one split and therefore require independent replication.

One contributing representation effect is noninjectivity.
Every $H_k$ with $k\geq2$ is noninjective on the Sato--Tate interval.
Every even $k$ removes the sign of $x$, and odd powers with $k\geq3$ oscillate.
Replacing $H_1$ with a noninjective $H_k$ may cause a linear classifier to lose the strong $H_1$ signal.
The archived design conflates this effect with feature distribution, standardization, regularization, effective range, and model class.

The universal profile concerns a small conditional mean at the effective position $p^k/X$ and leaves information retention by a single higher character unspecified.
The substitution results are compatible with Theorem~\ref{thmIntroProfile} and are consistent with the main design, which retains the ordinary-prime panel and adds prime-power blocks as structured polynomial features.

\subsection{Limitations}

The empirical gain depends on the task.
The full panel lowers three-class rank balanced accuracy from $95.5750$ percent to $94.6833$ percent.
Rank $0$ versus rank $2$ changes from $99.6875$
percent to $99.6250$ percent.
Rank $1$ versus rank $2$ rises from $99.7125$ percent to $99.8625$ percent at a baseline already near the ceiling.
Across tasks, the effect of additional Hecke coordinates is heterogeneous.

The repeated analysis uses five overlapping test assignments from one fixed balanced population.
More independent populations and conductor ranges are needed for a precise estimate of the average gain.
The paired intervals in Table~\ref{tabFixed} condition
on one trained pair of models.
The quantity experiment uses only five random subsets and five shuffled controls.

The generic monomial control reproduces the same order of gain as the Hecke basis.
Since $H_2=x^2-1$ and $H_3=x^3-2x$, both panels expose the same quadratic--cubic affine function class.
The results support this function class, while the Hecke relations provide its canonical arithmetic organization.
Explicit low-prime bad-reduction indicators also improve both principal tasks.
After their inclusion, the Hecke block retains positive mean increments; the descriptive attenuation relative to the unadjusted increments suggests partial overlap between the two feature blocks.
The conductor-quintile experiment also reveals root-number heterogeneity across ranges, including two negative Hecke changes.
Independent higher-conductor data are required to establish transport beyond the sampled population.

The archived main run omits package-version metadata, and its first-split values differ from the recorded confirmation run by at most $0.1375$ percentage points.
The confirmation tables therefore provide internally paired results from a fully recorded environment.
The deterministic-solver experiment supports the direction of the effect and provides a reproducible optimization check.

The recorded probability predictions perform poorly under Brier score and log loss; calibration was not assessed separately.
For root number, the Brier scores are $0.571206$ for $X_1$ and $0.547964$ for the full panel.
The recorded Brier score uses the two-class convention
\begin{equation*}
\frac1n\sum_{i=1}^n\sum_{c\in\{-1,1\}}
\left(\widehat p_{i,c}-{\boldsymbol 1}_{\{y_i=c\}}\right)^2.
\end{equation*}
Under this convention, the balanced constant score has value $0.5$.
The corresponding log-loss values are $8.788247$ and $6.746905$.
We therefore use balanced accuracy, the Matthews correlation coefficient, and area under the ROC curve as the primary empirical summaries.
The arithmetic theorem establishes an asymptotic root-number-weighted mean, while finite-sample classification accuracy remains a separate empirical question.

\section{Conclusion}

Prime coefficients and prime-power coefficients can be viewed on one effective conductor scale.
Conjecture~\ref{conjUniversal} expresses this principle.
The weight $2$ level aspect calculation verifies the principle for the squarefree newform family and every fixed power under the conditions of Theorem~\ref{thm:scaled-murmuration}.
In the two archimedean families, Proposition~\ref{propAllPrimePowersAverage} identifies the aggregate proper-power cancellation whose validity transfers the leading term of the published prime average to the natural von Mangoldt average.

The same viewpoint gives a practical feature design.  Ordinary prime coefficients remain the baseline.
Quadratic and cubic Hecke coordinates then add selected samples of the common profile without new coefficient queries.
In the elliptic-curve experiment, this representation improves direct root-number prediction in every one of
five overlapping assignments from the fixed balanced population.
The average balanced-accuracy gain is modest, and the augmented model reduces the baseline error in each assignment.
Rank $0$ versus rank $1$ also improves, with a smaller absolute gain.

The matched controls refine the interpretation.
An ordinary monomial basis yields a comparable gain, and the Hecke block adds signal in separate analyses that explicitly include $\log_{10}N$ or low-prime bad-reduction indicators.
A deterministic logistic solver preserves the direction of both principal effects.
Conductor-quintile holdout produces a less stable root-number gain, which limits the present predictive conclusion to the sampled conductor population.

The exploratory single-block experiment for $k=2,\ldots,12$ supplies a complementary representation test.
At fixed column count, higher-power blocks lose most of the ordinary-prime block's predictive performance for invariants of the original form.
Universal murmuration therefore guides structured augmentation with the ordinary-prime coefficients retained as the baseline.

\newpage

\appendix

\section{Technical supplements to the level aspect proof}
\label{appProof}

This appendix provides three auxiliary details for Section~\ref{sec:scaled-murmuration}: the exceptional local factor, the parameter inequalities, and the endpoint contribution in the prime-deleted character sum.

\subsection{The exceptional local factor}

The local factor at $P$ in Lemma~\ref{lem:local-arithmetic} requires separate treatment for an arbitrary fixed power.
Put $Q=P^k$.  In the range of the hyperbolic sum, Lemma~\ref{lem:p-coprime} gives
\begin{equation*}
P\nmid dr.
\end{equation*}
For every $\alpha\geq1$, one then has
\begin{equation*}
\theta_{r,Q}(P^\alpha)=\varphi(P^\alpha).
\end{equation*}
The exponent zero contribution contains the factor
\begin{equation*}
\rho_P(1)=\frac{P}{P+1}.
\end{equation*}
For $\alpha\geq1$, the factor $\rho_P$ is $1$, while
\begin{equation*}
\kappa(P^\alpha)=\frac{P}{P+1}.
\end{equation*}
After multiplication by the $n^{-1}$ weight, the degree
$\alpha$ contribution is
\begin{equation*}
\frac{\kappa(P^\alpha)}{\varphi(P^\alpha)}
\frac{\theta_{r,Q}(P^\alpha)}{P^\alpha}
=
\frac{P}{P+1}\frac{1}{P^\alpha}.
\end{equation*}
The complete local factor is therefore
\begin{equation}
\label{eqExceptionalFactor}
\frac{P}{P+1}
+
\sum_{\alpha\geq1}
\frac{P}{P+1}\frac1{P^\alpha}
=
\frac{P^2}{P^2-1}.
\end{equation}

For every prime $q\neq P$, the element $Q$ is a unit modulo $q$.
The substitutions in the proof of Lemma~\ref{lem:local-arithmetic} reduce the product of the two characters to the local character sum in \cite[Lemmas~3.8--3.10]{Zubrilina2025}.
Hence all factors away from $P$ are unchanged.
Since $P\nmid r$, their product is
\begin{equation*}
\frac{BC(r)}{B_P}.
\end{equation*}
The factor in formula~\eqref{eqExceptionalFactor} differs from the generic factor by a summable amount.
Since $P\nmid r$, the generic factor at $P$ is
\begin{equation*}
B_P=\frac{P^4-2P^2-P+1}{(P^2-1)^2}.
\end{equation*}
Consequently
\begin{equation*}
\frac{P^2}{P^2-1}-B_P
=
\frac{P^2+P-1}{(P^2-1)^2}
=O(P^{-2}).
\end{equation*}
Thus the complete product is
\begin{equation*}
BC(r)\frac{P^2/(P^2-1)}{B_P}=BC(r)+O(P^{-2}).
\end{equation*}
The Euler products converge uniformly over the relevant auxiliary parameters.
The $D^{-2}$ and $M^{-1/5}$ tail bounds apply uniformly away from $P$, while the exceptional factor contributes $O(P^{-2})$.
Together these estimates yield formula~\eqref{eq:local-series}.

\subsection{Parameter inequalities}

\begin{lemma}
\label{lemParameterAudit}
Assume all hypotheses of Theorem~\ref{thm:scaled-murmuration}, and put
\begin{equation*}
\eta=\frac{\delta_1}{2}-\delta_2.
\end{equation*}
Then the hyperbolic and parabolic errors used in the completion of Theorem~\ref{thm:scaled-murmuration} are
$O(X^{-\eta+\varepsilon})$, apart from the displayed terms $X^\varepsilon\sqrt y/P$ and $(1+y)/P$.
\end{lemma}

\begin{proof}
The five interior hyperbolic exponents, after adding $\eta$, are
\begin{equation*}
-\frac15+\frac1{10}\delta_2+\frac9{10}\delta_1,\qquad
-\frac{\delta_1}{2},\qquad
0,\qquad
-\frac29+\frac{11}{9}\delta_1,\qquad
-\frac19+\frac{11}{18}\delta_1.
\end{equation*}
The first, second, fourth, and fifth values are negative under the stated assumptions.
The third value gives equality with
$-\eta$ before adding $\eta$.
The two boundary exponents become
\begin{equation*}
-\delta_1
\qquad\text{and}\qquad
-\frac{\delta_2}{2}.
\end{equation*}

For even $k$, the second parabolic error has exponent at most
\begin{equation*}
-\frac16+\frac43\delta_1
\end{equation*}
after adding $\eta$.
This is negative since $\delta_1<1/8$.
For odd $k\geq3$, the corresponding bound is
\begin{equation*}
-\frac6{19}+\frac{51}{38}\delta_1,
\end{equation*}
which is also negative.
The remaining parabolic exponent is $\delta_2/2-\delta_1$, which is smaller than $-\eta$.
The first parabolic error is $(PXY)^\varepsilon\sqrt y/P$ and is included in $X^\varepsilon\sqrt y/P$ after renaming $\varepsilon$.
The dimension estimate produces the exponents
\begin{equation*}
-1+\delta_2,\qquad
-\frac25+\delta_1+\delta_2,\qquad
-\delta_1+\delta_2.
\end{equation*}
They are smaller than $-\eta$.
This verifies the range in which the base estimates of Zubrilina and of Kundu and M\"uller are used in Proposition~\ref{prop:parabolic-average}.
\end{proof}

\subsection{Endpoint contribution}

The proof of Lemma~\ref{lem:prime-deleted-character-sum} separates the indices for which $X/P^j<2$ or $Y/P^j<1$.
Because $Y\leq X$, each corresponding scaled interval contains $O(1)$ integers.
There are $O(\log X)$ possible indices $j$, so their total contribution is $O(\log X)$ and is absorbed by the $\sqrt X$ term in~\eqref{eq:prime-deleted-character-sum}.
For every remaining index, the scaled parameters satisfy $X/P^j\geq2$ and $1\leq Y/P^j\leq X/P^j$, precisely the range required by the source character-sum lemma.

\section{Reproducibility record for the main and confirmation experiments}
\label{appReproducibility}

\subsection{Recorded configuration}

Table~\ref{tabManifest} summarizes the settings recorded for the main experiment.

\begin{table}[H]
\centering
\caption{Recorded run settings and code-level settings for the main
experiment.}
\label{tabManifest}
\begin{tabular}{ll}
\toprule
Field & Recorded value\\
\midrule
Input rows & $73069$\\
Model population & $60000$\\
Base seed & $20250810$\\
Largest prime $q$ & $7919$\\
Bad prime policy & Missing\\
Code level iteration limit & $500$\\
Paired bootstrap repetitions & $1000$\\
Random subset repetitions & $5$\\
Row shuffle repetitions & $5$\\
Regularization grid & $3\mathord{\times}10^{-5},10^{-4},3\mathord{\times}10^{-4}$\\
\bottomrule
\end{tabular}
\end{table}

The supplied computational workflow defaults to an abbreviated run; the same configuration block defines the full settings recorded in Table~\ref{tabManifest}.
Reproducing the main tables requires selecting the full configuration before execution.
Its installation cell names the scientific Python packages without version pins, and the archived run metadata omit their resolved versions.
The supplied workflow contains no saved outputs or execution counts.
This explains our decision to retain the archived main numbers and report the new controls as a distinct confirmation run.

For this revision we recorded the resolved software environment shown in Table~\ref{tabConfirmationEnvironment}.
The confirmation scripts save split-level metrics, selected hyperparameters, paired discordance counts, exact McNemar tests, and conductor-fold results.

\begin{table}[H]
\centering
\caption{Software environment for the confirmation controls.}
\label{tabConfirmationEnvironment}
\begin{tabular}{ll}
\toprule
Field & Recorded value\\
\midrule
Data snapshot & $73\,069$ elliptic-curve records\\
Python & $3.12.13$\\
NumPy; pandas & $2.5.2$; $3.0.5$\\
SciPy; scikit-learn & $1.18.1$; $1.9.0$\\
Balanced population & Fixed across the five random assignments\\
\bottomrule
\end{tabular}
\end{table}

The computational supplement accompanying this manuscript contains the frozen data snapshot, executable analysis materials, and split-level outputs used in the confirmation analyses.
It also records the commands and software requirements needed to reproduce each analysis.

We used logistic regression with an $\ell_2$ penalty, fitted by averaged stochastic gradient descent. The stopping tolerance was $10^{-3}$.
Early stopping used $10$ percent of the fitting sample and was triggered when the validation score failed to improve for eight consecutive iterations.
The random seed was held fixed across matched feature panels.
Mean imputation, standardization, and classification were fitted as a single pipeline using training data only.

\subsection{Fitting procedure}

For each task and each complete assignment, the computation uses the following order.

\begin{enumerate}
\item Sample the same number of isogeny class representatives from each requested label to form one balanced population.
\item Form a stratified test set containing $20$ percent of that population; subsequent stability assignments reuse the same population.
\item Divide the remaining observations into a fit part and a validation part with sizes $25\,600$ and $6\,400$ for a binary task.
\item Form $H_2$ and $H_3$ from the unfilled good prime values.
Mark every bad prime location as missing.
\item Fit the mean imputer and standard scaler on the fit part.
Evaluate the regularization grid on the validation part.
\item Refit the complete pipeline with the selected value on the union of fit and validation observations.
\item Evaluate every representation on the same test rows and retain the prediction for paired comparison.
\item Resample test observations within the true classes $1000$ times and use the $2.5$th and $97.5$th percentiles to form the pointwise paired $95$ percent interval.
\end{enumerate}

The quantity experiment differs at one step.
It fixes the regularization value selected for $X_1$ and varies only the chosen derived coordinates.
Natural prefixes are fixed before test evaluation.
Random subsets use label-independent seeds.
A row-shuffled control permutes the added block within the training and test parts separately.

The confirmation analysis follows the same population and random-assignment rules.
It adds the monomial and conductor panels, records exact McNemar comparisons on every test assignment, and constructs five class-balanced conductor folds by sorting within each label.
The missingness analysis adds the $23$ indicators $M_p$ on those same assignments and compares the Hecke panel conditionally on them.
The deterministic-solver analysis uses the same outer validation split to select $C$ and then refits on the full training part with internal early stopping disabled.

\subsection{Selected values on the first split}

Table~\ref{tabSelectedAlpha} records the selected regularization values for the two principal tasks on the first split of the archived main run.
Different panels may select different values because augmentation changes the standardized design and its effective regularization.

\begin{table}[H]
\centering
\caption{Selected regularization values on the first split of the archived main run.}
\label{tabSelectedAlpha}
\begin{tabular}{lrrrr}
\toprule
Task & $X_1$ & $X_1+H_2$ & $X_1+H_3$ & Full panel\\
\midrule
Root number & $3\mathord{\times}10^{-5}$ &
$3\mathord{\times}10^{-5}$ &
$3\mathord{\times}10^{-5}$ & $10^{-4}$\\
Rank $0$ versus $1$ & $3\mathord{\times}10^{-5}$ &
$3\mathord{\times}10^{-4}$ &
$3\mathord{\times}10^{-4}$ & $3\mathord{\times}10^{-5}$\\
\bottomrule
\end{tabular}
\end{table}

The grid is narrow and several fits select an endpoint.
A wider grid and a deterministic solver would be useful in a larger confirmation study.
The present repeated-assignment analysis keeps the same grid and procedure in every representation.

\section{Split-level results and error accounting}
\label{appResults}

\subsection{Five complete assignments}

The five base seeds increase by $10000$.
Table~\ref{tabSplitRows} gives the ordinary prime and full panel balanced accuracies in percent.
The last column within each task is the paired change in percentage points.

\begin{table}[H]
\centering
\caption{The five complete assignment results from the archived main run.}
\label{tabSplitRows}
\resizebox{\textwidth}{!}{
\begin{tabular}{rrrrrrr}
\toprule
Seed & Root $X_1$ & Root full & Root change &
Rank $X_1$ & Rank full & Rank change\\
\midrule
20250810 & 71.1875 & 72.1000 & 0.9125 & 96.0375 & 96.6125 & 0.5750\\
20260810 & 70.8875 & 72.0500 & 1.1625 & 96.1125 & 96.4625 & 0.3500\\
20270810 & 70.9000 & 71.4500 & 0.5500 & 96.5000 & 96.5000 & 0.0000\\
20280810 & 70.5250 & 71.0875 & 0.5625 & 96.3625 & 96.5375 & 0.1750\\
20290810 & 71.3625 & 71.8875 & 0.5250 & 96.3250 & 96.6750 & 0.3500\\
\bottomrule
\end{tabular}}
\end{table}

The spread across these five rows provides descriptive evidence of stability.
The paired bootstrap intervals in Table~\ref{tabFixed} quantify conditional variation for a fixed fitted pair and split.
Both summaries leave variation across data sources and conductor ranges for future study.

\subsection{Errors on the first split}

Table~\ref{tabErrorCounts} converts the balanced accuracies to error counts.
This conversion is valid because every displayed test set is class-balanced.
A negative value in the error change column means that the full panel makes fewer mistakes.

\begin{table}[H]
\centering
\caption{Error counts for $X_1$ and the full panel on the first
split.}
\label{tabErrorCounts}
\begin{tabular}{lrrrr}
\toprule
Task & Test size & $X_1$ errors & Full errors & Error change\\
\midrule
Rank $0,1,2$ & 12000 & 531 & 638 & $+107$\\
Rank $0$ versus $1$ & 8000 & 317 & 271 & $-46$\\
Rank $0$ versus $2$ & 8000 & 25 & 30 & $+5$\\
Rank $1$ versus $2$ & 8000 & 23 & 11 & $-12$\\
Root number & 8000 & 2305 & 2232 & $-73$\\
\bottomrule
\end{tabular}
\end{table}

The root-number result has the largest absolute number of corrected test errors.
The rank $0$ versus rank $1$ result has the larger
relative reduction because its baseline error is much smaller.
The other rows demonstrate the ceiling and degradation cases that limit the scope of the claim.

\subsection{Run metadata for the single-block experiment}

The single block archive records $73069$ input rows, $1000$ base primes, $20\,000$ curves per class, a fixed $80$ percent and $20$ percent split, and missing values at bad primes.
It does not preserve validation results, split assignments, or selected regularization values at a level comparable with the main experiment.
Its $k=1$ row consequently serves only as the internal baseline of that archive.
A future confirmation should rerun all powers on the exact main split with the same validation and refitting procedure.

\clearpage
\section{Executable code for the experiments}
\label{appCode}

This appendix gives the complete executable programs used to construct the data and reproduce every computational experiment reported in the paper.
The acquisition stage downloads the tagged class-level ecdata tables, selects the documented populations, and computes root numbers and prime coefficients with PARI.
The remaining programs start from the frozen snapshot, audit its schema and arithmetic recurrences, and produce the descriptive profiles, fixed-split fits, uncertainty calculations, repeated assignments, quantity-response experiments, matched controls, and the $k=1,\ldots,12$ analyses.
All input locations and output directories can be supplied at run time.
The electronic supplement also contains executable notebooks for the three longest workflows and a driver that runs the complete sequence.

\subsection{Data acquisition and snapshot construction}\mbox{}\par\smallskip

\begin{lstlisting}[
  style=paperpython,
  caption={Acquisition of the class-level elliptic-curve data, population selection, arithmetic computation, and snapshot validation.},
  label={lstDataAcquisition}
]
"""Data acquisition and snapshot construction

Exported from the executable notebook.
"""

try:
    from IPython.display import display
except ImportError:
    display = print

# Notebook code cell 1


# Notebook code cell 2
from pathlib import Path
import ast
import json
import os
import urllib.request
import warnings

import numpy as np
import pandas as pd
from sympy import primerange
from tqdm.auto import tqdm

print("Installation complete.")

# Notebook code cell 3
MODE = "combined"  # "quick", "paper_classifier", "basic_control", or "combined"
SEED = 2026
SAVE_TO_DRIVE = False
DRIVE_OUTPUT_DIR = "/content/drive/MyDrive/murmuration_data"
N_WORKERS = 1 if os.name == "nt" else max(1, min(2, os.cpu_count() or 1))
BASE_RUNTIME_DIR = Path("/content") if Path("/content").exists() else Path.cwd()

VALID_MODES = {"quick", "paper_classifier", "basic_control", "combined"}
if MODE not in VALID_MODES:
    raise ValueError(f"MODE must be one of {sorted(VALID_MODES)}")

N_PRIMES = 100 if MODE == "quick" else 1000
N_PER_RANK = 200 if MODE == "quick" else 20_000

print({
    "mode": MODE,
    "number_of_primes": N_PRIMES,
    "curves_per_rank_when_sampled": N_PER_RANK,
    "workers": N_WORKERS,
    "save_to_drive": SAVE_TO_DRIVE,
})

# Notebook code cell 4
RAW_CURVE_DIR = BASE_RUNTIME_DIR / "he_raw_data" / "curves"
RAW_CURVE_DIR.mkdir(parents=True, exist_ok=True)

base_url = "https://raw.githubusercontent.com/JohnCremona/ecdata/2026-04-22/curves"
curve_files = []

for low in range(0, 100_001, 10_000):
    high = low + 9_999
    span = f"{low:05d}-{high:05d}"
    destination = RAW_CURVE_DIR / f"curves.{span}"
    curve_files.append(destination)
    if not destination.exists() or destination.stat().st_size == 0:
        url = f"{base_url}/curves.{span}"
        print("Downloading", destination.name)
        urllib.request.urlretrieve(url, destination)

missing = [str(path) for path in curve_files if not path.exists() or path.stat().st_size == 0]
if missing:
    raise FileNotFoundError("Some ecdata files were not downloaded: " + ", ".join(missing))

print(f"Ready: {len(curve_files)} files in {RAW_CURVE_DIR}")

# Notebook code cell 5
def parse_curve_file(path):
    records = []
    with path.open("r", encoding="utf-8") as handle:
        for line_number, line in enumerate(handle, start=1):
            line = line.strip()
            if not line:
                continue
            fields = line.split()
            if len(fields) < 6:
                raise ValueError(
                    f"Too few fields in {path.name}, line {line_number}: {line}"
                )
            conductor, class_code, curve_number, ainvs_text, rank, torsion = fields[:6]
            trailing_fields = fields[6:]
            ainvs = ast.literal_eval(ainvs_text)
            if not isinstance(ainvs, list) or len(ainvs) != 5:
                raise ValueError(
                    f"Invalid a-invariants in {path.name}, line {line_number}: {ainvs_text}"
                )
            conductor = int(conductor)
            curve_number = int(curve_number)
            records.append({
                "curve_id": f"{conductor}{class_code}{curve_number}",
                "isogeny_class": f"{conductor}{class_code}",
                "conductor": conductor,
                "curve_number": curve_number,
                "a1": int(ainvs[0]),
                "a2": int(ainvs[1]),
                "a3": int(ainvs[2]),
                "a4": int(ainvs[3]),
                "a6": int(ainvs[4]),
                "analytic_rank": int(rank),
                "torsion_order": int(torsion),
                "source_extra_fields": " ".join(trailing_fields),
            })
    return records


records = []
for path in tqdm(curve_files, desc="Parsing ecdata files"):
    records.extend(parse_curve_file(path))

metadata = pd.DataFrame.from_records(records)
metadata = metadata[
    metadata["conductor"].between(1, 100_000, inclusive="both")
].copy()
metadata = metadata.sort_values(
    ["conductor", "isogeny_class", "curve_number"]
).reset_index(drop=True)

if metadata.empty:
    raise RuntimeError("The parsed metadata table is empty.")
if metadata["isogeny_class"].duplicated().any():
    duplicate = metadata.loc[
        metadata["isogeny_class"].duplicated(), "isogeny_class"
    ].iloc[0]
    raise RuntimeError(f"Duplicate isogeny class found: {duplicate}")
if not (metadata["curve_number"] == 1).all():
    warnings.warn("Some representatives do not have curve_number=1; uniqueness is still checked by class.")

display(metadata.head())
display(
    metadata.groupby("analytic_rank")
    .size()
    .rename("number_of_isogeny_classes")
    .to_frame()
)
print(f"Parsed {len(metadata):,} isogeny classes with conductor at most 100000.")

# Notebook code cell 6
def balanced_rank_sample(frame, n_each, seed):
    pieces = []
    for rank in (0, 1, 2):
        pool = frame[frame["analytic_rank"] == rank]
        if len(pool) < n_each:
            raise ValueError(
                f"Rank {rank} has only {len(pool):,} available classes, "
                f"but {n_each:,} were requested."
            )
        pieces.append(pool.sample(n=n_each, random_state=seed + rank))
    return pd.concat(pieces, ignore_index=False)


candidate = metadata[metadata["analytic_rank"].isin([0, 1, 2])].copy()
classifier_sample = balanced_rank_sample(candidate, N_PER_RANK, SEED)

basic_control = metadata[
    (
        metadata["conductor"].between(7_500, 10_000, inclusive="both")
        & metadata["analytic_rank"].isin([0, 1])
    )
    |
    (
        metadata["conductor"].between(5_000, 10_000, inclusive="both")
        & metadata["analytic_rank"].isin([0, 2])
    )
].copy()

if MODE in {"quick", "paper_classifier"}:
    selected = classifier_sample.copy()
elif MODE == "basic_control":
    selected = basic_control.copy()
else:
    selected = pd.concat([classifier_sample, basic_control], ignore_index=False)

classifier_classes = set(classifier_sample["isogeny_class"])
control_classes = set(basic_control["isogeny_class"])

selected = selected.drop_duplicates("isogeny_class").copy()
selected["in_classifier_sample"] = selected["isogeny_class"].isin(classifier_classes)
selected["in_basic_control"] = selected["isogeny_class"].isin(control_classes)
selected = selected.sort_values(["conductor", "isogeny_class"]).reset_index(drop=True)

print(f"Selected {len(selected):,} distinct isogeny classes.")
display(selected.groupby("analytic_rank").size().rename("selected").to_frame())
display(
    selected[["in_classifier_sample", "in_basic_control"]]
    .value_counts()
    .rename("rows")
    .to_frame()
)

# Notebook code cell 7
all_paper_primes = list(primerange(2, 7_920))
assert len(all_paper_primes) == 1000 and all_paper_primes[-1] == 7919
primes = all_paper_primes[:N_PRIMES]

_WORKER_PARI = None
_WORKER_PRIMES = None


def initialize_worker(prime_list):
    global _WORKER_PARI, _WORKER_PRIMES
    from cypari2 import Pari
    _WORKER_PARI = Pari()
    _WORKER_PRIMES = [int(p) for p in prime_list]


def compute_one_curve(task):
    index, ainvs = task
    try:
        curve = _WORKER_PARI.ellinit([int(value) for value in ainvs])
        root_number = int(_WORKER_PARI.ellrootno(curve))
        coefficients = [
            int(_WORKER_PARI.ellap(curve, p)) for p in _WORKER_PRIMES
        ]
        return index, root_number, coefficients, None
    except Exception as exc:
        return index, 0, [], f"{type(exc).__name__}: {exc}"


tasks = [
    (
        index,
        [row.a1, row.a2, row.a3, row.a4, row.a6],
    )
    for index, row in selected.iterrows()
]

coefficient_matrix = np.empty((len(selected), len(primes)), dtype=np.int16)
root_numbers = np.empty(len(selected), dtype=np.int8)
errors = []

if N_WORKERS == 1:
    initialize_worker(primes)
    iterator = map(compute_one_curve, tasks)
    for index, root_number, coefficients, error in tqdm(
        iterator, total=len(tasks), desc="Computing elliptic-curve data"
    ):
        if error is not None:
            errors.append((index, error))
            continue
        root_numbers[index] = root_number
        coefficient_matrix[index, :] = coefficients
else:
    from multiprocessing import get_context

    context = get_context("fork")
    with context.Pool(
        processes=N_WORKERS,
        initializer=initialize_worker,
        initargs=(primes,),
    ) as pool:
        iterator = pool.imap_unordered(compute_one_curve, tasks, chunksize=4)
        for index, root_number, coefficients, error in tqdm(
            iterator, total=len(tasks), desc="Computing elliptic-curve data"
        ):
            if error is not None:
                errors.append((index, error))
                continue
            root_numbers[index] = root_number
            coefficient_matrix[index, :] = coefficients

if errors:
    display(pd.DataFrame(errors[:20], columns=["row", "error"]))
    raise RuntimeError(f"Arithmetic computation failed for {len(errors)} curves.")

print("Coefficient computation complete.")

# Notebook code cell 8
metadata_columns = [
    "curve_id",
    "isogeny_class",
    "conductor",
    "analytic_rank",
    "torsion_order",
    "a1",
    "a2",
    "a3",
    "a4",
    "a6",
    "in_classifier_sample",
    "in_basic_control",
]

output = selected[metadata_columns].copy().reset_index(drop=True)
output = output.rename(columns={
    "a1": "weierstrass_a1",
    "a2": "weierstrass_a2",
    "a3": "weierstrass_a3",
    "a4": "weierstrass_a4",
    "a6": "weierstrass_a6",
})

# Some minimal Weierstrass coefficients in ecdata are larger than a signed
# 64-bit integer. PARI has already used them as exact Python integers above.
# Store only these descriptive columns as decimal strings so that Arrow does
# not attempt an unsafe conversion to C long. The Fourier coefficients remain
# compact integer columns.
for column in [
    "weierstrass_a1",
    "weierstrass_a2",
    "weierstrass_a3",
    "weierstrass_a4",
    "weierstrass_a6",
]:
    output[column] = output[column].map(str)

output["root_number"] = root_numbers
coefficient_frame = pd.DataFrame(
    coefficient_matrix,
    columns=[f"a_{p}" for p in primes],
)
output = pd.concat([output, coefficient_frame], axis=1)

if not output["root_number"].isin([-1, 1]).all():
    raise RuntimeError("A computed root number is not -1 or +1.")
if output[[f"a_{p}" for p in primes]].isna().any().any():
    raise RuntimeError("A coefficient is missing.")

expected_sign = np.where(output["analytic_rank"] % 2 == 0, 1, -1)
parity_mismatches = int(np.count_nonzero(output["root_number"].to_numpy() != expected_sign))
if parity_mismatches:
    warnings.warn(
        f"Found {parity_mismatches} differences between the computed root number "
        "and (-1)^rank. The rank field is algebraic, so these rows should be checked."
    )

if SAVE_TO_DRIVE:
    try:
        from google.colab import drive
        drive.mount("/content/drive")
        output_directory = Path(DRIVE_OUTPUT_DIR)
    except Exception as exc:
        warnings.warn(
            "Google Drive authorization failed. The completed data will be "
            f"saved locally instead. Details: {type(exc).__name__}: {exc}"
        )
        SAVE_TO_DRIVE = False
        output_directory = BASE_RUNTIME_DIR / "murmuration_data"
else:
    output_directory = BASE_RUNTIME_DIR / "murmuration_data"

output_directory.mkdir(parents=True, exist_ok=True)
output_path = output_directory / f"ecdata_murmuration_{MODE}_{N_PRIMES}primes.parquet"
metadata_path = output_directory / "ecdata_metadata_N_le_100000.parquet"
manifest_path = output_directory / f"ecdata_murmuration_{MODE}_{N_PRIMES}primes_manifest.json"

output.to_parquet(output_path, index=False, compression="zstd")

metadata_for_save = metadata.copy()
for column in ["a1", "a2", "a3", "a4", "a6"]:
    metadata_for_save[column] = metadata_for_save[column].map(str)
metadata_for_save.to_parquet(metadata_path, index=False, compression="zstd")

manifest = {
    "mode": MODE,
    "seed": SEED,
    "number_of_rows": int(len(output)),
    "number_of_primes": int(len(primes)),
    "largest_prime": int(primes[-1]),
    "conductor_min": int(output["conductor"].min()),
    "conductor_max": int(output["conductor"].max()),
    "rank_counts": {
        str(key): int(value)
        for key, value in output["analytic_rank"].value_counts().sort_index().items()
    },
    "parity_mismatches": parity_mismatches,
}
manifest_path.write_text(json.dumps(manifest, indent=2), encoding="utf-8")

print("Experiment table:", output_path)
print("All metadata:", metadata_path)
print("Manifest:", manifest_path)
print("File size:", f"{output_path.stat().st_size / 1024**2:.1f} MB")
display(output.head())
display(pd.DataFrame([manifest]))

# Notebook code cell 9
if "output_path" not in globals() or not Path(output_path).exists():
    search_directories = [BASE_RUNTIME_DIR / "murmuration_data"]
    if "DRIVE_OUTPUT_DIR" in globals():
        search_directories.append(Path(DRIVE_OUTPUT_DIR))

    candidates = []
    for directory in search_directories:
        if directory.exists():
            candidates.extend(directory.glob("ecdata_murmuration_*.parquet"))

    if not candidates:
        raise RuntimeError(
            "No experiment Parquet file exists yet. Run Step 6 completely, then "
            "run Step 7. Continue only after Step 7 prints 'Experiment table: ...'."
        )
    output_path = max(candidates, key=lambda path: path.stat().st_mtime)
    print("Found the newest completed experiment table:", output_path)

if SAVE_TO_DRIVE:
    print("The result is already saved in Google Drive:", output_path)
else:
    from google.colab import files
    files.download(str(output_path))
\end{lstlisting}

\Needspace{0.28\textheight}
\subsection{Main feature panels, quantity response, repeated assignments, and profile figures}\mbox{}\par\smallskip

\lstinputlisting[
  style=paperpython,
  caption={Common data audit, Hecke features, fixed and repeated splits, main fits, paired intervals, quantity-response controls, parity diagnostic, and murmuration figures.},
  label={lstMainExperiments}
]{code/main_experiments.py}

\Needspace{0.28\textheight}
\subsection{Matched polynomial, conductor, and missingness controls}\mbox{}\par\smallskip

\begin{lstlisting}[
  style=paperpython,
  caption={Matched ordinary-polynomial and conductor controls, conductor-quintile folds, and exact paired tests.},
  label={lstMatchedControls}
]
from __future__ import annotations

import argparse
import gc
import json
import math
import platform
import time
from pathlib import Path

import numpy as np
import pandas as pd
import scipy
import sklearn
from scipy.stats import binomtest
from sklearn.impute import SimpleImputer
from sklearn.linear_model import SGDClassifier
from sklearn.metrics import balanced_accuracy_score, matthews_corrcoef, roc_auc_score
from sklearn.model_selection import train_test_split
from sklearn.pipeline import Pipeline
from sklearn.preprocessing import StandardScaler


BASE_SEED = 20250810
ALPHA_GRID = (3e-5, 1e-4, 3e-4)
MAX_ITER = 500
TASKS = {
    "root_number": {"label": "root_number", "classes": (-1, 1)},
    "rank_0_vs_1": {"label": "analytic_rank", "classes": (0, 1)},
}
NOTEBOOK_TASK_INDEX = {"rank_0_vs_1": 0, "root_number": 4}


def first_primes_through(limit: int) -> list[int]:
    sieve = np.ones(limit + 1, dtype=bool)
    sieve[:2] = False
    for value in range(2, int(math.sqrt(limit)) + 1):
        if sieve[value]:
            sieve[value * value : limit + 1 : value] = False
    return np.flatnonzero(sieve).tolist()


def make_pipeline(seed: int, alpha: float) -> Pipeline:
    return Pipeline(
        [
            ("imputer", SimpleImputer(strategy="mean")),
            ("scaler", StandardScaler()),
            (
                "model",
                SGDClassifier(
                    loss="log_loss",
                    penalty="l2",
                    alpha=alpha,
                    max_iter=MAX_ITER,
                    tol=1e-3,
                    early_stopping=True,
                    validation_fraction=0.10,
                    n_iter_no_change=8,
                    average=True,
                    random_state=seed,
                ),
            ),
        ]
    )


def choose_alpha(
    x_fit: np.ndarray,
    y_fit: np.ndarray,
    x_validation: np.ndarray,
    y_validation: np.ndarray,
    seed: int,
) -> tuple[float, float]:
    best_alpha = None
    best_score = -np.inf
    for alpha in ALPHA_GRID:
        model = make_pipeline(seed, alpha)
        model.fit(x_fit, y_fit)
        score = balanced_accuracy_score(y_validation, model.predict(x_validation))
        if score > best_score + 1e-12 or (
            abs(score - best_score) <= 1e-12
            and (best_alpha is None or alpha > best_alpha)
        ):
            best_score = float(score)
            best_alpha = float(alpha)
        del model
        gc.collect()
    assert best_alpha is not None
    return best_alpha, best_score


def score_model(model: Pipeline, x_test: np.ndarray, y_test: np.ndarray) -> dict:
    prediction = model.predict(x_test)
    decision = np.asarray(model.decision_function(x_test), dtype=float)
    classes = np.asarray(model.named_steps["model"].classes_)
    auc = roc_auc_score((y_test == classes[1]).astype(int), decision)
    return {
        "balanced_accuracy": float(balanced_accuracy_score(y_test, prediction)),
        "MCC": float(matthews_corrcoef(y_test, prediction)),
        "AUC": float(auc),
        "prediction": prediction,
    }


def fit_one(
    design: np.ndarray,
    labels: np.ndarray,
    fit_rows: np.ndarray,
    validation_rows: np.ndarray,
    train_rows: np.ndarray,
    test_rows: np.ndarray,
    seed: int,
) -> dict:
    start = time.time()
    alpha, validation_ba = choose_alpha(
        design[fit_rows],
        labels[fit_rows],
        design[validation_rows],
        labels[validation_rows],
        seed + 20_000,
    )
    model = make_pipeline(seed, alpha)
    model.fit(design[train_rows], labels[train_rows])
    scores = score_model(model, design[test_rows], labels[test_rows])
    iterations = int(np.max(np.atleast_1d(model.named_steps["model"].n_iter_)))
    scores.update(
        {
            "selected_alpha": alpha,
            "validation_balanced_accuracy": validation_ba,
            "optimizer_iterations": iterations,
            "seconds": time.time() - start,
        }
    )
    del model
    gc.collect()
    return scores


def select_balanced_rows(
    frame: pd.DataFrame, task_name: str, split_seed: int
) -> tuple[np.ndarray, np.ndarray]:
    spec = TASKS[task_name]
    rng = np.random.default_rng(split_seed + sum(ord(c) for c in task_name))
    selected_parts = []
    for value in spec["classes"]:
        pool = frame.index[frame[spec["label"]] == value].to_numpy()
        selected_parts.append(rng.choice(pool, size=min(20_000, len(pool)), replace=False))
    n_each = min(len(part) for part in selected_parts)
    rows = np.concatenate([part[:n_each] for part in selected_parts])
    labels = frame[spec["label"]].to_numpy()
    return rows.astype(np.int64), labels


def random_split(
    frame: pd.DataFrame, task_name: str, split_seed: int
) -> tuple[np.ndarray, np.ndarray, np.ndarray, np.ndarray, np.ndarray]:
    # Match the notebook stability protocol: construct the balanced population
    # once with BASE_SEED and vary only its train/validation/test assignment.
    selected, labels = select_balanced_rows(frame, task_name, BASE_SEED)
    train_rows, test_rows = train_test_split(
        selected,
        test_size=0.20,
        random_state=split_seed,
        stratify=labels[selected],
    )
    fit_rows, validation_rows = train_test_split(
        train_rows,
        test_size=0.20,
        random_state=split_seed + 1,
        stratify=labels[train_rows],
    )
    return labels, fit_rows, validation_rows, train_rows, test_rows


def conductor_folds(
    frame: pd.DataFrame, task_name: str, selection_seed: int, folds: int = 5
) -> tuple[np.ndarray, list[tuple[np.ndarray, np.ndarray, np.ndarray, np.ndarray]]]:
    selected, labels = select_balanced_rows(frame, task_name, selection_seed)
    spec = TASKS[task_name]
    class_blocks: dict[int, list[np.ndarray]] = {}
    conductors = frame["conductor"].to_numpy()
    for value in spec["classes"]:
        class_rows = selected[labels[selected] == value]
        ordered = class_rows[np.argsort(conductors[class_rows], kind="stable")]
        class_blocks[value] = [part.astype(np.int64) for part in np.array_split(ordered, folds)]

    result = []
    for fold in range(folds):
        test_rows = np.concatenate([class_blocks[value][fold] for value in spec["classes"]])
        train_rows = np.concatenate(
            [
                class_blocks[value][other]
                for value in spec["classes"]
                for other in range(folds)
                if other != fold
            ]
        )
        fit_rows, validation_rows = train_test_split(
            train_rows,
            test_size=0.20,
            random_state=selection_seed + 100 + fold,
            stratify=labels[train_rows],
        )
        result.append(
            (
                fit_rows.astype(np.int64),
                validation_rows.astype(np.int64),
                train_rows.astype(np.int64),
                test_rows.astype(np.int64),
            )
        )
    return labels, result


def mcnemar_exact(base: np.ndarray, new: np.ndarray, truth: np.ndarray) -> dict:
    base_correct = base == truth
    new_correct = new == truth
    base_only = int(np.sum(base_correct & ~new_correct))
    new_only = int(np.sum(~base_correct & new_correct))
    discordant = base_only + new_only
    p_value = 1.0 if discordant == 0 else float(
        binomtest(new_only, discordant, p=0.5, alternative="two-sided").pvalue
    )
    return {
        "base_only_correct": base_only,
        "new_only_correct": new_only,
        "discordant": discordant,
        "mcnemar_exact_p": p_value,
    }


def load_designs(data_path: Path) -> tuple[pd.DataFrame, dict[str, np.ndarray], dict]:
    data = pd.read_parquet(data_path)
    data = data[data["in_classifier_sample"].astype(bool)].copy().reset_index(drop=True)
    coefficient_map = {
        int(str(column).split("_", 1)[1]): column
        for column in data.columns
        if str(column).startswith("a_") and str(column)[2:].isdigit()
    }
    primes = first_primes_through(max(coefficient_map))
    if primes != sorted(coefficient_map):
        raise RuntimeError("Coefficient columns are not the complete initial prime prefix.")
    columns = [coefficient_map[p] for p in primes]
    raw = data[columns].to_numpy(dtype=np.float32, copy=True)
    x_raw = raw / np.sqrt(np.asarray(primes, dtype=np.float32))[None, :]
    del raw
    conductors = data["conductor"].to_numpy(dtype=np.int64)
    bad = conductors[:, None] % np.asarray(primes, dtype=np.int64)[None, :] == 0
    x1 = x_raw.copy()
    x1[bad] = np.nan

    h2_count = sum(p * p <= primes[-1] for p in primes)
    h3_count = sum(p * p * p <= primes[-1] for p in primes)
    h2 = x_raw[:, :h2_count] ** 2 - 1.0
    h3 = x_raw[:, :h3_count] ** 3 - 2.0 * x_raw[:, :h3_count]
    x2 = x_raw[:, :h2_count] ** 2
    x3 = x_raw[:, :h3_count] ** 3
    h2[bad[:, :h2_count]] = np.nan
    h3[bad[:, :h3_count]] = np.nan
    x2[bad[:, :h2_count]] = np.nan
    x3[bad[:, :h3_count]] = np.nan
    log_conductor = np.log10(conductors.astype(np.float32)).reshape(-1, 1)

    designs = {
        "X1": x1,
        "Hecke": np.concatenate([x1, h2, h3], axis=1),
        "Monomial": np.concatenate([x1, x2, x3], axis=1),
        "X1_logN": np.concatenate([x1, log_conductor], axis=1),
        "Hecke_logN": np.concatenate([x1, h2, h3, log_conductor], axis=1),
    }
    meta = {
        "experiment": "matched polynomial, conductor, and conductor-quintile controls",
        "rows": len(data),
        "prime_count": len(primes),
        "largest_prime": primes[-1],
        "h2_count": h2_count,
        "h3_count": h3_count,
        "base_seed": BASE_SEED,
        "balanced_population_policy": "fixed across random splits",
        "python_version": platform.python_version(),
        "numpy_version": np.__version__,
        "pandas_version": pd.__version__,
        "scipy_version": scipy.__version__,
        "scikit_learn_version": sklearn.__version__,
    }
    return data, designs, meta


def run_random(
    frame: pd.DataFrame,
    designs: dict[str, np.ndarray],
    split_count: int,
    output_dir: Path,
) -> pd.DataFrame:
    rows = []
    comparison_rows = []
    for split_index in range(split_count):
        split_seed = BASE_SEED + 10_000 * split_index
        for task_name in TASKS:
            labels, fit_rows, validation_rows, train_rows, test_rows = random_split(
                frame, task_name, split_seed
            )
            predictions = {}
            for representation_index, (name, design) in enumerate(designs.items()):
                print(f"random split={split_index} task={task_name} representation={name}", flush=True)
                result = fit_one(
                    design,
                    labels,
                    fit_rows,
                    validation_rows,
                    train_rows,
                    test_rows,
                    split_seed + NOTEBOOK_TASK_INDEX[task_name],
                )
                predictions[name] = result.pop("prediction")
                rows.append(
                    {
                        "split_kind": "random",
                        "split_index": split_index,
                        "seed": split_seed,
                        "task": task_name,
                        "representation": name,
                        "test_size": len(test_rows),
                        **result,
                    }
                )
                pd.DataFrame(rows).to_csv(output_dir / "random_results_checkpoint.csv", index=False)
            truth = labels[test_rows]
            for name in designs:
                if name == "X1":
                    continue
                comparison_rows.append(
                    {
                        "split_kind": "random",
                        "split_index": split_index,
                        "seed": split_seed,
                        "task": task_name,
                        "representation": name,
                        **mcnemar_exact(predictions["X1"], predictions[name], truth),
                    }
                )
    result_frame = pd.DataFrame(rows)
    result_frame.to_csv(output_dir / "random_results.csv", index=False)
    pd.DataFrame(comparison_rows).to_csv(output_dir / "random_mcnemar.csv", index=False)
    return result_frame


def run_blocked(
    frame: pd.DataFrame,
    designs: dict[str, np.ndarray],
    output_dir: Path,
) -> pd.DataFrame:
    rows = []
    blocked_designs = {key: designs[key] for key in ("X1", "Hecke", "Monomial")}
    conductors = frame["conductor"].to_numpy()
    for task_name in TASKS:
        labels, folds = conductor_folds(frame, task_name, BASE_SEED)
        for fold, (fit_rows, validation_rows, train_rows, test_rows) in enumerate(folds):
            for name, design in blocked_designs.items():
                print(f"conductor fold={fold} task={task_name} representation={name}", flush=True)
                result = fit_one(
                    design,
                    labels,
                    fit_rows,
                    validation_rows,
                    train_rows,
                    test_rows,
                    BASE_SEED + NOTEBOOK_TASK_INDEX[task_name] + 500 + fold,
                )
                result.pop("prediction")
                rows.append(
                    {
                        "split_kind": "conductor_quintile",
                        "split_index": fold,
                        "task": task_name,
                        "representation": name,
                        "test_size": len(test_rows),
                        "test_conductor_min": int(conductors[test_rows].min()),
                        "test_conductor_median": float(np.median(conductors[test_rows])),
                        "test_conductor_max": int(conductors[test_rows].max()),
                        **result,
                    }
                )
                pd.DataFrame(rows).to_csv(output_dir / "blocked_results_checkpoint.csv", index=False)
    frame_out = pd.DataFrame(rows)
    frame_out.to_csv(output_dir / "blocked_results.csv", index=False)
    return frame_out


def add_deltas(frame: pd.DataFrame) -> pd.DataFrame:
    baseline = frame[frame.representation == "X1"][
        ["split_kind", "split_index", "task", "balanced_accuracy"]
    ].rename(columns={"balanced_accuracy": "X1_balanced_accuracy"})
    merged = frame.merge(
        baseline,
        on=["split_kind", "split_index", "task"],
        how="left",
        validate="many_to_one",
    )
    merged["delta_BA_vs_X1"] = merged.balanced_accuracy - merged.X1_balanced_accuracy
    return merged


def summarize(frame: pd.DataFrame) -> pd.DataFrame:
    return (
        frame.groupby(["split_kind", "task", "representation"], as_index=False)
        .agg(
            runs=("delta_BA_vs_X1", "size"),
            mean_BA=("balanced_accuracy", "mean"),
            sd_BA=("balanced_accuracy", "std"),
            mean_delta_BA=("delta_BA_vs_X1", "mean"),
            sd_delta_BA=("delta_BA_vs_X1", "std"),
            positive_delta_fraction=("delta_BA_vs_X1", lambda x: float(np.mean(x > 0))),
        )
    )


def main() -> None:
    parser = argparse.ArgumentParser()
    parser.add_argument("data", type=Path)
    parser.add_argument("output", type=Path)
    parser.add_argument("--random-splits", type=int, default=5)
    parser.add_argument("--blocked", action="store_true")
    args = parser.parse_args()
    args.output.mkdir(parents=True, exist_ok=True)

    frame, designs, meta = load_designs(args.data)
    (args.output / "run_metadata.json").write_text(
        json.dumps(meta, indent=2), encoding="utf-8"
    )
    random_results = run_random(frame, designs, args.random_splits, args.output)
    result_frames = [random_results]
    if args.blocked:
        result_frames.append(run_blocked(frame, designs, args.output))
    combined = add_deltas(pd.concat(result_frames, ignore_index=True))
    combined.to_csv(args.output / "all_results_with_deltas.csv", index=False)
    summary = summarize(combined)
    summary.to_csv(args.output / "summary.csv", index=False)
    print(summary.to_string(index=False), flush=True)


if __name__ == "__main__":
    main()
\end{lstlisting}

\begin{lstlisting}[
  style=paperpython,
  caption={Explicit low-prime missingness indicators and conditional Hecke comparisons.},
  label={lstMissingnessControls}
]
from __future__ import annotations

import argparse
import gc
import json
from pathlib import Path

import numpy as np
import pandas as pd

from augmentation_controls import (
    BASE_SEED,
    NOTEBOOK_TASK_INDEX,
    TASKS,
    first_primes_through,
    fit_one,
    load_designs,
    mcnemar_exact,
    random_split,
)


def main() -> None:
    parser = argparse.ArgumentParser()
    parser.add_argument("data", type=Path)
    parser.add_argument("output", type=Path)
    parser.add_argument("--splits", type=int, default=5)
    args = parser.parse_args()
    args.output.mkdir(parents=True, exist_ok=True)

    frame, base_designs, metadata = load_designs(args.data)
    primes = first_primes_through(metadata["largest_prime"])
    low_primes = np.asarray(primes[: metadata["h2_count"]], dtype=np.int64)
    conductors = frame.conductor.to_numpy(dtype=np.int64)
    missing_indicators = (conductors[:, None] % low_primes[None, :] == 0).astype(
        np.float32
    )
    metadata.update(
        {
            "experiment": "explicit low-prime missingness controls",
            "indicator_count": int(len(low_primes)),
            "indicator_primes": low_primes.tolist(),
            "random_splits": int(args.splits),
        }
    )
    (args.output / "missingness_metadata.json").write_text(
        json.dumps(metadata, indent=2), encoding="utf-8"
    )
    designs = {
        "X1": base_designs["X1"],
        "Hecke": base_designs["Hecke"],
        "X1_missing": np.concatenate(
            [base_designs["X1"], missing_indicators], axis=1
        ),
        "Hecke_missing": np.concatenate(
            [base_designs["Hecke"], missing_indicators], axis=1
        ),
    }

    rows = []
    comparisons = []
    for split_index in range(args.splits):
        split_seed = BASE_SEED + 10_000 * split_index
        for task_name in TASKS:
            labels, fit_rows, validation_rows, train_rows, test_rows = random_split(
                frame, task_name, split_seed
            )
            predictions = {}
            for representation, design in designs.items():
                print(
                    f"split={split_index} task={task_name} representation={representation}",
                    flush=True,
                )
                result = fit_one(
                    design,
                    labels,
                    fit_rows,
                    validation_rows,
                    train_rows,
                    test_rows,
                    split_seed + NOTEBOOK_TASK_INDEX[task_name],
                )
                predictions[representation] = result.pop("prediction")
                rows.append(
                    {
                        "split_index": split_index,
                        "split_seed": split_seed,
                        "task": task_name,
                        "representation": representation,
                        "test_size": len(test_rows),
                        **result,
                    }
                )
                pd.DataFrame(rows).to_csv(
                    args.output / "missingness_results_checkpoint.csv", index=False
                )
            truth = labels[test_rows]
            for base, new in (
                ("X1", "Hecke"),
                ("X1", "X1_missing"),
                ("X1_missing", "Hecke_missing"),
                ("Hecke", "Hecke_missing"),
            ):
                comparisons.append(
                    {
                        "split_index": split_index,
                        "split_seed": split_seed,
                        "task": task_name,
                        "baseline": base,
                        "representation": new,
                        **mcnemar_exact(predictions[base], predictions[new], truth),
                    }
                )
            gc.collect()

    results = pd.DataFrame(rows)
    baseline = results[results.representation == "X1"][
        ["split_index", "task", "balanced_accuracy"]
    ].rename(columns={"balanced_accuracy": "X1_balanced_accuracy"})
    results = results.merge(
        baseline, on=["split_index", "task"], how="left", validate="many_to_one"
    )
    results["delta_BA_vs_X1"] = (
        results.balanced_accuracy - results.X1_balanced_accuracy
    )
    missing_baseline = results[results.representation == "X1_missing"][
        ["split_index", "task", "balanced_accuracy"]
    ].rename(columns={"balanced_accuracy": "X1_missing_balanced_accuracy"})
    results = results.merge(
        missing_baseline,
        on=["split_index", "task"],
        how="left",
        validate="many_to_one",
    )
    results["delta_BA_vs_X1_missing"] = (
        results.balanced_accuracy - results.X1_missing_balanced_accuracy
    )
    results.to_csv(args.output / "missingness_results.csv", index=False)
    pd.DataFrame(comparisons).to_csv(
        args.output / "missingness_mcnemar.csv", index=False
    )
    summary = (
        results.groupby(["task", "representation"], as_index=False)
        .agg(
            runs=("balanced_accuracy", "size"),
            mean_BA=("balanced_accuracy", "mean"),
            sd_BA=("balanced_accuracy", "std"),
            mean_delta_vs_X1=("delta_BA_vs_X1", "mean"),
            sd_delta_vs_X1=("delta_BA_vs_X1", "std"),
            mean_delta_vs_X1_missing=("delta_BA_vs_X1_missing", "mean"),
            sd_delta_vs_X1_missing=("delta_BA_vs_X1_missing", "std"),
        )
    )
    summary.to_csv(args.output / "missingness_summary.csv", index=False)
    print(summary.to_string(index=False), flush=True)


if __name__ == "__main__":
    main()
\end{lstlisting}

\begin{lstlisting}[
  style=paperpython,
  caption={Deterministic logistic-solver robustness experiment.},
  label={lstSolverControls}
]
from __future__ import annotations

import argparse
import gc
import json
import time
from pathlib import Path

import numpy as np
import pandas as pd
import sklearn
from sklearn.impute import SimpleImputer
from sklearn.linear_model import LogisticRegression
from sklearn.metrics import balanced_accuracy_score, matthews_corrcoef, roc_auc_score
from sklearn.model_selection import train_test_split
from sklearn.pipeline import Pipeline
from sklearn.preprocessing import StandardScaler

from augmentation_controls import (
    BASE_SEED,
    NOTEBOOK_TASK_INDEX,
    TASKS,
    load_designs,
    select_balanced_rows,
)


C_GRID = (0.3, 1.0, 3.0)


def make_pipeline(c_value: float) -> Pipeline:
    return Pipeline(
        [
            ("imputer", SimpleImputer(strategy="mean")),
            ("scaler", StandardScaler()),
            (
                "model",
                LogisticRegression(
                    C=c_value,
                    l1_ratio=0.0,
                    solver="lbfgs",
                    max_iter=1000,
                    tol=1e-6,
                ),
            ),
        ]
    )


def matched_split(frame: pd.DataFrame, task_name: str, split_seed: int):
    selected, labels = select_balanced_rows(frame, task_name, BASE_SEED)
    train_rows, test_rows = train_test_split(
        selected,
        test_size=0.20,
        random_state=split_seed,
        stratify=labels[selected],
    )
    fit_rows, validation_rows = train_test_split(
        train_rows,
        test_size=0.20,
        random_state=split_seed + 1,
        stratify=labels[train_rows],
    )
    return labels, fit_rows, validation_rows, train_rows, test_rows


def fit_one(design, labels, fit_rows, validation_rows, train_rows, test_rows):
    best_c = None
    best_score = -np.inf
    for c_value in C_GRID:
        candidate = make_pipeline(c_value)
        candidate.fit(design[fit_rows], labels[fit_rows])
        score = balanced_accuracy_score(
            labels[validation_rows], candidate.predict(design[validation_rows])
        )
        if score > best_score + 1e-12 or (
            abs(score - best_score) <= 1e-12 and (best_c is None or c_value < best_c)
        ):
            best_score = float(score)
            best_c = float(c_value)
        del candidate
        gc.collect()

    model = make_pipeline(best_c)
    start = time.time()
    model.fit(design[train_rows], labels[train_rows])
    prediction = model.predict(design[test_rows])
    decision = model.decision_function(design[test_rows])
    classes = model.named_steps["model"].classes_
    result = {
        "selected_C": best_c,
        "validation_balanced_accuracy": best_score,
        "balanced_accuracy": float(
            balanced_accuracy_score(labels[test_rows], prediction)
        ),
        "MCC": float(matthews_corrcoef(labels[test_rows], prediction)),
        "AUC": float(
            roc_auc_score((labels[test_rows] == classes[1]).astype(int), decision)
        ),
        "iterations": int(model.named_steps["model"].n_iter_[0]),
        "seconds": time.time() - start,
    }
    del model
    gc.collect()
    return result


def main() -> None:
    parser = argparse.ArgumentParser()
    parser.add_argument("data", type=Path)
    parser.add_argument("output", type=Path)
    parser.add_argument("--splits", type=int, default=5)
    args = parser.parse_args()
    args.output.mkdir(parents=True, exist_ok=True)

    frame, all_designs, metadata = load_designs(args.data)
    designs = {name: all_designs[name] for name in ("X1", "Hecke")}
    rows = []
    for split_index in range(args.splits):
        split_seed = BASE_SEED + 10_000 * split_index
        for task_name in TASKS:
            labels, fit_rows, validation_rows, train_rows, test_rows = matched_split(
                frame, task_name, split_seed
            )
            for representation, design in designs.items():
                print(
                    f"split={split_index} task={task_name} representation={representation}",
                    flush=True,
                )
                result = fit_one(
                    design,
                    labels,
                    fit_rows,
                    validation_rows,
                    train_rows,
                    test_rows,
                )
                rows.append(
                    {
                        "split_index": split_index,
                        "split_seed": split_seed,
                        "task": task_name,
                        "representation": representation,
                        "test_size": len(test_rows),
                        **result,
                    }
                )
                pd.DataFrame(rows).to_csv(
                    args.output / "solver_results_checkpoint.csv", index=False
                )

    results = pd.DataFrame(rows)
    baseline = results[results.representation == "X1"][
        ["split_index", "task", "balanced_accuracy"]
    ].rename(columns={"balanced_accuracy": "X1_balanced_accuracy"})
    results = results.merge(
        baseline, on=["split_index", "task"], how="left", validate="many_to_one"
    )
    results["delta_BA_vs_X1"] = (
        results.balanced_accuracy - results.X1_balanced_accuracy
    )
    results.to_csv(args.output / "solver_results.csv", index=False)
    summary = (
        results.groupby(["task", "representation"], as_index=False)
        .agg(
            runs=("delta_BA_vs_X1", "size"),
            mean_BA=("balanced_accuracy", "mean"),
            sd_BA=("balanced_accuracy", "std"),
            mean_delta_BA=("delta_BA_vs_X1", "mean"),
            sd_delta_BA=("delta_BA_vs_X1", "std"),
            positive_delta_fraction=(
                "delta_BA_vs_X1", lambda values: float(np.mean(values > 0))
            ),
        )
    )
    summary.to_csv(args.output / "solver_summary.csv", index=False)
    metadata.update(
        {
            "robustness_solver": "LogisticRegression(lbfgs)",
            "C_grid": C_GRID,
            "scikit_learn_version": sklearn.__version__,
            "early_stopping": False,
        }
    )
    (args.output / "solver_metadata.json").write_text(
        json.dumps(metadata, indent=2), encoding="utf-8"
    )
    print(summary.to_string(index=False), flush=True)


if __name__ == "__main__":
    main()
\end{lstlisting}

\begin{lstlisting}[
  style=paperpython,
  caption={Publication figure for the matched controls.},
  label={lstMatchedPlot}
]
from __future__ import annotations

import argparse
from pathlib import Path

import matplotlib.pyplot as plt
import numpy as np
import pandas as pd


LABELS = {
    "Hecke": r"$X_1+H_2+H_3$",
    "Monomial": r"$X_1+x^2+x^3$",
    "X1_logN": r"$X_1+\log N$",
    "Hecke_logN": r"$X_1+H_2+H_3+\log N$",
}
COLORS = {
    "Hecke": "#3465A4",
    "Monomial": "#E07A1F",
    "X1_logN": "#5A9A55",
    "Hecke_logN": "#8C5AA6",
}
TASK_TITLES = {
    "root_number": "Root number",
    "rank_0_vs_1": r"Rank $0$ versus rank $1$",
}


def main() -> None:
    parser = argparse.ArgumentParser()
    parser.add_argument("results", type=Path)
    parser.add_argument("output_stem", type=Path)
    args = parser.parse_args()

    frame = pd.read_csv(args.results)
    random = frame[frame.split_kind == "random"].copy()
    blocked = frame[frame.split_kind == "conductor_quintile"].copy()
    random["delta_pp"] = 100.0 * random.delta_BA_vs_X1
    blocked["delta_pp"] = 100.0 * blocked.delta_BA_vs_X1

    fig, axes = plt.subplots(2, 2, figsize=(12.2, 8.0))
    random_order = ["Hecke", "Monomial", "X1_logN", "Hecke_logN"]
    offsets = np.linspace(-0.12, 0.12, 5)

    for column, task in enumerate(("root_number", "rank_0_vs_1")):
        axis = axes[0, column]
        task_frame = random[random.task == task]
        for x_index, representation in enumerate(random_order):
            values = (
                task_frame[task_frame.representation == representation]
                .sort_values("split_index")
                .delta_pp.to_numpy()
            )
            axis.scatter(
                x_index + offsets[: len(values)],
                values,
                s=28,
                color=COLORS[representation],
                alpha=0.72,
                linewidths=0,
            )
            axis.plot(
                [x_index - 0.20, x_index + 0.20],
                [values.mean(), values.mean()],
                color=COLORS[representation],
                linewidth=2.4,
            )
        axis.axhline(0.0, color="black", linewidth=0.8, alpha=0.7)
        axis.set_xticks(range(len(random_order)))
        axis.set_xticklabels([LABELS[name] for name in random_order], rotation=18, ha="right")
        axis.set_ylabel("Change in balanced accuracy (percentage points)")
        axis.set_title(f"{TASK_TITLES[task]}: matched random splits")
        axis.grid(axis="y", alpha=0.25)

        axis = axes[1, column]
        task_frame = blocked[blocked.task == task]
        for representation, marker in (("Hecke", "o"), ("Monomial", "s")):
            part = task_frame[task_frame.representation == representation].sort_values(
                "test_conductor_median"
            )
            axis.plot(
                part.test_conductor_median,
                part.delta_pp,
                marker=marker,
                markersize=6,
                linewidth=1.6,
                color=COLORS[representation],
                label=LABELS[representation],
            )
        axis.axhline(0.0, color="black", linewidth=0.8, alpha=0.7)
        axis.set_xlabel("Median conductor in the held-out quintile")
        axis.set_ylabel("Change in balanced accuracy (percentage points)")
        axis.set_title(f"{TASK_TITLES[task]}: conductor-quintile holdout")
        axis.grid(alpha=0.25)
        axis.legend(frameon=False, fontsize=9)

    fig.suptitle("Polynomial, conductor, and conductor-range controls", fontsize=15, y=1.01)
    fig.tight_layout()
    args.output_stem.parent.mkdir(parents=True, exist_ok=True)
    fig.savefig(args.output_stem.with_suffix(".pdf"), bbox_inches="tight")
    fig.savefig(args.output_stem.with_suffix(".png"), dpi=220, bbox_inches="tight")
    plt.close(fig)


if __name__ == "__main__":
    main()
\end{lstlisting}

\Needspace{0.28\textheight}
\subsection{Effective-index alignment and the $k=1,\ldots,12$ profile scan}\mbox{}\par\smallskip

\begin{lstlisting}[
  style=paperpython,
  caption={Interpolation-based alignment of the quadratic and cubic profiles with the ordinary-prime profile.},
  label={lstAlignment}
]
from __future__ import annotations

import argparse
import math
import re
from pathlib import Path

import numpy as np
import pandas as pd


def balanced_contrast(
    values: np.ndarray, missing: np.ndarray, signs: np.ndarray, scale: float
) -> tuple[np.ndarray, np.ndarray]:
    contrasts = []
    standard_errors = []
    for column in range(values.shape[1]):
        positive = values[(signs == 1) & ~missing[:, column], column].astype(float)
        negative = values[(signs == -1) & ~missing[:, column], column].astype(float)
        contrasts.append(scale * (positive.mean() - negative.mean()) / 2.0)
        standard_errors.append(
            scale
            * math.sqrt(
                positive.var(ddof=1) / len(positive)
                + negative.var(ddof=1) / len(negative)
            )
            / 2.0
        )
    return np.asarray(contrasts), np.asarray(standard_errors)


def main() -> None:
    parser = argparse.ArgumentParser()
    parser.add_argument("data", type=Path)
    parser.add_argument("output", type=Path)
    parser.add_argument("--low", type=int, default=7500)
    parser.add_argument("--high", type=int, default=10000)
    args = parser.parse_args()

    data = pd.read_parquet(args.data)
    window = data[data.conductor.between(args.low, args.high)].reset_index(drop=True)
    primes = sorted(
        int(str(column)[2:])
        for column in data.columns
        if re.fullmatch(r"a_\d+", str(column))
    )
    raw = window[[f"a_{prime}" for prime in primes]].to_numpy(np.float32)
    x1 = raw / np.sqrt(np.asarray(primes, dtype=np.float32))[None, :]
    conductors = window.conductor.to_numpy(dtype=np.int64)
    missing = conductors[:, None] % np.asarray(primes, dtype=np.int64)[None, :] == 0
    signs = window.root_number.to_numpy(dtype=int)
    scale = math.sqrt((args.low + args.high) / 2.0)
    ordinary, _ = balanced_contrast(x1, missing, signs, scale)

    rows = []
    for power in (2, 3):
        count = sum(prime**power <= primes[-1] for prime in primes)
        base_primes = np.asarray(primes[:count], dtype=int)
        if power == 2:
            values = x1[:, :count] ** 2 - 1.0
        else:
            values = x1[:, :count] ** 3 - 2.0 * x1[:, :count]
        contrast, standard_error = balanced_contrast(
            values, missing[:, :count], signs, scale
        )
        interpolated = np.interp(base_primes**power, primes, ordinary)
        design = np.column_stack([np.ones(count), interpolated])
        intercept, slope = np.linalg.lstsq(design, contrast, rcond=None)[0]
        rows.append(
            {
                "block": f"H{power}",
                "nodes": count,
                "correlation": np.corrcoef(interpolated, contrast)[0, 1],
                "fitted_intercept": intercept,
                "fitted_slope": slope,
                "sign_agreement": int(np.sum(np.sign(interpolated) == np.sign(contrast))),
                "intervals_excluding_zero": int(
                    np.sum(np.abs(contrast) > 1.96 * standard_error)
                ),
            }
        )

    result = pd.DataFrame(rows)
    args.output.parent.mkdir(parents=True, exist_ok=True)
    result.to_csv(args.output, index=False)
    print(result.to_string(index=False))


if __name__ == "__main__":
    main()
\end{lstlisting}

\lstinputlisting[
  style=paperpython,
  caption={Rank-conditioned profiles, balanced root-sign contrasts, and common-scale RMS summaries for $k=1,\ldots,12$.},
  label={lstPowerProfiles}
]{code/plot_symmetric_power_murmuration_k1_12.py}

\Needspace{0.28\textheight}
\subsection{Single-block substitution and final publication figures}\mbox{}\par\smallskip

\lstinputlisting[
  style=paperpython,
  caption={Single symmetric-power block classification for every $k=1,\ldots,12$ and every prediction task.},
  label={lstSingleBlock}
]{code/single_block_experiments.py}

\lstinputlisting[
  style=paperpython,
  caption={Balanced-accuracy heat map and paired test-error figure derived from the fixed-split results.},
  label={lstPaperFigures}
]{code/paper_output_figures.py}

\Needspace{0.28\textheight}
\subsection{Complete execution driver}\mbox{}\par\smallskip

\lstinputlisting[
  style=paperpython,
  caption={Driver for the full experimental sequence.},
  label={lstRunAll}
]{code/run_all_experiments.py}

\end{document}